\documentclass[aos]{imsart}

\RequirePackage{amsthm,amsmath,amsfonts,amssymb}
\RequirePackage[square,numbers]{natbib}
\usepackage{upgreek}

\usepackage{mathtools}
\RequirePackage[colorlinks,citecolor=blue,urlcolor=blue]{hyperref}
\usepackage{xr-hyper}
\RequirePackage[utf8]{inputenc}
\RequirePackage[OT1]{fontenc}
\usepackage{bbm}

\usepackage{algorithm}
\usepackage[noend]{algpseudocode}

\RequirePackage{graphicx}
\RequirePackage[usenames,dvipsnames]{xcolor}

\definecolor{mypink1}{rgb}{0.858, 0.188, 0.478}
\definecolor{mypink2}{RGB}{219, 48, 122}
\definecolor{mypink3}{cmyk}{0, 0.7808, 0.4429, 0.1412}
\definecolor{mygray}{gray}{0.6}

\newcommand*\xbar[1]{%
  \hbox{%
    \vbox{%
      \hrule height 0.5pt 
      \kern0.4ex
      \hbox{%
        \kern-0.2em
        \ensuremath{#1}%
        \kern-0.1em
      }%
    }%
  }%
}

\makeindex             

\startlocaldefs

\newcommand{\F}{\mathscr{F}}

\newcommand{\f}{\rightarrow}
\def\1{\mathbbm{1}}

\usepackage{accents}

\makeatletter
\newcommand{\hathat}[1]{%
\begingroup%
  \let\macc@kerna\z@%
  \let\macc@kernb\z@%
  \let\macc@nucleus\@empty%
  \hat{\mathchoice%
    {\raisebox{.35ex}{\vphantom{\ensuremath{\displaystyle #1}}}}%
    {\raisebox{.35ex}{\vphantom{\ensuremath{\textstyle #1}}}}%
    {\raisebox{.20ex}{\vphantom{\ensuremath{\scriptstyle #1}}}}%
    {\raisebox{.14ex}{\vphantom{\ensuremath{\scriptscriptstyle #1}}}}%
    \smash{\hat{#1}}}%
\endgroup%
}
\makeatother

\newlength{\dhatheight}

\newcommand{\norm}[1]{\left\lVert#1\right\rVert}

\newcommand{\Data}{Y^{(n)}} 

\newcommand{\vir}[1]{``#1''}

\newcommand{\pg}[1]{\left\{#1\right\}}
\newcommand{\pq}[1]{\left[#1\right]}
\newcommand{\pt}[1]{\left(#1\right)}
\newcommand{\abs}[1]{\left\vert#1\right\vert}

\newcommand{\di}{\mathrm{d}}

\usepackage{scalerel,stackengine}
\stackMath
\newcommand\reallywidehat[1]{%
\savestack{\tmpbox}{\stretchto{%
  \scaleto{%
    \scalerel*[\widthof{\ensuremath{#1}}]{\kern-.6pt\bigwedge\kern-.6pt}%
    {\rule[-\textheight/2]{1ex}{\textheight}}
  }{\textheight}%
}{0.5ex}}%
\stackon[1pt]{#1}{\tmpbox}%
}

\newcommand{\Conv}{%
  \mathop{\scalebox{1.1}{\raisebox{-0.2ex}{$\circledast$}}
  }
}

\numberwithin{equation}{section}
\theoremstyle{plain}
\newtheorem{thm}{Theorem}[section]
\newtheorem{prop}{Proposition}[section]
\newtheorem{rmk}{Remark}[section]

\newtheorem{cor}{Corollary}[section]
\newtheorem{lem}{Lemma}[section]

\newtheorem{ass}{Assumption}[section]

\newcommand{\mysetminus}{\setminus} 

\usepackage{mathrsfs} 
\usepackage{dsfont}
\usepackage{bm}

\usepackage[sectionbib]{bibunits}
\usepackage{nameref}
\usepackage[nameinlink,noabbrev,capitalise]{cleveref}

\endlocaldefs

\numberwithin{equation}{section}

\begin{document}

\begin{frontmatter}
\title{Wasserstein convergence in Bayesian and frequentist deconvolution models}
\runtitle{Wasserstein convergence in  deconvolution models}

\begin{aug}
\author[A]{\fnms{Judith} \snm{Rousseau}\ead[label=e1]{judith.rousseau@stats.ox.ac.uk}},
\author[B]{\fnms{Catia} \snm{Scricciolo}\ead[label=e2]{catia.scricciolo@univr.it}}
\address[A]{University of Oxford, Department of Statistics, Oxford, UK and Universit\'e Paris Dauphine PSL University, France,
\printead{e1}}

\address[B]{Dipartimento di Scienze Economiche, Università di Verona, Verona, Italy
\printead{e2}}
\end{aug}

\begin{abstract}
We study the multivariate deconvolution problem of recovering the distribution of a signal 
from independent and identically distributed observations additively contaminated with random errors (noise) 
from a known distribution.
We investigate whether a Bayesian nonparametric approach for modelling the latent distribution of the signal 
can yield inferences with frequentist asymptotic validity under the $L^1$-Wasserstein metric. 
For errors with independent coordinates having ordinary smooth densities, 
we recast the multidimensional problem as a one-dimensional problem leveraging the strong equivalence between the $1$-Wasserstein 
and the max-sliced $1$-Wasserstein metrics and derive an inversion inequality relating the $L^1$-Wasserstein distance between two distributions of the signal 
to the $L^1$-distance between the corresponding mixture densities of the observations.
This smoothing inequality outperforms existing inversion inequalities and, at least in dimension one,
leads to minimax-optimal rates of contraction for the posterior measure on the distribution of the signal,
lower bounds for $1$-Wasserstein deconvolution in any dimension $d\geq1$, possibly with Sobolev regular mixing densities, being derived here.
As an application of the inversion inequality to the Bayesian framework, we consider $1$-Wasserstein deconvolution with Laplace noise 
in dimension one using a Dirichlet process mixture of normal densities as a prior measure on the mixing distribution (or distribution of the signal). We construct an adaptive approximation of the sampling density by convolving the Laplace
density with a well-chosen mixture of normal densities and show that the posterior measure concentrates around the sampling density at a nearly minimax rate, up to a log-factor, in the $L^1$-distance. The same posterior law is also shown to automatically adapt to the unknown Sobolev regularity of the mixing density, thus leading to a new Bayesian adaptive estimation procedure for mixing distributions with regular densities under the $L^1$-Wasserstein metric. We illustrate utility of the inversion inequality also in a frequentist setting by showing that an appropriate isotone approximation of the classical kernel deconvolution estimator attains the minimax rate of convergence for $1$-Wasserstein deconvolution in any dimension $d\geq 1$, when only a tail condition is required on the latent mixing density.
\end{abstract}


\end{frontmatter}

\section{Introduction} \label{sec:intro}
Multivariate deconvolution problems occur when we observe random vectors $\mathsf Y_i=(Y_{i,1},\,\ldots,\,Y_{i,d})^t$ in $\mathbb R^d$, for $d\geq 1$, 
that are contaminated signals $\mathsf X_i$ with additive random errors $\boldsymbol\varepsilon_i$ as in the model
\begin{equation}\label{model}
\mathsf Y_i = \mathsf X_i + \boldsymbol\varepsilon_i,
\end{equation}
where the sequences $(\mathsf X_i)_{i\in\mathbb N}$ and $(\boldsymbol\varepsilon_i)_{i\in\mathbb N}$ are independent,
the random vectors $\mathsf X_i=(X_{i,1},\,\ldots,\,X_{i,d})^t$ are independent and identically distributed (i.i.d.) according to an unknown probability measure 
$\mu_{0X}$ and the random vectors $\boldsymbol\varepsilon_i=(\varepsilon_{i,1},\,\ldots,\,\varepsilon_{i,d})^t$ are i.i.d.
according to a product probability measure $\otimes_{j=1}^d\mu_{\varepsilon,j}$, 
with $\mu_{\varepsilon,j}$ the distribution of the $j$th coordinate $\varepsilon_{i,j}$, for errors with independent components.
The distribution of the observations $\mathsf Y_i$ in $\mathbb R^d$ is then the convolution $(\otimes_{j=1}^d\mu_{\varepsilon,j})\ast\mu_{0X}$.
The interest is in recovering the distribution $\mu_{0X}$ of $\mathsf X_i$, when the error distribution is supposed to be known.
This situation is very common in many real-life problems in econometrics, biometrics, medical statistics, image reconstruction and signal deblurring, operations management, online matching markets, queueing, networks, data privacy protection under local differential privacy as popularized by Dwork, see, \emph{e.g.}, \cite{dwork2008}, etc.

%

In this paper, we consider  nonparametric estimation of $\mu_{0X}$ with respect to the $L^1$-Wasserstein metric. 
Estimation of $\mu_{0X}$ is an extensively studied problem.
There exists a vast literature on frequentist estimation of the density $f_{0X}$ of $\mu_{0X}$, 
with ground-breaking papers of the early 90's using density estimators based on Fourier inversion techniques, see \cite{caroll:hall:88, fan1993, diggle:hall:93}, penalized contrast estimators as in \cite{comte:etal:06} 
or kernel \cite{Delaigle2004BootstrapBS} and projection \cite{penskyvida99} estimators for adaptive density estimation.
Minimax rates have been studied in \cite{butucea:tsybakov08,bututsy,butucomte}. 
All these papers, however, consider the one-dimensional case and pointwise or $L^2,\,L^1$-metrics as loss functions for $f_{0X}$. 
Multivariate adaptive kernel density deconvolution taking into account possible anisotropy for both the signal and noise densities 
has been studied in \cite{comtelacourihp}, where minimax rates under the $L^2$-loss for $f_{0X}$ are derived that are a natural extension of those in the univariate case. 

Some results have been recently obtained on convergence rates for
estimating $\mu_{0X}$ under $L^p$-Wasserstein metrics, for $p\geq1$,  see \cite{caillerie2011, dedecker:michel:2013, dedecker2015} and \cite{gao2016}.
For probability measures $\mu,\,\nu$ on $\mathbb R^d$ having finite $p$th moments, the $L^p$-Wasserstein distance $W_p(\mu,\,\nu)$ is defined as
$$W_p(\mu,\,\nu):=\inf_{\gamma\in \Gamma(\mu,\,\nu)}\pt{\int_{\mathbb R^d\times \mathbb R^d}|\mathsf x-\mathsf y|^p\,\gamma(\di\mathsf x,\,\di\mathsf y)}^{1/p},$$
where $|\mathsf x-\mathsf y|$ is the Euclidean distance between $\mathsf x,\,\mathsf y\in\mathbb R^d$ and $\Gamma(\mu,\,\nu)$ denotes the set of all couplings or transport plans having marginal distributions $\mu$ and $\nu$.  

Wasserstein metrics have lately become popular in statistics and machine learning because of their suitability to problems with unusual geometry, as in manifold learning, see, for instance, \cite{divol} and the references therein, or in deconvolution models \cite{caillerie2011}. In particular, an important aspect of Wasserstein metrics is that they are much more sensitive to differences in the supports of $\mu$ and $\nu$ compared to metrics like the Hellinger or the total variation. As an extreme case, for instance, when $d=1$, while the total variation distance between $\delta_{(0)}$ and $\delta_{(\epsilon)}$, where $\delta_{(x)}$ is the Dirac mass at $x$, is equal to 1 even when $\epsilon $ is small, $L^p$-Wasserstein distances converge to 0 when $|\mathsf \epsilon|$ goes to $0$. More discussion on the use of Wasserstein metrics in the analysis of convergence of latent mixing measures in mixture models can be found in
\cite{nguyen2013}.
 
In this paper, we consider the $L^1$-Wasserstein metric, the weakest of all $L^p$-Wasserstein metrics since $W_1\leq W_p$ for every $p\geq 1$. Another important feature of the $1$-Wasserstein metric is the Kantorovich-Rubenstein dual formulation,  see, \emph{e.g.}, \cite{villani2009},
\begin{equation}\label{dual}
W_1(\mu,\,\nu)=\sup_{f\in \mathrm{Lip}_1(\mathbb R^d)}\int_{\mathbb R^d} f(\mathsf x) (\mu-\nu)(\di\mathsf x),
\end{equation}
where $\mathrm{Lip}_1(\mathbb R^d)$ is the set of all $1$-Lipschitz functions from $\mathbb R^d$ to $\mathbb R$, 
which allows to control smooth linear functionals of $\mu_{0X}$. Furthermore, being equal to the $L^1$-distance between cumulative distribution functions, the $L^1$-Wasserstein metric is useful to study quantile estimation, see, for instance, \cite{MR3449779}. 

State-of-the-art results on Wasserstein convergence rates for univariate deconvolution models
are given in \cite{dedecker2015}, where a minimum distance estimator of $\mu_{0X}$ is constructed that attains optimal convergence rates 
under $W_1$, when the error distribution is known and ordinary smooth of order $\beta\geq\frac{1}{2}$.
In the multivariate case, minimax estimation under Wasserstein metrics has only been studied in the case where the distribution of the errors is  
supersmooth, see  \cite{dedecker:michel:2013}. Convergence rates in the multivariate deconvolution model have been obtained by \cite{caillerie2011} under the $L^2$-Wasserstein loss, but they lead to rather slow convergence rates. Until now, the question of minimax rates under $L^p$-Wasserstein metrics in the multivariate deconvolution problem with ordinary smooth noise remains open. In this paper, we partially fill this gap by providing lower bounds (see Theorem \ref{thm:lower bound}) and proposing a kernel type deconvolution estimator which achieves the optimal rates, up to a log-factor, for any $d\geq 1$ under the $1$-Wasserstein distance, see Section \ref{sec:comteLatour}. 


While frequentist deconvolution estimators have been extensively studied, 
little is known about the theoretical properties of Bayesian nonparametric procedures, whose
analysis is quite involved because, differently from kernel methods where the estimators are explicit,
the posterior distribution of $\mu_X$ is not explicit.
A way to assess how accurately the posterior distribution recovers $\mu_{0X} $ is to study posterior contraction rates, 
\emph{i.e.}, to determine a sequence $\epsilon_n=o(1)$ such that, given a sample $\mathsf Y^{(n)}:= (\mathsf Y_1,\,\dots,\, \mathsf Y_n)$ of $n$ i.i.d. random vectors $\mathsf Y_i\in\mathbb R^d$ from the convolution model in \eqref{model},
\begin{equation*}
\Pi(\mu_X:\,d(\mu_X,\,\mu_{0X}) > \epsilon_n \mid \mathsf Y^{(n)})=o_{\mathsf P}(1),
\end{equation*}
where $\mu_{0X}$ is the true mixing distribution, $\Pi(\cdot \mid \mathsf Y^{(n)})$ is the posterior distribution 
and $d(\cdot,\,\cdot)$ is some semi-metric on probability measures. 
In their seminal papers \cite{ghosal2000,ghosal:vdv:07}, the authors propose an elegant strategy to study posterior concentration rates which has 
been successful for a wide range of models and prior distributions under certain metrics or loss functions and, although more adapted to losses for direct problems in the form $d(f_Y,\, f_{0Y})$, it 
has been applied also to inverse problems by \cite{nickl:sohl,ray2013}. 
This approach, however, does not seem to easily lead to sharp upper bounds on posterior convergence rates for $\mu_X$ in deconvolution. 
An alternative approach is to obtain posterior convergence rates for the direct problem, \emph{i.e.}, for $\|f_Y - f_{0Y}\|_1$,  and then combine them with an inversion inequality that 
translates an upper bound on  $\|f_Y - f_{0Y}\|_1$ into an upper bound on $W_1(\mu_X,\,\mu_{0X})$. 
Using such an inversion inequality, posterior contraction rates in $L^p$-Wasserstein metrics, for $p\geq 1$, have been derived by \cite{gao2016} in the univariate case with the Laplace noise, when $\mu_{0X}$ has bounded support. This result has been extended to the case of unbounded support by \cite{scricciolo:18}, but the rates obtained in both  papers are sub-optimal.
Similarly, an inversion inequality is proposed by \cite{nguyen2013} in general mixture models, which is used to obtain $L^2$-Wasserstein posterior convergence rates  for the mixing distribution. However, in the deconvolution model with ordinary smooth error densities, the obtained rates are suboptimal. Therefore, the construction of Bayesian minimax-optimal procedures for estimating $\mu_{0X}$ under Wasserstein metrics in a multivariate setting remains an open issue, with the sharpest results obtained in \cite{gao2016,scricciolo:18}.  Recently, \cite{su2020nonparametric} studied density deconvolution under $W_2$, subject to heteroscedastic errors as well as symmetry about zero and shape constraints, in particular, unimodality. They proved posterior consistency for Dirichlet location-mixture of gamma densities, but did not study convergence rates. 

In this paper, we propose a novel inversion inequality between $W_1(\mu_X,\, \mu_{X}')$ and the corresponding $L^1$-distance $\|f_Y - f_{Y}'\|_1$, which holds for any ordinary smooth noise distribution and any dimension $d\geq 1$. This inversion inequality is sharper than any of the existing ones, \emph{i.e.}, \cite{gao2016, scricciolo:18, nguyen2013}, and is more general then those of \cite{gao2016, scricciolo:18}, since the latter only exist when $d=1$. We then use this inversion inequality in two approaches to the deconvolution problem: the nonparametric Bayesian framework and the frequentist setting with a kernel type estimator. 
In the Bayesian setting, we first derive a simple, but general theorem on posterior concentration rates with respect to the $L^1$-Wasserstein distance on $\mu_X$. We then apply it to the special case of univariate deconvolution models with a Laplace noise in Section \ref{sec:geneBayes}, for which we obtain the minimax rate $n^{-1/5}$, up to a $(\log n)$-term, thus improving the rate  $n^{-1/8}$ obtained by \cite{gao2016, scricciolo:18}. Furthermore, we prove that the same prior leads to posterior rate adaptation to Sobolev or H\"older regularity of the mixing density $f_{0X}$.
We also use the inversion inequality of Theorem \ref{theo:1} to study a kernel type deconvolution estimator, similar to the estimator of \cite{comtelacourihp}, for multivariate deconvolution. We show that this estimator achieves the minimax rate, up to a log-factor, for all $d\geq 1$, since the bound that we obtain match with the lower bound of Theorem \ref{thm:lower bound}. For the sake of simplicity, we consider the Laplace noise example, but the proof extends to other ordinary smooth distributions.

Another nontrivial  contribution of the paper is the study of posterior rates of convergence for mixture densities $f_{0Y}$ in the Laplace deconvolution model.
Posterior rates of convergence for $f_{0Y}$ have been widely studied in the literature on Bayesian nonparametric mixture models
mostly for Gaussian mixtures, see, \emph{e.g.}, \cite{ghosal2001,scricciolo:12}. When the noise follows a Laplace distribution, 
\cite{gao2016,scricciolo:18} have obtained the rate $n^{-3/8}$, up to a $(\log n )$-factor, in the Hellinger or $L^1$-distance 
using a Dirichlet process mixture of normals prior on $\mu_X$. As noted by \cite{gao2016},
this corresponds to the minimax estimation rate for densities belonging to Sobolev balls of order $\frac{3}{2}$, 
where, in this case, $\mu_{0Y}$ belongs to any Sobolev class of order smaller than $\frac{3}{2}$.
Under the assumption that $\mu_{0X}$ has Lebesgue density $f_{0X}$, in Theorem \ref{thm:31} we prove that this rate can be improved to $n^{-2/5}$.
We also study the case where $f_{0X}$ is Sobolev $\alpha$-regular and obtain an adaptive rate of convergence for $f_{0Y}$ of the order $O(n^{-(\alpha+2)/(2\alpha+5)})$, up to a logarithmic factor, see Theorem \ref{thm:4}. We believe that the theory developed in Section \ref{subsec:adapt} to approximate $f_{0Y} =f_\varepsilon\ast  f_{0X}$ by $f_\varepsilon\ast f_X$, where $f_X$ is modelled as a mixture of Gaussian densities, is also of interest in itself.

\smallskip

\noindent
The main contributions of the paper can be thus summarized:
\begin{itemize}
\item we derive a novel inversion inequality relating the $L^1$-Wasserstein distance between the distributions of the signal 
to the $L^1$-distance between the corresponding mixture densities of the observations
in a $d$-dimensional deconvolution problem, for known error distributions having independent coordinates 
with ordinary smooth densities (Theorem \ref{theo:1}). This inequality
leads to the minimax-optimal rate $n^{-1/(2\beta d+1)}$ when $\beta \geq1$. Besides improving upon the rates existing in the literature, cf. \cite{caillerie2011,nguyen2013,gao2016,scricciolo:18}, the inversion inequality sheds light on the impact of the dimension $d$ on the minimax rate, see the discussion after Theorem \ref{theo:1};\\[-10pt]
\item we establish a general theorem on posterior contraction rates for latent mixing distributions with respect to the $L^1$-Wasserstein metric under model \eqref{model} (Theorem \ref{thm:22}). The theorem gives sufficient conditions that connect to those existing in the literature for deriving posterior convergence rates in the direct problem, see \cite{ghosal2000,ghosal:vdv:07}, which have been checked to hold for many prior distributions;\\[-10pt]
\item we construct a new adaptive approximation of the sampling density $f_{0Y}$ by convolving the Laplace
density with a well-chosen mixture of normal densities when $d=1$ (Lemma \ref{lem:contapprox}) and show that the posterior distribution automatically adapts to the Sobolev regularity of the mixing density $f_{0X}$, thus leading to a new Bayesian adaptive estimation procedure for mixing distributions with Sobolev regular densities under the $L^1$-Wasserstein metric (Theorem \ref{thm:5});\\[-10pt]
\item we validate our approach by establishing lower bounds on the rates of convergence with respect to the $L^1$-Wasserstein metric for multivariate deconvolution with independent, ordinary smooth error coordinates (Theorem \ref{thm:lower bound}). 
These lower bounds match with the upper bounds obtained, in dimension $d=1$, using a Dirichlet process mixture-of-Laplace-normals prior to deconvolve a mixing distribution with Sobolev regular density and, in any dimension $d\geq 1$, using a kernel type deconvolution estimator (Theorem \ref{th:comtelacour}).
We, therefore, fill a gap present in the literature because minimax rates for multivariate $W_1$-deconvolution with ordinary smooth errors having independent coordinates were previously not known and provide theoretical guarantees, in terms of optimal asymptotic performance, of the proposed Bayesian deconvolution procedures.
\end{itemize}

The paper is organized as follows. In Section \ref{sec:notation}, we describe the set-up and introduce the notation. 
In Section \ref{sec:general}, we present the inversion inequality.
In Section \ref{sec:geneBayes}, we first state a general theorem on posterior contraction rates for the signal distribution with respect to the $L^1$-Wasserstein metric and then apply it to the case where the noise has Laplace distribution and 
the mixing density is modelled as a Dirichlet process mixture of Gaussian densities. By constructing
a novel approximation of the sampling density, we also prove posterior rate adaptation to Sobolev regularity of the mixing density 
under the $L^1$-Wasserstein metric. 
In Section \ref{sec:comteLatour}, we present lower bounds for $L^1$-Wasserstein deconvolution 
in any dimension $d\geq1$ with ordinary smooth error distribution having independent coordinates and signal density 
belonging to a Sobolev class and show that they are attained by a frequentist minimum distance estimator.
Main proofs are deferred to Section \ref{sec:proofs}.
Extensions and open problems are discussed in Section \ref{sec:frmks}. 
Auxiliary results are reported in the Supplement \cite{rousseau:scricciolo:supp}.

\section{Set-up and notation}\label{sec:notation}
We observe a sample $\mathsf Y^{(n)} = (\mathsf Y_1,\,\dots,\, \mathsf Y_n)$ of $n$ i.i.d. random vectors $\mathsf Y_i$ of $\mathbb R^d$ from the multivariate convolution model $\mathsf Y_i = \mathsf X_i + \boldsymbol \varepsilon_i$ in \eqref{model},
where the random vectors $\mathsf X_i$ are i.i.d. according to an unknown probability measure 
$\mu_{0X}$. In case of errors with independent and identically distributed coordinates, 
the random vectors $\boldsymbol \varepsilon_i$ are i.i.d. according to the $d$-fold product probability measure $\mu_\varepsilon^{\otimes d}$ 
of the known distribution $\mu_\varepsilon$ having Lebesgue density $f_\varepsilon$, which is 
assumed to be ordinary smooth of order $\beta>0$, \emph{i.e.}, for constants $d_0>0$, 
its Fourier transform $\hat f_{\varepsilon}$ verifies
\begin{equation*}\label{eq:1}
d_0|t|^{-\beta}|\leq|\hat f_{\varepsilon}(t)| 
\end{equation*}
Examples of ordinary smooth densities are the gamma distribution with shape parameter $\beta>0$
and the Linnik distribution with index $\beta\in (0,\,2]$, the Laplace being a special case for $\beta=2$.
The common distribution of the $\mathsf Y_i$'s is given by $\mu_{0Y}=\mu_{\varepsilon}^{\otimes d}\ast \mu_{0X}$.

Let $\mathscr P(\mathbb R^d)$ stand for the set of probability measures on $(\mathbb R^d,\,\mathcal B(\mathbb R^d))$ and 
$\mathscr P_0(\mathbb R^d)$ for the subset of Lebesgue absolutely continuous distributions on $\mathbb R^d$. 
For $p\geq 1$, define $\mathcal P_p(\mathbb R^d)$ to be the set of probability measures on $\mathbb R^d$ having finite $p$th moments, \emph{i.e.} if 
$M_p(\mu):=\int_{\mathbb R^d}|\mathsf x|^p\mu(\di \mathsf x)<\infty$. In symbols, 
$\mathcal P_p(\mathbb R^d)=\{\mu\in\mathscr P(\mathbb R^d): \, M_p(\mu)<\infty\}$.  
For $M>0$, let $\mathcal P_p(\mathbb R^d,\,M)=\{\mu\in\mathscr P(\mathbb R^d):\,M_p(\mu)\leq M\}$ 
be the subset of $\mathcal P_p(\mathbb R^d)$ consisting of probability measures having $p$th moments uniformly bounded by $M$.
We denote by $\F$ the class of probability measures $\mu_Y=\mu_\varepsilon^{\otimes d}\ast \mu_X$, with $\mu_X\in\mathscr P(\mathbb R^d)$. 
Since $\mu_Y$ is Lebesgue absolutely continuous, we denote by $f_Y=f_\varepsilon^{\otimes d}\ast\mu_X$ its density.
For any subset $\mathscr P_1\subseteq\mathscr P(\mathbb R^d)$, let $\F(\mathscr P_1)$ stand for the set of probability measures 
$\mu_Y=\mu_\varepsilon^{\otimes d}\ast\mu_X$, with $\mu_X \in\mathscr P_1$. 

We consider a prior distribution $\Pi_n$ on $\mathscr P(\mathbb R^d)$ and denote by $\Pi_n(\cdot \mid \mathsf Y^{(n)})$ 
the corresponding posterior measure
$$\Pi_n(B\mid \mathsf Y^{(n)}) = \frac{ \int_B \prod_{i=1}^n f_Y(\mathsf Y_i) \,\di\Pi_n(\mu_X)}{\int\prod_{j=1}^n f_Y(\mathsf Y_j)\,\di\Pi_n(\mu_X)}.$$
Our aim is to assess the posterior contraction rate for $\mu_X$ in the $L^1$-Wasserstein distance, namely, to find a sequence
$\epsilon_n=o(1)$ such that, if $ \mathsf Y^{(n)}$ is an $n$-sample from model \eqref{model} with true mixing distribution $\mu_{0X}$, then,
for a sufficiently large constant $M>0$,
$$ 
\Pi_n(\mu_X:\,W_1(\mu_X, \,\mu_{0X})\leq M\epsilon_n \mid \mathsf Y^{(n)})  \rightarrow 1 \mbox{ in $P_{0Y}^n$-probability,}
$$ 
where $P_{0Y}^n$ stands for the $n$-fold product measure of $P_{0Y}\equiv\mu_{0Y}$.

We hereafter review some useful facts on Wasserstein metrics. 
It is known that $\mathcal P_p(\mathbb R^d)$ endowed with $W_p$ is a Polish space, \emph{i.e.}, a
separable and completely metrizable space, see, \emph{e.g.}, Theorem 6.18 of \cite{villani2009}. 
For $d=1$, the following explicit expression of $W_p$ holds true:
\begin{equation}\label{eq:wascdf}
W_p(\mu,\,\nu)=\pt{\int_0^1|F_\mu^{-1}(s)-F_\nu^{-1}(s)|^p\,\di s}^{1/p},
\end{equation}
where $F_\mu^{-1}(s):=\inf\{x\in\mathbb R:\,\mu((-\infty,\,x])> s\}$ and
$F_\nu^{-1}(s):=\inf\{x\in\mathbb R:\,\nu((-\infty,\,x])> s\}$ are the generalized inverse distribution functions associated 
to $\mu,\,\nu\in\mathcal P_p(\mathbb R)$. For $p=1$, 
\begin{equation}\label{eq:lcdf}
W_1(\mu,\,\nu)=\int_0^1|F_\mu^{-1}(s)-F_\nu^{-1}(s)|\,\di s=\int_{\mathbb R}|F_\mu(x)-F_\nu(x)|\,\di x=\|F_\mu-F_\nu\|_1.
\end{equation}
Since in the $\mathbb R^d$-case the closed-form expression in
\eqref{eq:wascdf} of the $L^p$-Wasserstein distance in terms of the inverse distribution functions no longer holds, 
we can exploit the connection between the $L^p$-Wasserstein distance and its max-sliced version, 
which only requires estimating the $L^p$-Wasserstein distances of the projected uni-dimensional distributions.
Let $\mathbb S^{d-1}:=\{\mathsf v\in\mathbb R^d:\,|\mathsf v|=1\}\subset \mathbb R^d$ be the unit sphere. For $\mu\in\mathcal P_p(\mathbb R^d)$ 
and $\mathsf v\in\mathbb S^{d-1}$, we set $\mu_{\mathsf v}:=\mu\circ \mathsf v_\ast ^{-1}$ to be the image measure of $\mu$ by 
$\mathsf v_\ast$, where $ \mathsf v_\ast :\,\mathbb R^d\rightarrow \mathbb R$ is the map defined by $\mathsf v_\ast(\mathsf x):=\mathsf v\cdot \mathsf x=\sum_{j=1}^dv_jx_j$. 
Then, $\mu_{\mathsf v}\in\mathcal P_p(\mathbb R)$ because 
\begin{equation}\label{eq:momentsineq}
M_p(\mu_{\mathsf v})=\int_{\mathbb R}|x|^p\mu_{\mathsf v}(\di x)=\int_{\mathbb R^d}|\mathsf v\cdot \mathsf x|^p\mu(\di \mathsf x)\leq 
\int_{\mathbb R^d}|\mathsf x|^p\mu(\di \mathsf x)=M_p(\mu)<\infty.
\end{equation}
For $\mu,\,\nu\in \mathcal P_p(\mathbb R^d)$, the max-sliced Wasserstein distance $\overline W_p(\mu,\,\nu)$
is defined as 
$$\overline W_p(\mu,\,\nu):=\sup_{\mathsf v\in \mathbb S^{d-1}}W_p(\mu_{\mathsf v},\,\nu_{\mathsf v}) = \max_{\mathsf v\in \mathbb S^{d-1}}W_1(\mu_{\mathsf v},\,\nu_{\mathsf v}).$$ 
Of particular importance for what follows is the strong equivalence between $\overline W_1$ and $W_1$ 
due to \cite{bay}, Theorem 2.1(ii), pp. 4 and 6--7, according to which $\overline W_1$ and $W_1$ are strongly equivalent for all $d\geq 1$, that is, there exists a constant $C_d\geq 1$ 
such that, for all $\mu,\,\nu\in\mathcal P_1(\mathbb R^d)$, 
\begin{equation}\label{eq:equivalence}
\overline W_1(\mu,\,\nu)\leq W_1(\mu,\,\nu)\leq C_d \overline W_1(\mu,\,\nu).
\end{equation}

We now introduce some notation that will be used throughout the article. 
For probability measures $P,\,Q\in\mathscr P_0(\mathbb R^d)$, with respective densities $f_P,\,f_Q$ relative to some reference measure, 
let $d_\mathrm{H}(f_P,\,f_Q):=\|\sqrt{f_P} - \sqrt{f_Q} \|_2$ be the Hellinger distance between $f_P$ and $f_Q$, 
where $\|f_P\|_r$ is the $L^r$-norm of $f_P$, for $r\geq1$.
Letting $Pf$ stand for the expected value $\int f\di P$, where the integral extends over the entire domain,
we define the Kullback-Leibler divergence of $Q$ from $P$ as $\mathrm{KL}(P;\,Q):=P\log(f_P/f_Q)$ and, for $\epsilon>0$, 
the $\epsilon$-Kullback-Leibler type neighbourhood of $P$ as
$$B_{\mathrm{KL}}(P;\,\epsilon^2)=
\pg{Q\in\mathscr P_0(\mathbb R^d):\, \mathrm{KL}(P;\,Q)\leq\epsilon^2,\,\,\,
P\pt{\log\frac{f_P}{f_Q}}^2\leq\epsilon^2}.$$

For $f\in L^1(\mathbb R^d)$, let $\hat f (\mathsf t):= \int_{\mathbb R^d} e^{\imath \mathsf t \cdot \mathsf x} f(\mathsf x)\,\di \mathsf x$, $\mathsf t \in \mathbb R^d$, 
be its Fourier transform. When $d=1$, for $\alpha\geq 0$ and $f\in L^1(\mathbb R)$ such that $\int_{\mathbb R}|t|^{\alpha}|\hat f(t)|\,\di t<\infty$, we define the 
$\alpha$th fractional derivative of $f$ as $D^{\alpha}\hspace*{-1pt}f(x):={(2\pi)}^{-1}\int_{\mathbb{R}}e^{-\imath t x}(-\imath t)^\alpha\hat f(t)\,\di t$, with $D^0f \equiv f$. Let $C_b(S)$ be the set of bounded, continuous real-valued functions on $S\subseteq \mathbb R^d$.

For global estimation, we consider Sobolev spaces. 
For $\boldsymbol{\alpha}=(\alpha_1,\,\ldots,\,\alpha_d)^t$, let the anisotropic Sobolev space $\mathcal S_d(\boldsymbol{\alpha},\,L)$ be defined as the class of integrable functions $f:\,\mathbb R^d\rightarrow\mathbb R$ satisfying
\[\sum_{j=1}^d\int_{\mathbb R^d}|\hat f(\mathsf t)|^2(1+t_j^2)^{\alpha_j}\,\di \mathsf{t}\leq L^2.\]
For pointwise estimation, we consider H\"{o}lder classes. Let the H\"{o}lder class $\mathcal H_d(\boldsymbol{\alpha},\,L)$ be defined as the class of functions $f:\,\mathbb R^d\rightarrow\mathbb R$ that admit derivatives with respect to $x_j$ up to the order $\lfloor \alpha_j\rfloor$ and
\[\abs{\frac{\partial ^{\lfloor \alpha_j\rfloor}f}{(\partial x_j)^{\lfloor \alpha_j\rfloor}}(x_1,\,\ldots,\,x_{j-1},\,x_j',\,x_{j+1},\,\ldots,\,x_d)-\frac{\partial ^{\lfloor \alpha_j\rfloor}f}{(\partial x_j)^{\lfloor \alpha_j\rfloor}}(\mathsf x)}\leq L|x_j'-x_j|^{\alpha_j-\lfloor \alpha_j\rfloor},\]
where $\lfloor \alpha_j \rfloor:=\max\, \{k\in\mathbb{Z} : k<\alpha_j\}$ is the lower integer part of $\alpha_j$. 
In the isotropic case, for $\alpha_1=\ldots=\alpha_d=\alpha$, we simply write $\mathcal S_d(\alpha,\,L)$ and $\mathcal H_d(\alpha,\,L)$. 
The Sobolev and H\"{o}lder spaces of dimension one are denoted by $\mathcal S(\alpha,\,L)$ and $\mathcal H(\alpha,\,L)$, respectively.

For $\epsilon>0$, let $D(\epsilon,\,B,\,d)$ be the $\epsilon$-packing number of a set $B$ with 
metric $d$, that is, the maximal number of points in $B$ such that the $d$-distance between every pair is at least $\epsilon$,
where $d$ can be either the Hellinger or the $L^1$-distance.

We denote by $\phi (x) =(2\pi)^{-1/2} e^{-x^2/2}$, $x\in\mathbb R$, the density of a standard Gaussian random variable and 
by $\phi_{\mu,\sigma}(x)=(1/\sigma)\phi((x-\mu)/\sigma)$, $x\in\mathbb R$, its recentered and rescaled version.
We write  $a\vee b=\max\{a,\,b\}$,  $a\wedge b=\min\{a,\,b\}$ and
 $a_+=a\vee 0$.
Also, $a_n\lesssim b_n$  (resp. $a_n\gtrsim b_n$) means that  $a_n\leq Cb_n$ (resp. $a_n\geq Cb_n$) for some  $C>0$ that is universal or depends only on $P_{0Y}$ 
and $a_n\asymp b_n$ means that both $a_n \lesssim b_n$ and $b_n \lesssim a_n$ hold.
Let $\mathbb N_0:=\{0,\,1,\,2,\,\ldots\}$. For any $d\in\mathbb N$, let $[d]:=\{1,\,\ldots,\,d\}$.


\section{Inversion inequality between the direct and inverse problems}\label{sec:general}
In this section, we present an inversion inequality relating the $L^1$-Wasserstein distance $W_1(\mu_X,\,\mu_{0X})$ 
between the mixing distributions to the $L^1$-distance $\|f_Y - f_{0Y}\|_1$ between the corresponding mixture densities. 
The inequality, which is stated in Section \ref{subsec:inversion}, is the key tool for proving 
a general theorem on $L^1$-Wasserstein contraction rates for the posterior distribution of the 
mixing measure based on properties of the prior law and the data generating process. 
The inequality may also be of interest in itself.

\subsection{Assumptions}\label{subsec:Ass}
In order to obtain $L^1$-Wasserstein posterior contraction rates for the latent distribution $\mu_X$, 
we make assumptions on the single coordinate error distribution $\mu_\varepsilon$ and the 
\vir{true} mixing measure $\mu_{0X}$.\\

\noindent
\emph{Error assumptions}\\[-4pt]

\noindent
If $|\hat f_{\varepsilon}(t)|\neq 0$, $t\in\mathbb R$, then the reciprocal of $\hat f_{\varepsilon}$,
\begin{equation}\label{eq:re}
r_\varepsilon(t):=\frac{1}{\hat f_\varepsilon(t)}, \quad t\in\mathbb{R},
\end{equation}
is well defined. For an $l$-times differentiable Fourier transform
$\hat f_\varepsilon$, with $l\in\mathbb N_0$, the $l$th derivative of $r_\varepsilon$ is denoted by $r^{(l)}_\varepsilon$, 
with $r_\varepsilon^{(0)}\equiv r_\varepsilon$.

\begin{ass}\label{ass:identifiability+error}
\emph{The single coordinate error distribution $\mu_\varepsilon\in\mathscr P_0(\mathbb R)\cap\mathcal P_1(\mathbb R)$ 
has Fourier transform 
$|\hat f_{\varepsilon}(t)|\neq 0$, $t\in\mathbb{R}$. 
Furthermore, there exists $\beta>0$ such that, for $l=0,\,1$,
\begin{equation}\label{eq:deriv}
|r_\varepsilon^{(l)}(t)|\lesssim (1+|t|)^{\beta-l},\quad t\in\mathbb R.
\end{equation}
}
\end{ass}

Assumption \ref{ass:identifiability+error} requires that $\hat f_\varepsilon$ is everywhere non-null.
This is a standard hypothesis in density deconvolution problems, related to the identifiability with respect to the $L^1$-metric, 
which is a necessary condition for the existence of consistent density estimators of $f_{0X}$ 
with respect to the $L^1$-metric, see \cite{Meister:2009}, pp. 23--26.
Finiteness of the first moment of $\varepsilon$, that is, $M_1(\mu_\varepsilon)<\infty$, is a technical condition with a two-fold aim.
First, if also $M_1(\mu_{0X})<\infty$, then it entails that $M_1(\mu_{0Y})<\infty$, thus allowing to define the $L^1$-Wasserstein distance between $\mu_{0Y}$ and $\mu_Y$, provided that $\mu_Y$ has finite expectation too.
Secondly, it implies that $\hat f_\varepsilon$ is continuously differentiable on $\mathbb R$ and the derivative
is $\hat f^{(1)}_\varepsilon(t)=\int_{\mathbb R}e^{\imath tu}(\imath u) f_\varepsilon(u)\,\di u$, $t\in\mathbb R$. Then, 
$r_\varepsilon^{(1)}$ exists and is well defined. 
Differently from \cite{dedecker2015}, in condition \eqref{eq:deriv}, we do not assume that 
$r_\varepsilon$ is at least twice continuously differentiable. Instead,
as in \cite{MR3449779}, we only assume the existence of the first derivative such that 
$|r_\varepsilon^{(1)}(t)|\lesssim (1+|t|)^{\beta-1}$, $t\in\,\mathbb R$.
Note that, for $l=0$, condition \eqref{eq:deriv} is equivalent to $|\hat f_\varepsilon(t)|\gtrsim (1+|t|)^{-\beta}$, $t\in\mathbb R$. 
We mention that only the lower bound on $|\hat f_\varepsilon|$ is required to derive upper bounds on the convergence rates.
Assumption \ref{ass:identifiability+error} is satisfied for ordinary smooth error densities covering the following examples.
\begin{itemize}  
 \item[$\bullet$] 
The symmetric Linnik distribution with $\hat f_\varepsilon(t)=(1+|t|^\beta)^{-1}$, $t\in\mathbb R$, 
for index $0<\beta\leq 2$ and scale parameter equal to $1$. 
The standard Laplace distribution corresponds to $\beta=2$, see $\S$ 4.3 in \cite{Kotz2001}, pp. 249--276.
\item[$\bullet$]
The gamma distribution with $\hat f_\varepsilon(t)=(1-\imath t)^{-\beta}$, $t\in\mathbb R$,
for shape parameter $\beta>0$ and scale parameter equal to $1$. 
The standard exponential distribution corresponds to $\beta=1$. 
Exponential-type densities have great interest in physical contexts, see, for instance, the fluorescence model
studied in \cite{10.1214/12-EJS737}, where the measurement error density is fitted as an exponential-type distribution.
\item[$\bullet$]
An error distribution with characteristic function $\hat f_\varepsilon$ that is the reciprocal of a polynomial,
$r_\varepsilon(t)=\sum_{j=0}^m a_j t^{s_j}$, $t\in\mathbb R$, with $a_j\in \mathbb C$, for $j=0,\,\ldots,\,m$, and exponents $0\leq s_0< s_1<\ldots<s_m=\beta$, with $\beta>0$. This extends Example 1 in 
\cite{bissantz}, p. 487,
wherein the $s_j$'s are taken to be non-negative integers $s_j=j$, for $j=0,\,\ldots,\,\beta$.
\item[$\bullet$]
The error distribution in Example 2 of \cite{bissantz}, p. 487, with $f_\varepsilon(u)=\gamma[g_0(u-\mu)+g_0(u+\mu)]/2+(1-\gamma)g_0(u)$, 
$u\in\mathbb R$, for a density $g_0$, constants $0<\gamma<{1}/{2}$ and $\mu\neq 0$,
having $\hat f_\varepsilon(t)=[(1-\gamma)+\gamma\cos(\mu t)]\hat g_0(t)$, $t\in\mathbb R$, 
with $|\hat g_0(t)|\gtrsim (1+|t|)^{-\beta}$, for $\beta>0$.
\end{itemize}  
Location and/or scale transformations of random variables with distributions as in the previous examples, 
as well as their convolutions, verify condition \eqref{eq:deriv}. In fact, if we consider the $m$-fold self-convolution of $f_\varepsilon$, then 
we obtain an ordinary smooth error density with degree $\beta m$, because the corresponding Fourier transform is equal to $(\hat f_\varepsilon)^m$. 
Nevertheless, there are important distributions, such as the \emph{uniform}, \emph{triangular} and \emph{symmetric gamma},
that cannot be classified neither as ordinary smooth nor as supersmooth. 
For nonstandard error densities, see, \emph{e.g.}, \cite{Meister:2009}, pp. 45--46, and the references therein.\\

In this paper we derive results for any $\mu_{0X}$ but also some more precise results for smooth mixing densities, i.e. under the following assumptions:

\noindent
\emph{Regularity assumptions on the mixing distribution}\\[-4pt]

We consider Sobolev or H\"older regularity for the Lebesgue density $f_{0X}$ of the mixing distribution 
$\mu_{0X}\in\mathscr P_0(\mathbb R^d)$. 

\begin{ass}\label{ass:smoothXXX}
\emph{The mixing distribution $\mu_{0X}\in\mathscr P_0(\mathbb R^d)\cap\mathcal P_1(\mathbb R^d)$ 
is such that there exists $\alpha>0$ for which
\begin{equation}\label{ass:smoothsob12}
\max_{\mathsf v\in \mathbb S^{d-1}}
\int_{\mathbb R}|t|^\alpha|\hat \mu_{0X}(t\mathsf v)|\,\di t<\infty \quad \mbox{and} \quad
\max_{\mathsf v\in \mathbb S^{d-1}}\|D^\alpha f_{0X,\mathsf v}\|_1<\infty,
\end{equation}
where $D^\alpha f_{0X,\mathsf v}$ is the inverse Fourier transform of $(-\imath \cdot)^\alpha\hat\mu_{0X}(\cdot\mathsf v)$.
}
\end{ass}

For $d=1$, the conditions in \eqref{ass:smoothsob12} reduce to $\int_{\mathbb R}|t|^\alpha|\hat \mu_{0X}(t)|\,\di t<\infty$ and  
$D^{\alpha}\hspace*{-1pt}f_{0X}\in L^1(\mathbb R)$. In dimension one,
we also consider the case where $f_{0X}$ belongs to a H\"older class.

\begin{ass}\label{ass:smoothXXX1}
\emph{The mixing distribution $\mu_{0X}\in\mathscr P_0(\mathbb R)\cap\mathcal P_1(\mathbb R)$ 
has density $f_{0X}$ verifying the following condition:
there exist $\alpha>0$ and $L_0\in L^1(\mathbb R)$ such that the derivative $f^{(\ell)}_{0X}$ of order 
$\ell=\lfloor \alpha\rfloor$ exists and
\begin{equation*} 
|f^{(\ell)}_{0X}(x+\delta)-f^{(\ell)}_{0X}(x)|\leq L_0(x)|\delta|^{\alpha-\ell} \,\mbox{ for every }\delta,\,x\in\mathbb R.
\end{equation*}
}
\end{ass}

\smallskip

Thus, when $d=1$, when we consider smoothness assumptions of $f_{0X}$, we assume that $f_{0X}$ belongs to either
a Sobolev or a H\"older class of densities, which are common nonparametric classes of regular functions.
With Assumption \ref{ass:smoothXXX1}, the density $f_{0X}$ is required to be locally H\"older smooth, namely, it has $\ell$ derivatives, 
for $\ell$ the largest integer strictly smaller than $\alpha$, 
with the $\ell$th derivative being H\"older of order $\alpha-\ell$ and integrable envelope $L_0$, 
the latter condition being used to bound the $L^1$-norm of the bias of $F_{0X}$, cf. Lemma \ref{lem:der}. 
With Assumption \ref{ass:smoothXXX}, instead,
$f_{0X}$ is required to have global Sobolev regularity $\alpha$. Requiring that $D^\alpha\hspace*{-1pt}f_{0X} \in L^2(\mathbb R)$ is equivalent to
imposing that $f_{0X}\in\mathcal S(\alpha,\,L)$ for some $L>0$, the difference being that $D^\alpha\hspace*{-1pt}f_{0X}$ is here assumed 
to be in $L^1(\mathbb R)$.

\smallskip

To prove the inversion inequalities of Theorem \ref{theo:1} below 
we use a kernel whose choice 
depends on the type of regularity of $f_{0X}$.
We consider $K\in L^1(\mathbb R)\cap L^2(\mathbb R)$, with $zK(z)\in L^1(\mathbb R)$, such that 
\begin{itemize} 
\item[(a)]under  Assumption \ref{ass:smoothXXX}, $K$ is symmetric 
with $\hat K$ supported on $[-2,\,2]$, while $\hat K\equiv 1$ on $[-1,\,1]$;\\[-7pt]
\item[(b)]under Assumption \ref{ass:smoothXXX1}, $K$ is a \emph{kernel of order $\ell$}, see, \emph{e.g.}, \cite{Meister:2009}, pp. 38-39: $\int_{\mathbb R} K(z)\,\di z=1$, while $\int_{\mathbb R}z^j K(z)\,\di z= 0$, for $j\in[\ell]$, 
with $\hat K$ supported on $[-1,\,1]$.
\end{itemize}

\smallskip
\noindent
In case (a), a key property is that, for $A:=1+\|K\|_1<\infty$,
\begin{equation*}
\sup_{t\in\mathbb R\setminus\{0\}}\frac{|1-\hat K(t)|}{|t|^\alpha}\leq A \,\mbox{ for all }\alpha>0.
\end{equation*}
For $h>0$, we define $K_h(\cdot):=(1/h)K(\cdot/h)$ as the rescaled kernel and
$b_{F_X}(h):=F_X - F_X\ast K_h$ as the \vir{bias} of the distribution function $F_X$
of a probability measure $\mu_X$ on $\mathbb R$.
In general, for $d\geq1$, we consider a multivariate kernel on $\mathbb R^d$ with independent coordinates defined as
\begin{equation}\label{eq:kmultiv}
K^{\otimes d}(\mathsf x):=\prod_{j=1}^dK(x_j), \quad \mathsf x\in\mathbb R^d.
\end{equation} 
For $\mu_X\in\mathcal P_1(\mathbb R^d)$ and $\mathsf v\in\mathbb S^{d-1}$, let
$b_{F_{X,\mathsf v}}(h):=F_{X,\mathsf v}-F_{X,\mathsf v}\ast (K^{\otimes d}_h)_{\mathsf v}$ be the bias of the distribution function 
$F_{X,\mathsf v}$ associated to $\mu_{X,\mathsf v}\in\mathcal P_1(\mathbb R)$.

\subsection{Inversion inequality}\label{subsec:inversion}
In this section, we establish, in the $d$-dimensional case and for measurement errors with independent coordinates 
having ordinary smooth densities that are known, possibly up to a scale parameter, an inversion inequality relating the $L^1$-Wasserstein distance between the mixing distributions to
the $L^1$-norm distance between the corresponding mixture densities.
This inequality plays a crucial role
in the proofs of Theorems \ref{thm:22} and \ref{th:comtelacour}. The proof  of Theorem \ref{theo:1} is reported in Section \ref{sec:rth1}. 
Starting from \cite{dedecker2015}, the idea is to use a suitable kernel
to smooth the distribution functions $F_X$, $F_{0X}$ corresponding to the mixing measures $\mu_X$, $\mu_{0X}$
and then to bound the $L^1$-Wasserstein distance between 
the smoothed versions, meanwhile controlling the bias induced by the smoothing. 

\begin{thm}\label{theo:1}
Let $\mu_X,\,\mu_{0X}\in\mathcal P_1(\mathbb R^d)$, $d\geq1$, and let the error distribution
$\mu_\varepsilon^{\otimes d}$ have single coordinate measure 
$\mu_\varepsilon\in\mathcal P_{1}(\mathbb R)$ 
satisfying Assumption \ref{ass:identifiability+error} for $\beta>0$. 
Then, for probability measures 
$\mu_Y:=\mu_\varepsilon^{\otimes d}\ast\mu_X$, $\mu_{0Y}:=\mu_\varepsilon^{\otimes d}\ast\mu_{0X}$, 
having densities $f_Y$, $f_{0Y}$, respectively, and a sufficiently small $h>0$,
\begin{equation*}
W_1(\mu_X,\,\mu_{0X})\lesssim h+  W_1(\mu_Y,\,\mu_{0Y}) +T,
\end{equation*}
with 
\begin{equation}\label{eq:t2}
\begin{split}
T &= |\log h|\max_{ \mathsf v \in \mathbb S^{d-1}}  
\bigg(  |\log h|\1_{(\beta|I_h^\ast(\mathsf v)|\leq 1)}+ h^{-\beta |I_h^\ast(\mathsf v)|+1} 
\prod_{j\in I_h^\ast(\mathsf v)} |v_j|^\beta \1_{(\beta |I_h^\ast(\mathsf v)|> 1)} \bigg)\\
&\qquad\qquad\qquad\qquad\qquad\qquad\qquad\qquad\qquad\qquad\qquad\qquad\times\|f_{Y,\mathsf v} -f_{0Y,\mathsf v}\|_1,
\end{split}
\end{equation}
where, for each $\mathsf v \in \mathbb S^{d-1}$, we let $I_h^\ast(\mathsf v):=\{j\in[d]:\,|v_j|>h\}$.

\smallskip

If, in addition, $\mu_{0X}$ satisfies  Assumption \ref{ass:smoothXXX} for $\alpha>0$ and
there exist a constant $C_1>0$ and a kernel $K$ as in (a) such that
\begin{equation}\label{eq:ass1}
\max_{\mathsf v\in\mathbb S^{d-1}}\|b_{F_{X,\mathsf v}}(h)\|_1\leq C_1 h^{\alpha+1},
\end{equation}
then 
\begin{equation*}\label{eq:ineqw21}
W_1(\mu_X,\,\mu_{0X})\lesssim h^{\alpha+1}+W_1(\mu_Y,\,\mu_{0Y})+T,
\end{equation*}
with $T$ as in \eqref{eq:t2}.
\end{thm}

\begin{rmk}
\emph{The terms $h$ and $h^{\alpha+1}$ in $W_1(\mu_X,\,\mu_{0X}) \lesssim h +W_1(\mu_Y,\,\mu_{0Y}) + T$ and $W_1(\mu_X,\, \mu_{0X}) \lesssim h^{\alpha+1}+W_1(\mu_Y,\,\mu_{0Y}) + T$, respectively, stem from bounding
$$\max_{{\mathsf v}\in\mathbb S^{d-1}}b_{F_{X,{\mathsf v}}}(h)+\max_{{\mathsf v}\in\mathbb S^{d-1}}b_{F_{0X,{\mathsf v}}}(h).$$}
\end{rmk}

\begin{rmk}\label{rkT}
\em{The key quantity in the inversion inequality is $h+T$ (or $h^{\alpha+1} + T$ when $\alpha>0$) because  typically  $W_1(\mu_Y,\, \mu_{0Y})$ can be bounded by a term of the same order as $\|f_Y- f_{0Y}\|_1$, up to a log-factor, see, {\em e.g.}, Theorem \ref{lem:1}.}
\end{rmk}

\begin{rmk}\label{rmk:ext1}
\emph{
For $d=1$, Theorem \ref{theo:1} also holds if Assumption \ref{ass:smoothXXX1}, in place of 
Assumption \ref{ass:smoothXXX}, is in force. Then,
$K$ is taken to be an $(\lfloor\alpha\rfloor+1)$-order kernel satisfying also the condition
$\int_{\mathbb R}|z|^{\alpha+1}|K(z)|\,\di z<\infty$. For $\alpha>2$, the kernel $K$ is not a probability density because it takes negative values. 
Nonetheless, if, in addition to Assumption \ref{ass:smoothXXX1}, $\mu_{0X}$ satisfies condition \eqref{eq:ass1}, then 
we still have
\begin{eqnarray*}\label{eq:Wi}
W_1(\mu_X,\,\mu_{0X}) \lesssim h^{\alpha+1}+T, \quad\,\,\,\, T \lesssim  W_1(\mu_Y,\,\mu_{0Y})+ h^{-(\beta-1)_+}|\log h|^{1+\1_{(\beta\leq1)}}\|f_Y-f_{0Y}\|_1.
\end{eqnarray*}
}
\end{rmk}

\begin{rmk}\label{rk:minimax}
\em{Theorem \ref{theo:1} can  be used to study both Bayesian and frequentist deconvolution procedures. In Section \ref{sec:geneBayes} we consider Bayesian posterior convergence rates while  in Section \ref{sec:comteLatour}, we analyse an estimator based on the deconvolution kernel density estimator considered in \cite{comtelacourihp}. Using Theorem \ref{theo:1}, we show that in both cases, for the Laplace noise ($\beta=2$), the derived rate is minimax-optimal.}
\end{rmk}

\smallskip
The result of Theorem \ref{theo:1} falls within the scope of inversion inequalities, 
which translate an $L^p$-distance, $p\geq1$, between kernel mixtures into a proximity measure
between the corresponding mixing distributions.
A first inequality has been obtained by \cite{nguyen2013}, Theorem 2, p. 377, for ordinary and supersmooth
kernel densities in convolution mixtures, see also \cite{heinrich2018}.
In dimension one, with ordinary smooth error distributions, refined inversion inequalities from the Hellinger or $L^1/L^2$-distance between $f_Y$ and $f_{0Y}$ to the $L^1$-Wasserstein distance between the corresponding mixing measures $\mu_X$ and $\mu_{0X}$ have been elaborated by \cite{gao2016, scricciolo:18}, but neither of these inequalities are sharp to lead to minimax-optimal estimation rates.

To better understand the implications of Theorem \ref{theo:1}, we first analyse the case $d=1$. In the direct problem, under the bound $\|f_Y - f_{0Y}\|_1 \leq \tilde \epsilon_n$ and in the context of Remark \ref{rkT}, one obtains that, for $\beta>1$, choosing $h\equiv h_n = [\tilde \epsilon_n(\log n)]^{1/(\alpha+\beta)}$, 
$$W_1(\mu_X,\, \mu_{0X}) \lesssim (\tilde\epsilon_n\log n)^{(\alpha+1)/(\alpha+\beta)}.$$ 
This reasoning has been used in Section \ref{sec:geneBayes}. As mentioned in Remark \ref{rkT}, in the case of Bayesian estimation and posterior contraction rates, 
Theorem \ref{lem:1} states that the Kullback-Leibler prior mass condition \eqref{con1}, together with the assumptions that $\mu_{0Y}\in \mathscr P_0(\mathbb R^d)\cap \mathcal P_{2+\delta}(\mathbb R^d)$, for some $\delta>0$, and that the posterior distribution is asymptotically supported on probability measures with uniformly bounded $(2+\delta)$th moments, yields a posterior contraction rate for $W_1(\mu_Y,\,\mu_{0Y})$ of the order $O(\tilde\epsilon_n \log n)$, where $\tilde \epsilon_n$ is the posterior convergence rate of
$\|f_Y-f_{0Y}\|_1$. For $\tilde\epsilon_n = n^{-(\alpha + \beta )/[2(\alpha + \beta )+1]} (\log n)^{q_1}$, we get
$$W_1(\mu_X,\,\mu_{0X}) \lesssim n^{-(\alpha+1)/[2(\alpha+\beta)+1]} (\log n)^{q_2}$$ 
for some $q_1,\, q_2>0$. For the sake of simplicity, we neglect logarithmic factors in the following discussion. 
The above rate $n^{-(\alpha+\beta)/[2(\alpha+\beta)+1]}$ for $\|f_Y-f_{0Y}\|_1$ in the direct density estimation problem is expected to occur when $f_{0X}$ has (H\"older or Sobolev) regularity $\alpha> 0$, see also Theorems \ref{thm:31} and \ref{thm:4} for the special case of a Laplace error. 
The rate $n^{-(\alpha+1)/[2(\alpha+\beta)+1]}$ matches with the lower bound on the $L^1$-Wasserstein risk for estimating $\mu_{0X}$ given in
Theorem \ref{thm:lower bound}, showing that, up to a log-factor, the rate $n^{-(\alpha+1)/[2(\alpha+\beta)+1]}$ is minimax-optimal.  
Theorem \ref{theo:1}, however, does not satisfactorily
cover the case when $0<\beta <1$. In this case, in fact,
it yields the rate $n^{-(\alpha+\beta)/[2(\alpha+\beta)+1]}$ when $f_{0X}$ is $\alpha$-regular and 
the rate $n^{-\beta/(2\beta+1)}$ when $\mu_{0X}$ is only known to have a density $f_{0X}$. 
Both rates are slower than the respective lower bounds $n^{-(\alpha+1)/[2\alpha+(2\beta\vee 1)+1]}$ and 
$n^{-1/[(2\beta\vee 1)+1]}$ given in Theorem \ref{thm:lower bound}.

When $d>1$, the use of the inversion inequality is less straightforward because, still assuming for the sake of simplicity that $\beta d>1$, the term
$$ \max_{ \mathsf v \in \mathbb S^{d-1}}  
\bigg(1 + h^{-\beta |I_h^\ast(\mathsf v)|+1} 
\prod_{j\in I_h^\ast(\mathsf v)} |v_j|^\beta\bigg)
\|f_{Y,\mathsf v} -f_{0Y,\mathsf v}\|_1
 $$
is more involved, even though it has the correct behaviour to control $W_1(\mu_X,\,\mu_{0X})$. 
In fact, it reduces the problem to univariate projections $\mathsf v\cdot \mathsf Y$, $\mathsf v\cdot \mathsf X$ and $\mathsf v \cdot \boldsymbol\varepsilon$, with a penalty in terms of $h$ that takes into account the \textit{correct regularity} of the resulting noise 
$\mathsf v \cdot \boldsymbol\varepsilon$, namely, $\beta |I_h^\ast(\mathsf v)|$. Following the above discussion and pretending that, for each $\mathsf v$, the kernel type deconvolution estimator $\tilde\mu_{1n}$ defined in Section \ref{sec:ke} only depends on the $(\mathsf v\cdot \mathsf Y_i)$'s, for $i\in[n]$, the distance $\|f_{\tilde\mu_{Yn,\mathsf v}}-f_{\mu_{0Y,\mathsf v}}\|_1$ would be bounded by $n^{-(\alpha+\beta |I_h^\ast(\mathsf v)|)/(2 \alpha + 2 \beta |I_h^\ast(\mathsf v)| + 1)}$. Then, considering $h = n^{-1/(2\alpha+2\beta d+1)}$ would yield $W_1(\tilde\mu_{1n},\, \mu_{0X}) \lesssim n^{-(\alpha+1)/(2 \alpha+2\beta d+1)}$, up to a log-factor. Obviously, $\tilde\mu_{1n}$ depends on the $\mathsf Y_i$'s and not only on the projected observations $(\mathsf v\cdot \mathsf Y_i)$'s, for $i\in[n]$. Nonetheless, we get a bound on $\|f_{\tilde\mu_{Yn,\mathsf v}}-f_{\mu_{0Y,\mathsf v}}\|_1$ of the order $O(n^{-(\alpha + \beta |I_h^\ast(\mathsf v)|)/(2 \alpha + 2 \beta d + 1)})$, which still leads to the minimax rate $n^{-(\alpha+1)/(2 \alpha + 2\beta d +1)}$. 

In a Bayesian framework, instead, controlling $\|f_{Y,\mathsf v} -f_{0Y,\mathsf v}\|_1$ for all $f_Y$ in the bulk of the posterior distribution is challenging and is left for future work.

\subsection{Error distribution with unknown scale parameter}
The inversion inequality presented in Theorem \ref{theo:1} 
goes through to the  
convolution model where the coordinate error distribution is known up to a common scale parameter.
Consider observations
\begin{equation}\label{eq:unknown_noise}
\mathsf Y_i= \mathsf X_i+\tau_0 \boldsymbol \varepsilon_i, \quad i\in[n],
\end{equation}
where the $\mathsf X_i$'s and $\boldsymbol \varepsilon_i$'s are as described in Section \ref{sec:intro}.  
There are two unknown elements in this model that need to be recovered: 
the common law $\mu_{0X}$ of the $\mathsf X_i$'s 
and the scale parameter $\tau_0>0$ of the coordinate error 
density $f_{\varepsilon,\tau_0}=(1/\tau_0)f_{\varepsilon}(\cdot/\tau_0)$.
\begin{prop}
Consider model \eqref{eq:unknown_noise} 
with the single coordinate error density satisfying the following condition:
there exists a constant $c>0$ such that
\begin{equation}\label{eq:lipschitz}
\forall\,\tau,\,\tau_0>0,\quad\|f_{\varepsilon,\tau}-f_{\varepsilon,\tau_0}\|_1\leq c\frac{|\tau-\tau_0|}{\tau\tau_0}.
\end{equation}
Under the assumptions of the first part of Theorem \ref{theo:1}, we have that
\begin{eqnarray}\label{eq:t2bis}
W_1(\mu_X, \mu_{0X}) &\lesssim & h+
W_1(\mu_{Y,\tau},\,\mu_{0Y,\tau_0})+|\tau-\tau_0| + T,
\end{eqnarray}
where $T$ is given by the expression in \eqref{eq:t2} with $\|f_{Y, \mathsf v}-f_{0Y, \mathsf v}\|_1$ replaced by 
$$\frac{|\tau-\tau_0|}{\tau\tau_0}+\|f_{Y, \tau,\mathsf v}-f_{0Y, \tau_0,\mathsf v}\|_1.$$
\end{prop}
\begin{proof}
By Theorem \ref{theo:1}, we have that
$$W_1(\mu_X, \mu_{0X}) \lesssim h+ W_1(\mu_{Y,\tau_0},\,\mu_{0Y,\tau_0})+ T_{\tau_0},$$
where $T_{\tau_0}$ is given by the expression in \eqref{eq:t2} with $\|f_{Y, \mathsf v}-f_{0Y, \mathsf v}\|_1$ replaced by $\|f_{Y,\tau_0,\mathsf v}-f_{0Y, \tau_0,\mathsf v}\|_1$. 
 Let $\mathsf Y=\mathsf X+\tau\boldsymbol \varepsilon$ be distributed according to $\mu_{Y,\tau}$ and $\mathsf Y=\mathsf X+\tau_0\boldsymbol \varepsilon$ according to $\mu_{Y,\tau_0}$.
By the triangle inequality and the bound 
$W_1(\mu_{Y,\tau},\,\mu_{Y,\tau_0})\leq \mathbb E[|(\mathsf X+ \tau\boldsymbol \varepsilon)-(\mathsf X+\tau_0\boldsymbol \varepsilon)|]=M_1(\mu_\varepsilon^{\otimes d})|\tau-\tau_0|
\lesssim |\tau-\tau_0|$, we have that 
$$W_1(\mu_{Y,\tau_0},\,\mu_{0Y,\tau_0})\lesssim W_1(\mu_{Y,\tau},\,\mu_{0Y,\tau_0})+|\tau-\tau_0|.$$ Besides, 
from
$\|f_{Y,\tau,\mathsf v}-f_{Y,\tau_0,\mathsf v}\|_1\leq\|f_{Y,\tau}-f_{Y,\tau_0}\|_1\leq \|f_{\varepsilon,\tau}^{\otimes d}- f_{\varepsilon,\tau_0}^{\otimes d}\|_1\leq d\|f_{\varepsilon,\tau}-f_{\varepsilon,\tau_0}\|_1$ and condition \eqref{eq:lipschitz}, it follows that
\[
\begin{split}
\|f_{Y,\tau_0,\mathsf v}-f_{0Y,\tau_0,\mathsf v}\|_1 &\leq \|f_{Y,\tau_0,\mathsf v}-f_{Y,\tau,\mathsf v}\|_1 +\|f_{Y,\tau,\mathsf v}-f_{0Y,\tau_0,\mathsf v}\|_1\\
&\leq d c\frac{|\tau-\tau_0|}{\tau\tau_0}+
\|f_{Y,\tau,\mathsf v}-f_{0Y,\tau_0,\mathsf v}\|_1,
\end{split}
\]
which completes the proof. 
\end{proof}
\begin{rmk}
\emph{
Condition \eqref{eq:lipschitz} is verified for the ordinary smooth error distributions listed in Section \ref{subsec:Ass}. 
Specifically, for Laplace densities 
see, \emph{e.g.}, (A.6) with $p=1$ in Lemma A.2 of \cite{scricciolo:2011}, p. 300; for Linnik densities the result follows from the fact that 
they are scale mixtures of Laplace, while for gamma densities it can be directly checked when $|\tau-\tau_0|<1$.}
\end{rmk}

\begin{rmk}
\emph{Whether the inversion inequality with the term $T$ bounded as in \eqref{eq:t2bis} can be used to recover the mixing distribution 
in a convolution model with single coordinate error density known up to a scale parameter
is a critical question related to the identifiability as a sufficient condition for the existence of consistent estimators.
In the present context, it is not clear whether 
the distribution of the $\mathsf Y_i$'s uniquely determines the scale parameter $\tau_0$ and the distribution $\mu_{0X}$. In fact,
as remarked by \cite{butuceamathias:2005}, p. 312,
it is important that the distribution to be deconvolved be significantly less smooth than the error distribution,
which is not the case when both the error and mixing distributions are ordinary smooth. 
We should mention that, at least for $d=2$, identifiability has been proved by \cite{10.1214/21-AOS2106}, see Theorem 2.1, p. 306.
It remains, however, an open question whether fast rates of convergence are possible.
Typically, in presence of identifiability problems,
either more restrictive conditions are imposed on the mixing distribution or additional data are required.
If the scale parameter is estimable without loss in the speed of convergence, then the inversion inequality 
can be used to estimate the mixing distribution. Yet, a thorough investigation of this issue is 
beyond the scope of this paper and we refer the reader to Chapter 2 of \cite{Meister:2009}, pp. 5--105,
as well as to the references therein, for a more complete discussion of the 
various aspects of the problem.
}
\end{rmk}
 



\section{Application to Bayesian estimation: posterior rates of convergence for $L^1$-Wasserstein deconvolution} \label{sec:geneBayes}
In this section, we first provide a general theorem on posterior rates of convergence for $W_1(\mu_X, \mu_{0X})$ and then apply it to the univariate deconvolution problem using a Dirichlet process mixture-of-normals prior on the mixing density $f_X$. 
 
\subsection{Posterior rates of convergence for deconvolution on $\mathbb R^d$}\label{subsec:Wrate}
We state a general theorem on posterior contraction rates. The proof is reported in Section \ref{sec:proofs}. 

\begin{thm}\label{thm:22}
Let $\Pi_n$ be a prior distribution on $\mathscr P(\mathbb R^d)$, $d\geq1$. Suppose that,
for $\delta>0$, we have
$\mu_{0X}\in\mathcal P_{4+\delta}(\mathbb R^d)$ and the error distribution is $\mu_\varepsilon^{\otimes d}$, with single coordinate distribution $\mu_\varepsilon\in\mathcal P_{4+\delta}(\mathbb R)$ satisfying Assumption \ref{ass:identifiability+error} for some $\beta>0$. 
Furthermore, for a sequence $\tilde{\epsilon}_n\geq \sqrt{(\log n)/n}$ such that
$\tilde{\epsilon}_n\f0$,
constants $c_1,\, c_2,\,c_3,\,c_4,\,K'>0$ and sets $\mathscr P_n\subseteq \{\mu_X:\,M_{4+\delta} (\mu_Y)\leq K'\tilde \epsilon_n^{-2}\}$, 
\begin{equation}\label{con1}
\begin{split}
\log D(\tilde{\epsilon}_n,\,\F(\mathscr P_n),\,d) &\leq c_1n\tilde{\epsilon}^2_n,\\[-3pt]
 \Pi_n(\mathscr P_n^c)&\leq c_3\exp{(-(c_2+4)n\tilde{\epsilon}^2_n)},\\[-3pt] 
\Pi_n(B_{\mathrm{KL}}(P_{0Y};\,\tilde\epsilon_n^2))&\geq c_4\exp{(-c_2n\tilde{\epsilon}^2_n)}.
\end{split}
\end{equation}
Then, for $\epsilon_n:=[\tilde\epsilon_n(\log n)^{1+\1_{(\beta d\leq1)}}]^{1/(\beta d\vee 1)}$ and sufficiently large constant $\bar C>0$,
$$\Pi_n(\mu_X:\,W_1(\mu_X,\,\mu_{0X})>\bar C\epsilon_n\mid \mathsf Y^{(n)})\rightarrow0 \mbox{ in $P_{0Y}^n$-probability.}$$

If, in addition, $\mu_{0X}$ satisfies  Assumption \ref{ass:smoothXXX} for $\alpha>0$ and
there exist a constant $C_1>0$ and a kernel $K$ as in (a) such that, for every $\mu_X\in\mathscr P_n$, 
\begin{equation*}
\max_{\mathsf v\in\mathbb S^{d-1}}\|b_{F_{X,\mathsf v}}(h_n)\|_1\leq C_1 h_n^{\alpha+1}, \quad \mbox{ with } \, h_n=[\tilde\epsilon_n(\log n)^{1+\1_{(\beta d\leq 1)}}]^{1/[\alpha+(\beta d\vee 1)]},
\end{equation*}
then, for $\epsilon_{n,\alpha}:=[\tilde\epsilon_n(\log n)^{1+\1_{(\beta d\leq1)}}]^{(\alpha+1)/[\alpha+(\beta d\vee 1)]}$ and $C_\alpha>0$ large enough,
\begin{equation*}\label{eq:convsmooth}
\Pi_n(\mu_X:\,W_1(\mu_X,\,\mu_{0X})>C_\alpha\epsilon_{n,\alpha}\mid \mathsf Y^{(n)})\rightarrow0 \mbox{ in $P_{0Y}^n$-probability.}
\end{equation*}
\end{thm}


\smallskip

Theorem \ref{thm:22} provides sufficient conditions on the prior distribution and the data generating process 
so that the corresponding posterior measure asymptotically concentrates on $L^1$-Wasserstein balls centered at $\mu_{0X}$. As a consequence, the posterior mean $\hat \mu_n^{\mathrm B}:=\int\mu_X\,\di\Pi_n(\mu_X\mid \mathsf Y^{(n)})$ converges to $\mu_{0X}$ in the $L^1$-Wasserstein distance at least as fast as $\epsilon_n$ or $\epsilon_{n,\alpha}$.

\begin{cor}\label{cor:postmean}
Under the assumptions of Theorem \ref{thm:22}, the posterior mean $\hat \mu_n^{\mathrm B}$
converges to $\mu_{0X}$ in the $L^1$-Wasserstein distance at rate $\epsilon_n$, namely, there exists $M'>0$ such that, with 
$P_{0Y}^n$-probability tending to $1$,
$$W_1(\hat \mu_n^{\mathrm B},\, \mu_{0X}) \leq M' \epsilon_n,$$
or $\epsilon_{n,\alpha}$ under the assumptions of the second part of Theorem \ref{thm:22}. 
\end{cor}

Corollary \ref{cor:postmean} can be proved using standard arguments, see, {\em e.g.}, Theorem 8.8 of \cite{bookgvdv}, p. 196.

\smallskip

Some remarks and comments on two main issues, (i) the relationship between the rates for the direct and the inverse problems,
(ii) rate optimality, are in order. As for issue (i),
Theorem \ref{thm:22} connects to existing results that give sufficient conditions
for assessing posterior convergence rates in the direct density estimation problem. 
In fact, the conditions in \eqref{con1} imply that, for a sufficiently large $\,\xbar M>0$, 
$$\mathbb E_{0Y}^n[\Pi_n(\mu_X:\,\|f_{Y}-f_{0Y}\|_1>\,\xbar M \,\tilde \epsilon_n\mid \mathsf Y^{(n)})]\rightarrow0,$$
see \cite{ghosal2000}, Theorem 2.1, p. 503, which states that the posterior concentration rate on 
$L^1$-neighbourhoods of $f_{0Y}$ is $\tilde\epsilon_n$. Alternative conditions for assessing posterior contraction rates
in $L^r$-metrics, $1\leq r\leq \infty$, are given in \cite{Gine:Nickl:11}, see Theorems 2 and 3, pp. 2891--2892. 
As for issue (ii), a remarkable feature of Theorem \ref{thm:22} is the fact that, 
to obtain $L^1$-Wasserstein posterior convergence rates for $\mu_X$, 
which is a mildly ill-posed inverse problem, it is enough to derive 
posterior contraction rates relative to the $L^1$-metric in the direct density estimation problem, 
which is more gestible. In fact, granted Assumption \ref{ass:smoothXXX},
the essential conditions to verify are those listed in \eqref{con1}, which are sufficient for the posterior 
distribution to contract at rate $\tilde\epsilon_n$ around $f_{0Y}$. This simplification is due to the inversion inequality of 
Theorem \ref{theo:1}, which holds true under Assumption \ref{ass:identifiability+error} only,
when no smoothness condition is imposed on $\mu_{0X}$, and jointly with
condition \eqref{eq:ass1}, when the smoothness Assumption \ref{ass:smoothXXX} on $\mu_{0X}$ is in force. 
Application of Theorem \ref{thm:22} to specific models gives further insight into this aspect.
In Section \ref{sec:lap+exp}, for the case $d=1$, we consider a Dirichlet process mixture-of-Laplace-normals prior 
and find the rate $n^{-1/5}(\log n)^\kappa$ when the latent distribution $\mu_{0X}$ is only known to have a density $f_{0X}$, and the rate $n^{-(\alpha +1)/(2\alpha+5)}(\log n)^{\kappa'}$ when the mixing density $f_{0X}$ is $\alpha$-Sobolev regular. These rates match with the lower bound given in Theorem \ref{thm:lower bound} and are, therefore, minimax-optimal, up to log-factors. 
When $d\geq2$ and $\beta d\geq 1$, to assess $W_1$-posterior contraction rates for $\mu_{0X}$ under no regularity assumptions on $f_{0X}$, we would need a posterior contraction rate for the direct density estimation problem (with respect to the $L_1$-norm distance between $f_Y$ and $f_{0Y}$) of the order 
$n^{-\beta d/[(2\beta +1)d]}=n^{-\beta/(2\beta +1)}$. However, the theory developed in Section \ref{sec:lap+exp} based on a Dirichlet process mixture-of-Laplace-normals prior does not immediately extend to the multivariate case.

\subsection{Deconvolution on $\mathbb R$ by a Dirichlet process mixture-of-Laplace-normals prior}\label{sec:lap+exp}
In this section, we study the problem of density deconvolution on the real line for mixtures with a Laplace error distribution, whose Fourier transform is given by $\hat f_\varepsilon(t) = (1 + t^2)^{-1}$, $t\in\mathbb R$.
The problem of density deconvolution with a Laplace error distribution arises also in nonparametric inference under local differential privacy, 
when a Laplace density is used in a convolution-based privacy mechanism, see, \emph{e.g.}, \cite{dwork2008}. 
In this case, in fact, the problem of recovering the common density, say $f_{0X}$ in our notation, of the original data before a perturbed version with additive errors is released, 
boils down to a density deconvolution problem with Laplace noise. 
Data privacy protection is nowadays a major issue due to the massive amount of data collected and stored.  
Local differential privacy, in particular, has lately attracted a lot of attention as a way to construct data privacy preserving mechanisms, 
see, for instance, \cite{dinur-nissim2003, Evfimievski2003LimitingPB, dwork-nissim2004, dwork2008, duchi2018} and the recent articles \cite{rohde20, butucea20} on nonparametric adaptive estimation of $f_{0X}$.

We use a Dirichlet process mixture-of-normals prior on the mixing density 
$f_X=\phi_\sigma\ast \mu_H$, so that the model density is $f_Y=f_\varepsilon \ast f_X=f_\varepsilon \ast (\phi_\sigma\ast \mu_H)$, with 
$\mu_H\sim\mathscr D_{H_0}$, a Dirichlet process with finite, positive base measure $H_0$ on $\mathbb R$, 
and $\sigma\sim\Pi_\sigma$. We consider the following assumptions on $H_0$ and $\Pi_\sigma$.

\begin{ass}\label{ass:basemeasure1}
\emph
{
The base measure $H_0$ has a continuous and positive density $h_0$ on $\mathbb R$ such that,
for constants $b_0,\,b_0',\,c_0,\,c_0'>0$ and $\iota>0$,
$$c_0\exp{(-b_0|u|^\iota)} \leq h_0(u) \leq c_0'\exp{(-b_0'|u|^\iota)},\quad u\in\mathbb R.$$
}
\end{ass}

\begin{ass}\label{ass:priorscale1}
\emph
{The prior distribution $\Pi_\sigma$ for $\sigma$ has a continuous 
density $\pi_\sigma$ on $(0,\,\infty)$ 
such that, for constants $D_1,\,D_2>0$ and $s_1, s_2,\,t_1,\,t_2\geq0$,
\[\sigma^{-s_1}\exp{(-D_1\sigma^{-1}|\log\sigma|^{t_1})}\lesssim  \pi_\sigma(\sigma) \lesssim \sigma^{-s_2}\exp{(-D_2\sigma^{-1}|\log\sigma|^{t_2})}
\]
for all $\sigma$ in a neighborhood of $0$. Furthermore, for constants $D_3,\,\varpi>0$, 
the tail probability $\Pi_{\sigma}((\bar\sigma,\infty))\lesssim \exp{(-D_3\bar\sigma^\varpi)}$ as $\bar\sigma\rightarrow \infty$.
}
\end{ass}

Assumption \ref{ass:basemeasure1} on the base measure $H_0$ of the Dirichlet process is analogous to (4.8) in \cite{scricciolo:2011}, p. 288, and holds true, for example, when $h_0$ is the density of an exponential power distribution with shape parameter $\iota>0$, which includes the Laplace distribution ( $\iota=1$), and the Gaussian distribution ( $\iota=2$).

The first part of Assumption \ref{ass:priorscale1} on the scale parameter $\sigma$ of the Gaussian kernel has become common in the literature since the articles \cite{vdV+vZanten, dejonge, kruijer:rousseau:vdv:10}.
Here we consider in addition the  tail condition for large values of $\sigma$, which requires $\Pi_\sigma$ to have 
an exponentially decaying tail also at infinity.  Examples of densities satisfying these two conditions are  inverse Gamma distribution restricted to $(0,\,\bar\sigma]$,
for $0<\bar\sigma<\infty$. An example of distribution supported on $(0,\,\infty)$ that verifies Assumption \ref{ass:priorscale1} is given in \cite{scricciolo:2011}, p. 291, where $\pi_\sigma$ is 
proportional to an inverse-gamma $\mathrm{IG}(1,\,\zeta)$ on $(0,\,1]$ and to a Weibull $W(\zeta,\,\nu)$ on $(1,\,\infty)$,
where $\zeta>0$ is the scale parameter and $\nu>0$ the shape parameter.
Then,  $s_1=s_2=\zeta+1$, $t_1=t_2=0$ and $\varpi =\nu$.
The assumption on the upper tail of $\Pi_\sigma$ is used to guarantee that condition \eqref{prior:moment} is satisfied, 
which, in virtue of Theorem \ref{lem:1}, allows to control $W_1(\mu_Y,\,\mu_{0Y})$ in terms of $\|f_Y-f_{0Y}\|_1$.

\smallskip

We also consider the following assumption on the tails of the mixing distribution: 

\begin{ass}\label{ass:twicwtailcond}
\emph{
The mixing distribution $\mu_{0X}\in\mathscr P_0(\mathbb R)$ 
has density $f_{0X}(x)\lesssim e^{-(1+C_0)|x|}$, $x\in\mathbb R$, with some constant $C_0>0$.
}
\end{ass}

First we study the case in which mixing distribution satisfies only Assumption \ref{ass:twicwtailcond} and then the  case where it also has a density
Sobolev regularity $\alpha$. In the latter case, the prior distribution on the mixing density does not depend on $\alpha$, 
yet it yields an adaptive posterior contraction rate. 
We refer to these two cases as non-adaptive and adaptive, respectively, and treat them separately.

\subsection{Non-adaptive case}\label{sec:nonadaptive}
Let $\Pi$ be the prior distribution induced on $\mathscr F$ by the product measure $\mathscr D_{H_0}\otimes \Pi_\sigma$ 
on the parameter $(\mu_H,\,\sigma)$ of the density $f_Y=f_\varepsilon\ast(\phi_\sigma\ast \mu_H)$, 
for a standard Laplace error density $f_\varepsilon$.
Let also the sampling density $f_{0Y}=f_\varepsilon\ast f_{0X}$ be a Laplace mixture,
with mixing density $f_{0X}$ satisfying the following exponential tail decay condition.


We begin by assessing posterior contraction rates in the $L^1$-metric for Laplace 
convolution mixtures with mixing distributions having exponentially decaying tails.

\begin{thm}\label{thm:31}
Let $Y_1,\,\ldots,\,Y_n$ be i.i.d. observations from
$f_{0Y}:=f_\varepsilon\ast f_{0X}$, where $f_\varepsilon$ is the density of the 
standard Laplace distribution and $f_{0X}$ satisfies Assumption \ref{ass:twicwtailcond}. 
Let $\Pi$ be the prior distribution induced by $\mathscr D_{H_0} \otimes \Pi_\sigma$, where 
${H_0}$ verifies Assumption \ref{ass:basemeasure1} and $\Pi_\sigma$ verifies
Assumption \ref{ass:priorscale1}.  
Then, the conditions in \eqref{con1} are satisfied for $\tilde\epsilon_n=n^{-2/5}(\log n)^{\varphi}$, 
with some $\varphi>0$, and there exists $D$ large enough so that
\[
\Pi(\mu_Y:\,\|f_Y-f_{0Y}\|_1>D \tilde\epsilon_n
\mid \Data)\rightarrow0\,\mbox{ in $P_{0Y}^n$-probability}.
\]
\end{thm}

\begin{proof}
We argue that the conditions in \eqref{con1} are 
satisfied for $\tilde\epsilon_n$ as in the statement. 
The small ball prior probability estimate in the third inequality of \eqref{con1}
is verified taking into account Remark \ref{rmk:KLneigh} and Lemma \ref{lem:KL}, 
which is based on the construction of an approximation of $f_{0Y}$ by 
$f_\varepsilon\ast(\phi_\sigma\ast\mu_H)$, for a carefully chosen probability measure $\mu_H$. 
This construction adapts the proof of Lemma 2 of \cite{gao2016}, pp. 615--616, to obtain an approximation error of the order $O(\tilde\epsilon_n)$, 
as shown in Lemma \ref{lem:discrete}.  
The entropy and remaining mass conditions, the first two inequalities in \eqref{con1},
are consequences of Theorem 5 of \cite{ghosal:shen:tokdar}, p. 631, because, for any pair of densities
 $f_1$ and $f_2$, we have $ \| f_\varepsilon \ast (f_1-f_2) \|_1 \leq \|f_1-f_2\|_1$. 
Finally, since $\mu_X$ has density $\phi_\sigma\ast\mu_H$ so that $X=\sigma Z+ U$, 
with $Z\sim N(0,\,1)$ and $U\sim \mu_H$, we have
$M_1(\mu_X)\leq\sigma\mathbb E[|Z|]+M_1(\mu_H)<\infty$, that is, $\mu_X\in \mathcal P_1(\mathbb R)$ almost surely,
because $\mathscr D_{H_0}(\mu_H:\,M_1(\mu_H)=\infty)=0$. 
The assertion follows.
\end{proof}


A rate of the order $O(n^{-2/5})$, up to a logarithmic factor,
is achieved for estimating mixtures of Laplace densities
if a kernel mixture prior on the mixing density is constructed using a Gaussian kernel,
with an inverse-gamma type bandwidth $\sigma$ and a Dirichlet process prior on $\mu_H$.
The result is new in Bayesian density estimation and is a preliminary step for the following
$L^1$-Wasserstein deconvolution result.

\begin{thm}\label{thm:32}
Let $Y_1,\,\ldots,\,Y_n$ be i.i.d. observations from
$f_{0Y}:=f_\varepsilon\ast f_{0X}$, where $f_\varepsilon$ is the density of the 
standard Laplace distribution and $f_{0X}$ satisfies Assumption \ref{ass:twicwtailcond}. 
Let $\Pi$ be the prior distribution induced by $\mathscr D_{H_0} \otimes \Pi_\sigma$, where 
${H_0}$ verifies Assumption \ref{ass:basemeasure1} for $\iota>1$ and $\Pi_\sigma$ verifies
Assumption \ref{ass:priorscale1}  with $\varpi>1$. Then, 
there exist $K$ large enough and $\kappa>0$ so that
\begin{equation*}\label{eq:69}
\Pi(\mu_X:\,W_1(\mu_X,\,\mu_{0X})>Kn^{-1/5}(\log n)^\kappa\mid\Data)
\rightarrow0\,\mbox{ in $P_{0Y}^n$-probability.}
\end{equation*}
\end{thm}

\begin{proof}
We apply Theorem \ref{thm:22}.
For any $\delta>0$, with the Laplace distribution we have $\mu_\varepsilon\in\mathcal P_{4+\delta}(\mathbb R)$.
By Assumption \ref{ass:twicwtailcond}, also $\mu_{0X}\in\mathcal P_{4+\delta}(\mathbb R)$.
We know from Theorem \ref{thm:31} that the conditions in \eqref{con1} hold
for $\tilde\epsilon_n=n^{-2/5}(\log n)^{\varphi}$. 
It remains to show that, for a suitable $c>0$,
\begin{equation}\label{posteriorPnc}
\Pi(\mu_X:\, M_{4+\delta}(\mu_X) >K'' \tilde \epsilon_n^{-2})\lesssim \exp{(-cn\tilde\epsilon_n^2)}.
\end{equation}
Recalling that $\mu_X$ has density $\phi_\sigma\ast\mu_H$ so that $X=\sigma Z+ U$, 
with $Z\sim N(0,\,1)$ and $U\sim \mu_H$, we have
$M_{4+\delta}(\mu_X)\lesssim\sigma^{4+\delta}\mathbb E[|Z|^{4+\delta}]+M_{4+\delta}(\mu_H)$.
Therefore, for $M_1>0$, 
\[
M_{4+\delta}(\mu_X) \lesssim \sigma^{4+\delta} + M_{4+\delta}(\mu_H)\lesssim \sigma^{4+\delta} + M_1\tilde \epsilon_n^{-2}+
M_{4+\delta}(\1_{(|U|^{4+\delta}>M_1\tilde \epsilon_n^{-2})} \mu_H).
\]
Assumption \ref{ass:priorscale1} on the upper tail of $\Pi_\sigma$ implies that, 
for a suitable $c_1>0$, 
$$\Pi_\sigma (\sigma:\,\sigma>(M_1\tilde \epsilon_n^{-2})^{1/(4+\delta)})\leq 
\exp{(- D_3(M_1\tilde \epsilon_n^{-2})^{\varpi/(4+\delta)})} 
\lesssim \exp{(-c_1n\tilde \epsilon_n^2)}$$
provided that $0<\delta\leq4(\varpi-1)$, where $\varpi>1$ by hypothesis.
Besides, for a suitable $c_2>0$, 
\begin{align*}
\mathscr D_{H_0}(\mu_H:\,M_{4+\delta}(\1_{(|U|^{4+\delta}>M_1\tilde \epsilon_n^{-2})} \mu_H)> M_1\tilde \epsilon_n^{-2})&\lesssim \int_{|u|^{4+\delta}>M_1\tilde \epsilon_n^{-2}}|u|^{4+\delta} h_0(u)\,\di u \\
&\lesssim \exp{(- b_0'(M_1\tilde \epsilon_n^{-2})^{-\iota/(4+\delta)}/2)}\\
&\lesssim \exp{(-c_2n\tilde \epsilon_n^2)}
\end{align*}
provided that $0<\delta\leq4(\iota-1)$, where $\iota>1$ by hypothesis. Hence, choosing $\delta$ small enough so that both the above requirements are satisfied, condition \eqref{posteriorPnc} holds true and the proof is complete. 
\end{proof}

\smallskip
%

We now exhibit another example of convolution model for which
a statement in the same spirit as that of Theorem \ref{thm:32} can be obtained.
Let the random variable $Z$ have monotone non-increasing density $f_Z$ on $(0,\,\infty)$. 
Following \cite{williamson:56}, it is known that $f_Z$ is a scale mixture of uniform densities, $f_Z(z) = \int_0^\infty[\1_{[0,\,v]}(z)/v]\, \di F(v)$,
so that $Y = \log Z = X - \varepsilon$, 
where $\varepsilon \sim \mathrm{Exp}(1)$ is independent of $X$.  
Under suitable conditions, the posterior convergence rate at $f_{0Y}$ relative to the Hellinger or $L^1$-distance
is $n^{-1/3}$, up to a logarithmic factor, see, \emph{e.g.}, Theorem 2 in \cite{JBSalomond2014}, pp. 1384--1385.
Then, the posterior distribution of $\mu_X$ concentrates around $\mu_{0X}$ at rate $n^{-1/3}$, up to a log-factor,
in a metric similar to the $L^1$-Wasserstein. In fact, writing $f_Z(z) = \int_0^\infty [\1_{[z,\,\infty)}(v)/v]\, \di F(v)$ and 
$f_{0Z}(z) = \int_0^\infty[\1_{[z,\,\infty)}(v)/v]\, \di F_0(v)$, we have
$$\Pi(F:\, \tilde W(F,\,F_0) > M n^{-1/3} (\log n)^\nu \mid Z^{(n)})]\rightarrow 0\,\mbox{ in $P_{0Z}^n$-probability},$$
where $\tilde W(F,\,F_0):= \int_0^\infty[|F(v) - F_0(v)|/v]\, \di v$. 
We believe that this rate is optimal, up to a log-factor, since $F(v)=1-f_Z(v)/f_Z(0)$, see \cite{Balabdaoui&Wellner:2007}, p. 2538.
\subsection{Sobolev-regularity adaptive case}\label{subsec:adapt}
In this section, we focus on the case where the sampling density $f_{0Y}$ is a mixture of 
Laplace densities with a Sobolev regular mixing density. 
We still consider the prior distribution $\Pi$ induced on $\mathscr F$ by the product measure $\mathscr D_{H_0}\otimes \Pi_\sigma$ 
for the parameter $(\mu_H,\,\sigma)$ of $f_Y=f_\varepsilon\ast(\phi_\sigma\ast\mu_H)$, 
with a standard Laplace error density $f_\varepsilon$. Let the corresponding posterior distribution $\Pi(\cdot\mid Y^{(n)})$ 
be based on i.i.d. observations $Y_1,\,\ldots,\,Y_n$ from 
$f_{0Y}=f_\varepsilon\ast f_{0X}$, which is a Laplace mixture with mixing density $f_{0X}$ 
satisfying the following conditions. 
\begin{ass}\label{ass:sobolevcond}
\emph
{There exists $\alpha>0$ such that
\begin{equation*} \label{Sobolev:cond}
\forall\,b=\mp\frac{1}{2},\quad\int_{\mathbb R} |t|^{2\alpha}|\widehat{(e^{b\cdot}f_{0X})}(t)|^2\,\di t < \infty.
\end{equation*}
}
\end{ass}
\begin{ass}\label{ass:smoothf0}
\emph{For given $\alpha>0$, there exist $0<\upsilon\leq1$, $L_0\in L^1(\mathbb R)$ and $R\geq (2m/\upsilon)$, with the smallest integer $m\geq[2\vee(\alpha+2)/2]$, such that $f_{0X}$
satisfies
\begin{equation}\label{eq:holderf0}
|f_{0X}(x+\zeta)-f_{0X}(x)|\leq L_0(x)|\zeta|^{\upsilon} \,\mbox{ for every }x,\,\zeta\in\mathbb R,
\end{equation}
and
\begin{equation}\label{eq:ratiofo}
\int_{\mathbb R}e^{|x|/2}f_{0X}(x)\pt{\frac{L_0}{f_{0X}}(x)}^R\di x<\infty.
\end{equation}
}
\end{ass}

\smallskip
\noindent
Assumption \ref{ass:sobolevcond} requires that, for $b=\mp\frac{1}{2}$, the function 
$e^{b\cdot}f_{0X}$ is $\alpha$-Sobolev regular, while Assumption \ref{ass:smoothf0} requires that $f_{0X}$ is locally $\upsilon$-H\"older smooth, with envelope function 
$L_0$ satisfying the integrability condition \eqref{eq:ratiofo}. 
The model $f_Y=f_\varepsilon\ast(\phi_\sigma\ast\mu_H)$ acts as an approximation scheme 
for automatic posterior rate adaptation to the global regularity of $f_{0Y}$, without any knowledge of 
the regularity of $f_{0X}$ being used in the prior specification. 
We show that a rate-adaptive estimation procedure for Laplace mixtures 
can be obtained if the prior distribution is properly constructed, for instance, as a mixture of Laplace-normal convolutions, with
an inverse-gamma type bandwidth and a Dirichlet process on the mixing distribution.

\begin{thm}\label{thm:4}
Let $Y_1,\,\ldots,\,Y_n$ be i.i.d. observations from
$f_{0Y}:=f_\varepsilon\ast f_{0X}$, where $f_\varepsilon$ is the density of the 
standard Laplace distribution and $f_{0X}$ 
satisfies Assumptions \ref{ass:twicwtailcond}--\ref{ass:smoothf0}. 
Let $\Pi$ be the prior distribution induced by $\mathscr D_{H_0} \otimes \Pi_\sigma$, where 
${H_0}$ verifies Assumption \ref{ass:basemeasure1} and $\Pi_\sigma$ verifies
Assumption \ref{ass:priorscale1}.
Then, the conditions in \eqref{con1} are satisfied for 
$\tilde \epsilon_n = n^{-(\alpha+2)/(2\alpha+5)}(\log n)^{\varphi'}$, 
with some $\varphi'>0$, and there exists $D'$ large enough so that
\[
\Pi(\mu_Y:\,\|f_Y-f_{0Y}\|_1>D' \tilde \epsilon_n
\mid \Data)\rightarrow0\,\mbox{ in $P_{0Y}^n$-probability}.
\]
\end{thm}

\begin{proof}
The entropy and remaining mass conditions, as well as the small ball prior probability estimate in \eqref{con1},
are satisfied for $\tilde\epsilon_n$ as in the statement.
For details of the entropy and remaining mass conditions, 
see, \emph{e.g.}, Theorem 5 of \cite{ghosal:shen:tokdar}, p. 631, while
for the small ball prior probability estimate apply Lemma \ref{lem:deltabound1},
together with a modified version of Lemma \ref{lem:KL}, with $\beta$ replaced by $(\alpha+2)$.
\end{proof}

Theorem \ref{thm:4} is based on the approximation Lemmas \ref{lem:contapprox} and \ref{lem:deltabound1}. 
The approximation of $f_{0Y}$ by $f_\varepsilon\ast(\phi_\sigma\ast \mu_H)$ used in the non-adaptive case of Theorem \ref{thm:31} 
is remarkably simpler than the approximation used in Lemma \ref{lem:contapprox} for the adaptive case. 
The latter is also different from the construction in \cite{kruijer:rousseau:vdv:10}. 
As in the non-adaptive case, $L^1$-Wasserstein posterior convergence rates for $\mu_X$ are derived from Theorem \ref{thm:4} 
by controlling the prior probability of the event in \eqref{posteriorPnc} and the $L^1$-norm of the bias in \eqref{eq:ass1}.

\begin{thm}\label{thm:5}
Granted the assumptions of Theorem \ref{thm:4} on $f_{0Y}$ and considered 
the same prior with $\iota>1$ and $\varpi>1$,
there exist $M'$ large enough and $\kappa'>0$ so that
\begin{equation*}\label{eq:70}
\Pi(\mu_X:\,W_1(\mu_X,\,\mu_{0X})>M'n^{-(\alpha+1)/(2\alpha+5)}(\log n)^{\kappa'}\mid\Data)
\rightarrow0\,\mbox{ in $P_{0Y}^n$-probability.}
\end{equation*}
\end{thm}

\begin{proof}
Applying Theorem \ref{thm:4}, we get $\tilde \epsilon_n = n^{-(\alpha+2)/(2\alpha+5)}(\log n)^{\varphi'}$ for some $\varphi'>0$. 
Then, reasoning as in the proof of Theorem \ref{thm:32}, it can be shown that, for some $c,\,\delta,\,K''>0$, 
condition \eqref{posteriorPnc} is satisfied.
By Lemma \ref{lem:biasmixgaus}, for $0<h\sqrt{(2\alpha+1)|\log h|}\leq\sigma<1$, we have
$\|b_{F_X}(h)\|_1 \lesssim h^{\alpha+1}$. 
For $q>0$, replace $h$ with $h_n=[\tilde\epsilon_n(\log n)^q]^{1/(\alpha+2)}=n^{-1/(2\alpha+5)}(\log n)^{(q+\varphi')/(\alpha+2)}$.
From the proof of Theorem \ref{thm:4}, over the sieve set $\mathscr P_n$, we have $\sigma \geq \sigma_n \equiv n^{-1/(2\alpha+5)} (\log n)^{q'}$, for some $q'>0$.
We can choose $q'$ so that $\sigma_n\geq h_n|\log h_n|^{\alpha+1/2}$. Then, for every $\mu_X\in\mathscr P_n$, we have
$\|b_{F_X}(h_n)\|_1 \lesssim h_n^{\alpha+1}$ as prescribed by condition \eqref{eq:ass1} and the proof is complete. 
\end{proof}

\subsubsection{Approximation result}\label{sec:approxresult}
When assessing posterior rates of convergence for kernel mixture priors, a crucial step consists in
finding a suitable approximation of the true density within the model. 
Lemma \ref{lem:contapprox}, stated below, constructs an approximation of 
a Laplace mixture density $f_{0Y}=f_\varepsilon\ast f_{0X}$, with an $\alpha$-Sobolev regular mixing density 
$f_{0X}$, by a Laplace-normal convolution $f_\varepsilon\ast (\phi_\sigma\ast\mu_H)$
so that the \vir{bias}, the $L^2$-distance between the true density and the approximation,
is of the correct order $O(\sigma^{\alpha+2})$ in terms of the kernel bandwidth $\sigma$.
Even if the approximation is a crucial technical device within the Bayesian framework, 
the result is independent of the inferential paradigm adopted and is of interest in itself. 

For $h>0$, let $$H(x):=\frac{1}{2\pi}\reallywidehat{(\tau\ast\phi_h)}(x)=\frac{1}{2\pi}\hat \tau(x)\reallywidehat{\phi_h}(x)
=\frac{1}{2\pi}\hat \tau(x) e^{-(h x)^2/2}, \quad x\in\mathbb R,$$
where $|\hat\tau(x)|\leq (16^2/15)e^{-\sqrt{|x|/15}}$, $x\in\mathbb R$, is the Fourier transform of
$\tau:\,\mathbb R\rightarrow [0,\,1]$ defined in Theorem 25 of \cite{bourgain}, p. 29, such that
$$\tau (u) = \left\{ \begin{array}{ll}
1,&\quad \text{if }  |u| < 1, \\[3pt]
0,&\quad \text{if }  |u| > 17/15.
 \end{array} \right. $$
The function $\tau$ is such that $\hat \tau$ is infinitely differentiable and 
\begin{equation}\label{eq:poldec}
\mbox{for any $i\in\mathbb N_0$,}\quad|\hat\tau^{(i)}(x)|=O(|x|^{-\nu})\quad \mbox{for large $|x|$ and every $\nu>0$.}
\end{equation}
Given $m\in\mathbb N$, $b=\mp\frac{1}{2}$, $\delta,\,\sigma>0$ and a function $f:\,\mathbb R\rightarrow\mathbb R$, we define the operator
$$f\mapsto T_{m,b,\sigma} f:=f+\sum_{k=1}^{m-1}\frac{(-\sigma^2/2)^k}{k!}
\sum_{j=0}^{2k} \binom{2k}{j}(-b)^{2k-j}[f\ast (e^{-b\cdot}D ^jH_\delta)],$$
where $H_\delta(\cdot):=(1/\delta)H(\cdot/\delta)$. Since $\delta$ will be chosen proportional to $\sigma$, the operator does not ultimately depend on $\delta$.
If $M_{0X}(b):=\int_{\mathbb R}e^{bx}f_{0X}(x)\,\di x<\infty$, introduced the density
\begin{equation}\label{eq:barh}
\bar h_{0,b}:=\frac{e^{b\cdot}f_{0X}}{M_{0X}(b)}
\end{equation}
and the constant $\gamma:=-(1-e^{-\sigma^2/8})$, let the function $h_{m,b,\sigma}:\,\mathbb R\rightarrow \mathbb R$ be defined as
\begin{equation}\label{hm}
h_{m,b,\sigma}:=\frac{1}{\gamma} \sum_{k=1}^{m-1}\frac{ (-\sigma^2/2)^k}{k!}
\sum_{j=0}^{2k} \binom{2k}{j}(-b)^{2k-j}(\bar h_{0,b}\ast D^jH_\delta).
\end{equation}
Then,
\begin{equation}\label{eq:ident}
\frac{e^{b\cdot}}{M_{0X}(b)}T_{m,b,\sigma}f_{0X}=\bar h_{0,b}+\gamma  h_{m,b,\sigma}.
\end{equation}

\smallskip
The following lemma provides the order of the approximation error, in terms of the Gaussian bandwidth $\sigma$,
of the $L^2$-norm distance between the Laplace mixture density $f_{0Y}=f_\varepsilon\ast f_{0X}$ 
and the normal-Laplace mixture of the transformation $T_{m,b,\sigma}f_{0X}$ of $f_{0X}$.

\begin{lem}\label{lem:contapprox}
Let $f_\varepsilon$ be the standard Laplace density.
Let $f_{0X}$ be a density such that $(e^{|\cdot|/2}f_{0X})\in L^1(\mathbb R)$ and
satisfies Assumption \ref{ass:sobolevcond} for $\alpha>0$.
Then, for $m\geq [2\vee(\alpha+2)/2]$ and $\sigma>0$ small enough,
\begin{equation}\label{bound:tildehm12}
\sum_{b=\mp 1/2}
\|e^{b\cdot}\{f_\varepsilon\ast [\phi_\sigma\ast (T_{m,b,\sigma}f_{0X})-f_{0X}]\}\|_2^2\lesssim\sigma^{2(\alpha+2)}
\end{equation}
and
\begin{equation}\label{eq:bound}
\forall\, b=\mp\frac{1}{2},\quad \int_{\mathbb R} h_{m,b,\sigma}(x)\,\di x = 1 + O(\sigma^{2(m-1)}).
\end{equation}
\end{lem}

The proof of Lemma \ref{lem:contapprox} is reported in Section \ref{sec:rth3}. We note that 
the result also holds when only $(e^{|\cdot|/2}f_{0X})\in L^1(\mathbb R)\cap L^2(\mathbb R)$.
This case can be regarded as corresponding to $\alpha=0$ so that $m\geq2$. The approximation error in 
\eqref{bound:tildehm12} is then of the order $O(\sigma^4)$. 
However, in this case, we can directly prove the existence of a compactly supported discrete mixing probability measure $\mu_H$, with a sufficiently small number of support points, such that the corresponding Laplace-normal mixture 
$f_\varepsilon\ast (\phi_\sigma\ast\mu_H)$ has squared Hellinger distance of the order $O(\sigma^4)$ from $f_{0Y}$, see Lemma \ref{lem:discrete}.
When $\alpha>0$, to obtain the correct order of approximation $O(\sigma^{2(\alpha+2)})$ of the squared $L^2$-bias for a
Sobolev regularity $(\alpha+2)$ of the density $f_{0Y}=f_\varepsilon\ast f_{0X}$, we construct a modification of $f_{0Y}$, to be convolved with the Gaussian kernel $\phi_\sigma$, such that, for a global level of regularity strictly larger than $2$, the new function $\phi_\sigma\ast(f_{0Y}-f_\varepsilon\ast f_1) $, with a suitable $f_1$, outperforms the natural candidate $\phi_\sigma\ast f_{0Y}$ 
for the approximation. However, the new function is not a probability density and needs to be modified. 
The resulting high quality approximation allows to use the correct bandwidth, which is selected by the prior distribution for the scale parameter from the appropriate range. Thus, the posterior contracts at the minimax-optimal rate (up to a logarithmic factor) near $f_{0Y}$, without actually 
knowing the regularity of $f_{0Y}$ and without using that knowledge in the definition of the prior on the bandwidth, yet automatically adapting to the given regularity level.
Even if the idea of constructing a correct approximation for a given level of regularity by subtracting appropriate terms from $f_{0Y}$ has previously appeared in \cite{rousseau:09, kruijer:rousseau:vdv:10, ghosal:shen:tokdar}, 
there are two main differences with the approximation results of these articles:
first, we consider global regularity on a Sobolev scale, whereas all
the above articles deal with local smoothness on a H\"{o}lder scale; second, our approximation 
is more involved as it employs a double smoothing by the Gaussian kernel $\phi_\sigma$ and by another 
super-smooth kernel $H_\delta$ proportional to the Fourier transform of a normal density
to control the error for low frequencies.


\section{$W_1$-lower bound rates for deconvolution in any dimension and application of the inversion inequality to a frequentist estimator}
\label{sec:comteLatour}
In this section, we provide lower bounds on the $L^1$-Wasserstein deconvolution convergence rates in any dimension $d\geq 1$.
These bounds are attained by the Bayes' estimator for $d=1$ and, as shown in Section \ref{sec:ke}, 
by a frequentist minimum distance estimator for every $d\geq1$. 

\subsection{Lower bounds}\label{sec:LB}
To get a validation of our results, we derive lower bound rates for the $L^1$-Wasserstein risk 
extending Theorem 4.1 in \cite{dedecker2015}, pp. 246--248, to a multivariate setting with Sobolev regular mixing densities.
\begin{thm}\label{thm:lower bound}
Assume that there exists $\beta>0$ such that, for every $l=0,\,1,\,2$,
\begin{equation}\label{eq:condderiv}
|\hat f_{\varepsilon}^{(l)}(t)|\leq d_{l} (1+|t|)^{-(\beta+l)},\quad t\in\mathbb R,
\end{equation}
with $d_{l}>0$. For any $d\geq 1$, given $\alpha,\,L,\,M>0$, let $\mathcal D_d:=\mathcal P_1(\mathbb R^d,\,M)\cap \mathcal S_d(\alpha,\,L)$ and
$$\psi_n:=n^{(\alpha+1)/[2\alpha+(2\beta d\vee 1)+1]}.$$
Then, there exists $C>0$ such that, for any estimator $\hat \mu_n$,
\begin{equation*}\label{eq:LB1}
\mathop{\underline{\lim}}_{n \to \infty}\psi_n
\sup_{\mu\in\mathcal D_d}\mathbb E^n_{(\mu\ast \mu_\varepsilon^{\otimes d})}W_1(\hat\mu_n,\,\mu)>C.
\end{equation*}
\end{thm}
\smallskip
The proof of Theorem \ref{thm:lower bound} is reported in Appendix \ref{sec:lower bounds}.
Note that, for $d=1$, $\mathcal D_1=\mathcal P_1(\mathbb R,\,M)$ and $0<\beta<\frac{1}{2}$, the lower bound rate $n^{-1/2}$ of Theorem \ref{thm:lower bound} improves upon the lower bound $n^{-1/(2\beta+1)}$ of Theorem 4.1 in \cite{dedecker2015}, p. 246. The sharper lower bound $n^{-1/2}$ matches with the upper bound for the minimum distance deconvolution estimator proposed by \cite{dedecker2015}, see Theorem 3.1, p. 243, thus showing that, for all $\beta>0$, it attains minimax-optimal rates, up to log-factors.

\begin{table}[ht]
\centering
\begin{tabular}{||c|c|c||}
\hline\\[2pt]
 {\bf Mixing distribution $\mu_{0X}$}& {\bf Dimension ${d=1}$} & {\bf Any dimension ${d\geq 1}$}  \\[3pt]
\hline\\
$\mbox {$\mu_{0X}\in \mathcal P_1(\mathbb R^d,\,M)$}$ 
  &  $n^{-1/(2\beta+1)}$ \mbox{[Dedecker \emph{et al.} (2015)]} & \\ [10pt]
  &  $\boldsymbol{n^{-1/[(2\beta\vee 1)+1]}}$ & $\boldsymbol{n^{-1/[(2\beta d\vee 1)+1]}}$\\[15pt]
\mbox{$\mu_{0X}\in\mathcal P_1(\mathbb R^d,\,M)\cap \mathcal S_d(\alpha,\,L)$} & $\boldsymbol{n^{-(\alpha+1)/[2\alpha+(2\beta\vee 1)+1]}}$ & $\boldsymbol{n^{-(\alpha+1)/[2\alpha+(2\beta d\vee 1)+1]}}$\\[15pt]
\hline
\end{tabular}
\vspace{0.3cm} 
\caption{In bold our lower bound rates on the $L^1$-Wasserstein risk for error distributions $\mu_\varepsilon^{\otimes d}$ with ordinary $\beta$-smooth single coordinate distribution and Sobolev $\alpha$-regular mixing densities.}
\label{tab:ssos}
 \end{table}  

When $d=1$, as a consequence of Corollary \ref{cor:postmean} and Theorems \ref{thm:32}, \ref{thm:5}, 
the Bayes' estimator, namely the posterior expected mixing distribution, attains minimax rates, up to log-factors, under the Laplace noise.
Then a natural question is whether the lower bound rates of Theorem \ref{thm:lower bound} can also be attained when $d>1$.
For $d\geq1$, \cite{caillerie2011} consider a modification of the standard deconvolution
kernel estimator and find slower rates than $n^{-1/[(2\beta d\vee 1)+1]}$ with respect to the
$L^2$-Wasserstein distance.
In Section \ref{sec:ke}, using the inversion inequality of Theorem \ref{theo:1}, 
we show that a frequentist estimator of $\mu_{0X}$ attains the lower bound $n^{-1/(2\beta d+1)}=n^{-1/(4d+1)}$, when $\beta>1$.

\subsection{Frequentist deconvolution estimator}\label{sec:ke}
In this section, we consider a frequentist estimator of $\mu_{0X}$ and, using the inversion inequality of Theorem \ref{theo:1}, we show that it achieves the minimax rate, up to a log-factor. For the sake of simplicity, we restrict to the case of $\alpha=0$ and a standard Laplace noise distribution, but the proof extends to any $\alpha>0$ and ordinary smooth noise distribution. 

Let $b_n = n^{-1/(2\beta d + 1)}$ and define $\tilde f_n$ as the inverse Fourier transform of
$\hat K_{b_n}^{\otimes d} \phi_n r_{\varepsilon}^{\otimes d}$, where $\phi_n (\mathsf t):= \mathbb P_n (e^{\imath \mathsf t\cdot\mathsf Y})$ is the empirical characteristic function and the kernel $K = \hat \tau$ is defined in Section \ref{sec:approxresult}. In symbols,
$$\tilde f_n(\mathsf x):=\frac{1}{(2\pi)^d}\int_{\mathbb R^d}e^{-\imath \mathsf t\cdot \mathsf x}\hat K_{b_n}^{\otimes d}(\mathsf t)\phi_n(\mathsf t) r_{\varepsilon}^{\otimes d}(\mathsf t)\,\di \mathsf t,\quad \mathsf x\in\mathbb R^d.$$
Since $\tilde f_n$ is not necessarily non-negative and 
$F_{\tilde\mu_n}$ is not necessarily a distribution function,
we define $\tilde \mu_{1n}$ to be the probability measure such that the corresponding distribution function
$F_{\tilde \mu_{1n}}$ is, up to a term of order $O(n^{-1/2})$, the closest one
to $F_{\tilde \mu_n}$ in the max-sliced $L^1$-distance, that is,
for every $\mu\in\mathcal P_1(\mathbb R^d)$, 
$$\sup_{\mathsf v \in \mathbb S^{d-1}} \|F_{\tilde \mu_{n,\mathsf v}}-F_{\tilde \mu_{1n,\mathsf v}}\|_1 \leq \sup_{\mathsf v \in \mathbb S^{d-1}} \|F_{\tilde \mu_{n,\mathsf v}}- F_{\mu_{\mathsf v}}\|_1+O(n^{-1/2}).$$
The idea of defining the estimator as an approximate minimizer over all distribution functions of the
$L^1$-metric is considered in \cite{dedecker2015} for $d=1$. 
Here, instead, we choose the estimator as an approximate minimizer over all distribution functions
of the max-sliced $L^1$-distance. We then have the following result whose proof is reported in Appendix \ref{sec:prthcomte}.
\begin{thm}\label{th:comtelacour}
Let $f_\varepsilon$ be the standard Laplace density. 
Assume that $f_{0X}$ has exponential tails, that is, there exists a constant $c_2>0$ such that
\begin{equation}\label{eq:exptails}
f_{0X} (\mathsf x) \lesssim e^{-c_2 |\mathsf x|}, \quad \text{for $|\mathsf x|$ large enough}.
\end{equation}
Then, for suitable $q>0$,
$$W_1(\tilde\mu_{1n},\,\mu_{0X}) = O_{\mathsf P}(n^{-1/(4d+1)}(\log n)^q).$$ 
\end{thm}

From the proof of Theorem \ref{th:comtelacour}, we see that the result extends straightforwardly to any noise distribution which satisfies Assumption \ref{eq:deriv} and such that 
 $$  f_\varepsilon(\varepsilon) \lesssim e^{ -c_2 |\varepsilon|}, $$
leading to a convergence rate of order $  O_{\mathsf P}(n^{-1/(\beta d+1)}(\log n)^q) $ as soon as $\beta >1/d$.

\section{Proofs}\label{sec:proofs}

We preliminarily recall an auxiliary result. 
For every $j\in \mathbb N$, let
$\hat f^{(j)}$ denote the $j$th derivative of the Fourier transform $\hat f$ of a function $f:\,\mathbb R\rightarrow \mathbb C$. 
If $\hat f^{(j)}\in L^1(\mathbb R)$, then
\begin{equation}\label{eq:identity21}
\mbox{for $z\neq0$}, \quad f(z)=\frac{1}{2\pi(\imath z)^j}\int_{\mathbb R}e^{-\imath tz}\hat f^{(j)}(t)\,\di t.
\end{equation}

\subsection{Proof of Theorem \ref{theo:1}} \label{sec:rth1}
Because $\mu_X,\,\mu_{0X}\in\mathcal P_1(\mathbb R^d)$ by assumption, we have $W_1(\mu_X,\,\mu_{0X})<\infty$,
see, \emph{e.g.}, \cite{villani2009}, p. 94. 
For $d\geq1$, the assumption 
$\mu_\varepsilon\in\mathcal P_{1}(\mathbb R)$ 
implies that $\mu_\varepsilon^{\otimes d}\in\mathcal P_1(\mathbb R^d)$ so that also $\mu_Y,\,\mu_{0Y}\in\mathcal P_1(\mathbb R^d)$ and $W_1(\mu_Y,\,\mu_{0Y})<\infty$.
From the strong equivalence, recalled in \eqref{eq:equivalence}, between the Wasserstein metric $W_1$ and the max-sliced Wasserstein metric $\overline W_1$, 
valid in any dimension $d\geq1$, we have that, for a constant $C_d\geq1$,
\[\frac{1}{C_d}
W_1(\mu_X,\,\mu_{0X})\leq
\overline W_1(\mu_X,\,\mu_{0X})=\max_{\mathsf v\in \mathbb S^{d-1}}
W_1(\mu_{X,\mathsf v},\,\mu_{0X,\mathsf v}).
\]
We now bound $W_1(\mu_{X,\mathsf v},\,\mu_{0X,\mathsf v})$. 
We first treat the case where only the condition 
$\mu_{0X}\in\mathcal P_1(\mathbb R^d)$ is required and then
the case where the smoothness Assumption \ref{ass:smoothXXX} holds.
 \vspace{0.2cm}

\noindent 
$\bullet$ \emph{Case 1: no smoothness assumption on $\mu_{0X}$}\\[4pt]
We consider a multivariate kernel with independent coordinates as in \eqref{eq:kmultiv}.
This assumption is not necessary, but simplifies the proof. The univariate kernel 
can be taken to be a symmetric probability density $K\in L^2(\mathbb R)\cap \mathcal P_{d\wedge2}(\mathbb R)$.
For $h>0$, let $K_h$ denote the rescaled kernel density and, with abuse of notation, let $K_h^{\otimes d}$ denote the corresponding $d$-fold product probability measure. 
For brevity, in what follows we also use the notation $K_{h,\mathsf v}:= (K_h^{\otimes d})_{\mathsf v}$ to denote the distribution of $\mathsf v\cdot \mathsf Z$ when $ \mathsf Z \sim K_h^{\otimes d}$.
By the triangle inequality for Wasserstein metrics,
\begin{eqnarray}\label{eq:triangneq}
W_1(\mu_{X,\mathsf v},\,\mu_{0X,\mathsf v}) &\leq& W_1(\mu_{X,\mathsf v},\,\mu_{X,\mathsf v} \ast K_{h,\mathsf v})+ W_1(\mu_{X,\mathsf v}\ast K_{h,\mathsf v},\,\mu_{0X,\mathsf v}\ast  K_{h,\mathsf v})\nonumber\\
&&\qquad\qquad\qquad\qquad\qquad\qquad\qquad\qquad\quad+ W_1(\mu_{0X,\mathsf v}\ast K_{h,\mathsf v},\,\mu_{0X,\mathsf v}).
\end{eqnarray}
 Also, 
$\mathsf v\cdot (\mathsf X +\mathsf Z)=(\mathsf v\cdot \mathsf X +\mathsf v\cdot \mathsf Z)\sim \mu_{X,\mathsf v} \ast K_{h,\mathsf v}$ and $W_1(\mu_{X,\mathsf v},\,\mu_{X,\mathsf v} \ast K_{h,\mathsf v})\leq \mathbb E[|\mathsf v\cdot \mathsf Z|]$. 
For $d=1$, we have $\mathbb E[|Z|]=h\int_{\mathbb R}|z|K(z)\,\di z<\infty$, while, for $d\geq 2$, we have
$\mathbb E[|\mathsf v\cdot \mathsf Z|]\leq(\mathbb E[|\mathsf Z|^2])^{1/2}=h(\int_{\mathbb R^d}|\mathsf z|^2K(\mathsf z)\,\di \mathsf z)^{1/2}<\infty$
as soon as $\int_{\mathbb R}z^2K(z)\,\di z<\infty$.
Thus, $W_1(\mu_{X,\mathsf v},\,\mu_{X,\mathsf v} \ast K_{h,\mathsf v})\lesssim h$ uniformly in $\mathsf v$. 
Analogously, letting $\mathsf X_0$ be distributed according to $\mu_{0X}$ and independent of $\mathsf Z$, we have
$W_1(\mu_{0X,\mathsf v},\,\mu_{0X,\mathsf v}\ast K_{h,\mathsf v})\lesssim h$.  
Also, $W_1(\mu_{X,\mathsf v}\ast K_{h,\mathsf v},\,\mu_{0X,\mathsf v}\ast  K_{h,\mathsf v}) \leq  \mathbb E[|\mathsf X-\mathsf X_0|]\leq M_1(\mu_X)+ M_1(\mu_{0X})<\infty$
because $\mu_X,\,\mu_{0X}\in\mathcal P_1(\mathbb R^d)$. 
Thus,
\begin{equation}\label{eq:distriang}
W_1(\mu_{X,\mathsf v},\,\mu_{0X,\mathsf v}) 
\lesssim h + W_1(\mu_{X,\mathsf v}\ast K_{h,\mathsf v},\,\mu_{0X,\mathsf v}\ast  K_{h,\mathsf v}).
\end{equation}
We derive an upper bound on $W_1(\mu_{X,\mathsf v}\ast K_{h,\mathsf v},\,\mu_{0X,\mathsf v}\ast  K_{h,\mathsf v})$.
\medskip

\noindent\emph{Control of the term $W_1(\mu_{X,\mathsf v}\ast K_{h,\mathsf v},\,\mu_{0X,\mathsf v}\ast  K_{h,\mathsf v})$}\\[-8pt]

\noindent
Taking into account that 
\begin{equation}\label{eq:rel1}
\hat\mu_{\mathsf v}(t)=\int_{\mathbb R}e^{\imath t x}\mu_{\mathsf v}(\di x)=\int_{\mathbb R^d}e^{\imath t\mathsf v\cdot\mathsf x}\mu(\di \mathsf x)=\hat\mu(t\mathsf v),\quad t\in\mathbb R,
\end{equation}
and using the representation of $W_1$, when $d=1$, as the 
$L^1$-distance between distribution functions, for all $\mu,\,\nu\in\mathcal P_1(\mathbb R^d)$ we have
\begin{eqnarray}\label{eq:identity}
W_1(\mu_{\mathsf v},\,\nu_{\mathsf v})=\|F_{\mu_{\mathsf v}}-F_{\nu_{\mathsf v}}\|_1 \nonumber &=&\frac{1}{2\pi}
\int_{\mathbb R}\abs{\int_{\mathbb R}e^{-\imath t x}\frac{\hat\mu_{\mathsf v}(t)-\hat\nu_{\mathsf v}(t)}{(-\imath t)}\,\di t}\,\di x \nonumber\\
&=&\frac{1}{2\pi}\int_{\mathbb R}\abs{\int_{\mathbb R}e^{-\imath t x}\frac{\hat\mu(t\mathsf v)-\hat\nu(t\mathsf v)}{(-\imath t)}\,\di t}\,\di x.
\end{eqnarray}
We introduce some more notation. Let $\chi:\,\mathbb R\rightarrow \mathbb R$ be a symmetric, continuously differentiable function,
equal to $1$ on $[-1,\,1]$ and to $0$ outside $[-2,\,2]$. 
For example, one such function could be $\chi(t)=e\exp{\{-1/[1-(|t|-1)^2]\}}$, $|t|\in(1,\,2)$. 
For the construction of smooth bump functions, see, \emph{e.g.}, \cite{FRY2002143}.
Define
$$w_{1,h}(t):=\hat K(ht)\chi(t)r_{\varepsilon}(t)\quad\mbox{ and }\quad w_{2,h}(t):=\hat K(ht)[1-\chi(t)]r_{\varepsilon}(t), \quad t\in\mathbb R.$$
Note that $K\in L^1(\mathbb R)$ implies that $\hat K$
is well-defined and $\|\hat K\|_\infty:=\sup_{t\in\mathbb R}|\hat K(t)|\leq \|K\|_1<\infty$. 
We consider a kernel with $\hat K$ supported on $[-1,\,1]$.
Since $\hat K$ is continuous and bounded on a compact, we have $\hat K\in L^1(\mathbb{R})$ and 
$K(\cdot)=(2\pi)^{-1}\int_{\mathbb R}e^{-\imath t \cdot}\hat K(t)\,\di t$. 
If $h<\frac{1}{2}$, the function $w_{1,h}$ is equal to $0$ outside $[-2,\,2]$, while $ w_{2,h}$ is equal to $0$ on $[-1,\,1]$ and
outside $[-1/ h,\,1/h]$. Thus, $w_{j,h}\in L^1(\mathbb R)$, for $j\in[2]$.
In fact, by the inequality \eqref{eq:deriv} with $l=0$, 
we have $
\|w_{1,h}\|_1
\lesssim
\int_{|t|\leq2}|\hat K(ht)||\chi(t)| (1+|t|)^\beta\,\di t <\infty$ because 
the integrand is in $C_b([-2,\,2])$.
Analogously, $\|w_{2,h}\|_1\lesssim 
\int_{1<|t|\leq1/h}|\hat K(ht)||1-\chi(t)| (1+|t|)^\beta\,\di t
<\infty$. Then, the inverse Fourier transform of $w_{j,h}$,
\begin{equation*}\label{eq:45}
z\mapsto K_{j,h}(z):=\frac{1}{2\pi}\int_{\mathbb R}
e^{-\imath tz}w_{j,h}(t)\,\di t
\end{equation*}
is well defined for $j\in[2]$. 
For $\mathsf v\in\mathbb S^{d-1}$, let $J_d^\ast(\mathsf v):=\{j\in [d]:\,v_j\neq0\}$ be the set of indices 
corresponding to non-null coordinates of $\mathsf v$. 
We denote by $|J_d^\ast(\mathsf v)|$ the cardinality of $J_d^\ast(\mathsf v)$.
Clearly, $J_d^\ast(\mathsf v)\neq\emptyset$ because $|\mathsf v|=1$. 
For later use, we note that
$$\reallywidehat{K_{h,\mathsf v}}(t)=\reallywidehat{(K_h^{\otimes d})}_{\mathsf v}(t)=\reallywidehat{(K_h^{\otimes d})}(t\mathsf v)=
\prod_{j=1}^d\hat K(v_jht)=\hat K^{\otimes d}(ht\mathsf v),\quad t\in\mathbb R.$$
By the inequality on the right-hand side of \eqref{eq:equivalence}, we have $W_1(\mu_X\ast K_h^{\otimes d},\,\mu_{0X}\ast K_h^{\otimes d})\leq
 C_d \overline{W}_1(\mu_X\ast K_h^{\otimes d},\,\mu_{0X}\ast K_h^{\otimes d})$, where, using the expression of 
$W_1(\mu_{\mathsf v},\,\nu_{\mathsf v})$ in \eqref{eq:identity}, we have 
\[\begin{split}
&W_1(\mu_{X,\mathsf v}\ast K_{h,\mathsf v},\,\mu_{0X,\mathsf v}\ast K_{h,\mathsf v}) \\
&\quad\qquad=\frac{1}{2\pi} \int_{\mathbb R}\abs{\int_{\mathbb R}e^{-\imath t x}\reallywidehat{(K_h^{\otimes d})}(t\mathsf v)\frac{\hat\mu_X(t\mathsf v)-\hat\mu_{0X}(t\mathsf v)}{(-\imath t)}\,\di t}\,\di x \\
&\quad\qquad=\frac{1}{2\pi}\int_{\mathbb R}\abs{\int_{\mathbb R}e^{-\imath t x}
\hat K^{\otimes d}(ht\mathsf v)r^{\otimes d}_{\varepsilon}(t\mathsf v)\frac{\hat\mu_Y(t\mathsf v)-\hat\mu_{0Y}(t\mathsf v)}{(-\imath t)}\,\di t}\,\di x \\
&\quad\qquad\leq \frac{1}{2\pi}\int_{\mathbb R}\abs{\int_{\mathbb R}e^{-\imath t x}\hat K^{\otimes d}(ht\mathsf v)r^{\otimes d}_{\varepsilon}(t\mathsf v)\chi^{\otimes d}(t\mathsf v)\frac{\hat\mu_Y(t\mathsf v)-\hat\mu_{0Y}(t\mathsf v)}{(-\imath t)}\,\di t}\,\di x \\
&\qquad\qquad\qquad + \frac{1}{2\pi}\int_{\mathbb R}\abs{\int_{\mathbb R}e^{-\imath t x}\hat K^{\otimes d}(ht\mathsf v)r^{\otimes d}_{\varepsilon}(t\mathsf v)[1-\chi^{\otimes d}(t\mathsf v)]\frac{\hat\mu_Y(t\mathsf v)-\hat\mu_{0Y}(t\mathsf v)}{(-\imath t)}\,\di t}\,\di x \\
&\quad\qquad=:T_1+T_2,
\end{split}
\]
for
$$\hat K^{\otimes d}(ht\mathsf v)r^{\otimes d}_{\varepsilon}(t\mathsf v)\chi^{\otimes d}(t\mathsf v)=\prod_{j=1}^d \hat K(v_jht)r_\varepsilon(v_j t)\chi(v_jt)=\prod_{j=1}^d w_{1,h}(v_jt)=\prod_{j\in J_d^\ast(\mathsf v)} w_{1,h}(v_jt)$$
because $w_{1,h}(v_jt)=1$ if $v_j=0$. Noting that the inverse Fourier transform of 
$\prod_{j\in J_d^\ast(\mathsf v)} w_{1,h}(v_jt)$
is $\Conv_{j\in J_d^\ast(\mathsf v)}[(1/v_j)K_{1,h}(\cdot/v_j)]$, we have
\[\begin{split}
2\pi T_1&\leq 
\int_{\mathbb R}\abs{\int_{\mathbb R}e^{-\imath t x}\hat K^{\otimes d}(ht\mathsf v)r^{\otimes d}_{\varepsilon}(t\mathsf v)\chi^{\otimes d}(t\mathsf v)\,\di t}\,\di x\times
\int_{\mathbb R}\abs{\int_{\mathbb R}e^{-\imath t x}\frac{\hat\mu_Y(t\mathsf v)-\hat\mu_{0Y}(t\mathsf v)}{(-\imath t)}\,\di t}\,\di x
\\
&=2\pi W_1(\mu_{Y,\mathsf v},\,\mu_{0Y,\mathsf v}) 
\int_{\mathbb R}\bigg|{\int_{\mathbb R}e^{-\imath t x}\prod_{j\in J_d^\ast(\mathsf v)} w_{1,h}(v_jt)\,\di t}\bigg|\,\di x
\\
&= (2\pi)^2 W_1(\mu_{Y,\mathsf v},\,\mu_{0Y,\mathsf v})
\int_{\mathbb R}\bigg|\bigg(\Conv_{j\in J_d^\ast(\mathsf v)}\pq{\frac{1}{v_j}K_{1,h}(\cdot/v_j)}\bigg)(x)\bigg|\,\di x \\
&\leq (2\pi)^2 W_1(\mu_{Y,\mathsf v},\,\mu_{0Y,\mathsf v})\prod_{j\in J_d^\ast(\mathsf v)}
\norm{\frac{1}{v_j}K_{1,h}(\cdot/v_j)}_1 = (2\pi)^2 W_1(\mu_{Y,\mathsf v},\,\mu_{0Y,\mathsf v})\norm{K_{1,h}}_1^{|J_d^\ast(\mathsf v)|},
\end{split}
\]
where $\|K_{1,h}\|_1=O(1)$ by Lemma \ref{lem:bigO} and 
$\max_{\mathsf v\in\mathbb S^{d-1}}\norm{K_{1,h}}_1^{|J_d^\ast(\mathsf v)|}=\max_{j\in[d]}\norm{K_{1,h}}_1^{j}<\infty$. Thus, $T_1\lesssim  W_1(\mu_{Y,\mathsf v},\,\mu_{0Y,\mathsf v})$.
Concerning the term $T_2$, set the position
$$w_{2,h,\mathsf v}(t):=\hat K^{\otimes d}(ht\mathsf v)r^{\otimes d}_{\varepsilon}(t\mathsf v)[1-\chi^{\otimes d}(t\mathsf v)]=
\bigg[1-\prod_{j=1}^d\chi(v_jt)\bigg]\prod_{k=1}^d\hat K(v_kht)r_\varepsilon(v_kt), \quad t\in\mathbb R,$$
by Lemma \ref{lem:F3}, recalling that $I_h^\ast(\mathsf v) = \{ j\in[d] :\, |v_j| >h\}$, we have
\[
\begin{split}
2\pi T_2&\leq
\int_{\mathbb R}\abs{\int_{\mathbb R}e^{-\imath t x}\frac{w_{2,h,\mathsf v}(t)}{(-\imath t)}\,\di t}\,\di x  \times
\int_{\mathbb R}\abs{\int_{\mathbb R}e^{-\imath t x}[\hat\mu_{Y,\mathsf v}(t)-\hat\mu_{0Y,\mathsf v}(t)]\,\di t}\,\di x\\
&
\lesssim |\log h|\left(  |\log h|\1_{(\beta|I_h^\ast(\mathsf v)|\leq 1)}+ h^{-\beta |I_h^\ast(\mathsf v)|+1} 
\prod_{j\in I_h^\ast(\mathsf v)} |v_j|^\beta \1_{(\beta |I_h^\ast(\mathsf v)|> 1)}\right)\|f_{Y,\mathsf v}-f_{0Y,\mathsf v}\|_1,
\end{split}\]
where $f_{Y,\mathsf v}$ and $f_{0Y,\mathsf v}$ are the densities of the measures $\mu_{Y,\mathsf v}$ and 
$\mu_{0Y,\mathsf v}$, respectively. 
Note also that, for every $\mathsf v\in\mathbb S^{d-1}$,
\[\begin{split}
\frac{1}{2}\|f_{Y,\mathsf v}-f_{0Y,\mathsf v}\|_1&=\sup_{A\in\mathscr B(\mathbb R)}|\mu_{Y,\mathsf v}(A)-\mu_{0Y,\mathsf v}(A)|\\
&\leq \sup_{B\in\mathscr B(\mathbb R^d)}|P_Y(\mathsf Y\in B)-P_{0Y}(\mathsf Y\in B)|=\frac{1}{2}\|f_Y-f_{0Y}\|_1.
\end{split}
\]
It follows that $\max_{\mathsf v\in \mathbb S^{d-1}}\|f_{Y,\mathsf v}-f_{0Y,\mathsf v}\|_1\leq \|f_Y-f_{0Y}\|_1$
and 
$$T_2\lesssim  |\log h| \left(  |\log h|\1_{(\beta|I_h^\ast(\mathsf v)|\leq 1)}+ h^{-\beta |I_h^\ast(\mathsf v)|+1} 
\prod_{j\in I_h^\ast(\mathsf v)} |v_j|^\beta \1_{(\beta |I_h^\ast(\mathsf v)|> 1)}\right)\|f_Y-f_{0Y}\|_1.$$
Combining the bounds on $T_1$ and $T_2$, we obtain the bound on $T$ reported in \eqref{eq:t2},
which, together with \eqref{eq:distriang}, proves the inversion inequality.

\smallskip

\noindent
$\bullet$ \emph{Case 2: smoothness Assumption \ref{ass:smoothXXX} on $\mu_{0X}$ is in force}\\[4pt]
If Assumption \ref{ass:smoothXXX} holds true, then $K\in L^1(\mathbb R)\cap L^2(\mathbb R)$ 
is taken to be a superkernel with $zK(z)\in L^1(\mathbb R)$
and $\int_{\mathbb R}z^2|K(z)|\,\di z<\infty$ when $d\geq2$. 
Since $\hat K\equiv 1$ on $[-1,\,1]$, while $\hat K\equiv 0$ on $[-2,\,2]^c$, 
by taking $\hat K(\cdot/2)$ the support reduces to $[-1,\,1]$. 
Note that, as $K$ need not be a probability density, the triangular inequality for the Wasserstein metric in \eqref{eq:triangneq} 
does not necessarily hold. Nevertheless, by the inequality on the right-hand side of \eqref{eq:equivalence} and
the representation of $W_1$, when $d=1$, as the $L^1$-distance between distribution functions, 
we have
\[\frac{1}{C_d}
W_1(\mu_X,\,\mu_{0X})\leq
\overline W_1(\mu_X,\,\mu_{0X})=\max_{\mathsf v\in \mathbb S^{d-1}}
W_1(\mu_{X,\mathsf v},\,\mu_{0X,\mathsf v})=\max_{\mathsf v\in \mathbb S^{d-1}}
\|F_{X,\mathsf v}-F_{0X,\mathsf v}\|_1.
\]
Then, by the triangular inequality for the $L^1$-norm distance,
\[
\begin{split}
\frac{1}{C_d}
W_1(\mu_X,\,\mu_{0X})&\leq\max_{\mathsf v\in \mathbb S^{d-1}}
\|F_{X,\mathsf v}-F_{0X,\mathsf v}\|_1\\
&\leq \max_{\mathsf v\in \mathbb S^{d-1}}
\|F_{X,\mathsf v}-(F_{X}\ast K_{\tilde h}^{\otimes d})_{\mathsf v}\|_1  +\max_{\mathsf v\in \mathbb S^{d-1}}
\|((F_{X}-F_{0X})\ast K_{\tilde h}^{\otimes d})_{\mathsf v}\|_1\\
&\qquad\qquad\qquad\qquad\qquad\qquad\quad\,\,\,\,\, +\max_{\mathsf v\in \mathbb S^{d-1}}
\|(F_{0X}\ast K_{\tilde h}^{\otimes d})_{\mathsf v}-F_{0X,\mathsf v}\|_1\\
&=:D_1+D_2+D_3,
\end{split}
\]
where $\tilde h:=h/2$. Note that, for $\mathsf v\in\mathbb S^{d-1}$, we have
$(F_{X}\ast K_{\tilde h}^{\otimes d})_{\mathsf v}=F_{X,\mathsf v}\ast (K_{\tilde h}^{\otimes d})_{\mathsf v}$ so that
$$b_{F_{X,\mathsf v}}(\tilde h) :=F_{X,\mathsf v}-
F_{X,\mathsf v}\ast (K_{\tilde h}^{\otimes d})_{\mathsf v}=F_{X,\mathsf v}-(F_{X}\ast K_{\tilde h}^{\otimes d})_{\mathsf v}.$$ Therefore, by condition \eqref{eq:ass1},
we have $D_1=\max_{\mathsf v\in\mathbb S^{d-1}}\|b_{F_{X,\mathsf v}}(\tilde h)\|_1=O(h^{\alpha+1})$. 
The term $D_2$ can be bounded using the same arguments 
as for $W_1(\mu_X\ast K_h^{\otimes d},\,\mu_{0X}\ast K_h^{\otimes d})$ in Case $1$, 
therefore
\[\begin{split}
D_2 &\lesssim W_1(\mu_Y,\,\mu_{0Y})\\
&\quad\,\, + |\log h| \left(  |\log h|\1_{(\beta|I_h^\ast(\mathsf v)|\leq 1)}+ h^{-\beta |I_h^\ast(\mathsf v)|+1} 
\prod_{j\in I_h^\ast(\mathsf v)} |v_j|^\beta \1_{(\beta |I_h^\ast(\mathsf v)|> 1)}\right)\|f_Y-f_{0Y}\|_1.
\end{split}\]
By the same arguments laid down for $D_1$, the term $D_3=\max_{\mathsf v\in \mathbb S^{d-1}}
\|b_{F_{0X,\mathsf v}}(\tilde h)\|_1$. We show that
\begin{equation}\label{eq:biasgeneral12}
D_3=O(h^{\alpha+1}).
\end{equation}
We make two preliminary remarks.
First, for $\mathsf v\in\mathbb S^{d-1}$, by \eqref{eq:rel1}, we have 
$\hat \mu_{\mathsf v}(t)=\hat\mu(t\mathsf v)$, $t\in\mathbb R$. Then, Assumption \ref{ass:smoothXXX} implies that
\begin{equation}\label{eq:supremum1}
\max_{\mathsf v\in \mathbb S^{d-1}}\|\reallywidehat{D^\alpha f_{0X,\mathsf v}}\|_1=
\max_{\mathsf v\in \mathbb S^{d-1}}\int_{\mathbb R}|t|^\alpha|\hat \mu_{0X,\mathsf v}(t)|\,\di t=\max_{\mathsf v\in \mathbb S^{d-1}}
\int_{\mathbb R}|t|^\alpha|\hat \mu_{0X}(t\mathsf v)|\,\di t<\infty.
\end{equation}
Second, note that $|1-\hat K^{\otimes d}(\tilde ht \mathsf v)|\neq 0$ for all those $t\in\mathbb R$
for which there exists at least an index $j\in J_d^\ast(\mathsf v)$ so that $|v_j\tilde ht|>1$. We define the set 
$$\mathscr D:=\{t\in\mathbb R:\,\exists\,j\in J_d^\ast(\mathsf v) \mbox{ so that }|v_j\tilde ht|>1\}.$$
The domain $\mathscr D$ depends on $h$ and $\mathsf v$, \emph{i.e.}, $\mathscr D\equiv \mathscr D_{h,\mathsf v}$, but we shall not emphasize this
dependence in what follows and simply write $\mathscr D$. Note that $
\mathscr D\subseteq
\{t\in\mathbb R:\,|t|>(\tilde h\|\mathsf v\|_\infty)^{-1}\}$,
where $\|\mathsf v\|_\infty:=\max_{j\in [d]} |v_j|\leq 1$.
By the same arguments used for the function $G_{2,h}$ in \cite{dedecker2015}, pp. 251--252, we have
\[\begin{split}
\|b_{F_{0X,\mathsf v}}(\tilde h)\|_1&=
\int_{\mathbb R}\bigg|\frac{1}{2\pi}\int_{\mathscr D}
e^{-\imath t x}\frac{[1-\hat K^{\otimes d}(\tilde ht\mathsf v)]}{(-\imath t)}\hat \mu_{0X,\mathsf v}(t)\,\di t
\bigg|\,\di x
\end{split}\]
because $t\mapsto[1-\hat K^{\otimes d}(\tilde ht\mathsf v)][\hat \mu_{0X,\mathsf v}(t)/t]$ is in $L^1(\mathbb R)$ due to \eqref{eq:supremum1}. 
To prove relationship \eqref{eq:biasgeneral12}, we write
 \[\begin{split}
\|b_{F_{0X,\mathsf v}}(\tilde h)\|_1&=\int_{\mathbb R}\bigg|\frac{1}{2\pi}\int_{\mathscr D}
e^{-\imath t x}\underbrace{(-\imath t)^{\alpha}\hat \mu_{0X,\mathsf v}(t)}_
{\reallywidehat{D^{\alpha}\hspace*{-1pt}f_{0X,\mathsf v}}(t)}\frac{[1-\hat K^{\otimes d}(\tilde ht\mathsf v)]}{(-\imath t)^{\alpha+1}}\,\di t
\bigg|\,\di x\\
&\leq \|D^{\alpha}\hspace*{-1pt}f_{0X,\mathsf v}\|_1 \times 
\bigg(\underbrace{\int_{|x|\leq h}}_{=:B_{1,\mathsf v}}+\underbrace{\int_{|x|>h}}_{=:B_{2,\mathsf v}}\bigg)\bigg|\frac{1}{2\pi}\int_{\mathscr D}
e^{-\imath t x}\frac{[1-\hat K^{\otimes d}(\tilde ht\mathsf v)]}{(-\imath t)^{\alpha+1}}\,\di t
\bigg|\,\di x,
\end{split}
\]
where $\|D^{\alpha}\hspace*{-1pt}f_{0X,\mathsf v}\|_1<\infty$ by Assumption \ref{ass:smoothXXX}.
Now, 
$$B_{1,\mathsf v}\lesssim h\int_{\mathscr D}
\frac{[1+|\hat K^{\otimes d}(\tilde ht \mathsf v)|]}{|t|^{\alpha+1}}\,\di t\lesssim 
h\int_{|t|>(\tilde h\|\mathsf v\|_\infty)^{-1}}
\frac{[1+|\hat K^{\otimes d}(\tilde ht \mathsf v)|]}{|t|^{\alpha+1}}\,\di t\lesssim
h^{\alpha+1}$$ 
because 
$\|\hat K\|_\infty<\infty$ and the bound is uniform over $\mathbb S^{d-1}$. Thus,
$\max_{\mathsf v\in\mathbb S^{d-1}}B_{1,\mathsf v}=O(h^{\alpha+1})$.
To bound $B_{2,\mathsf v}$, we use identity \eqref{eq:identity21}.
The conditions $K\in L^1(\mathbb R)$ and $zK(z)\in L^1(\mathbb R)$ jointly imply that $\hat K$ is 
continuously differentiable with $|\hat K^{(1)}(t)|\rightarrow 0$ as $|t|\rightarrow \infty$. 
Indeed, $\hat K^{(1)}(\cdot/2)\in C_b([-1,\,1])$. 
Define
$\hat f_{\mathsf v}(t):=[1-\hat K^{\otimes d}(\tilde ht\mathsf v)](-\imath t)^{-(\alpha+1)}$, $t\in\mathbb R$.
Taking into account that
\[\begin{split}
\frac{\di}{\di t}\bigg(\frac{[1-\hat K^{\otimes d}(\tilde ht\mathsf v)]}{t^{\alpha+1}}\bigg)=-\frac{\tilde h(\hat K^{\otimes d})^{(1)}(\tilde ht\mathsf v)}{t^{\alpha+1}}
-(\alpha+1)\frac{[1-\hat K^{\otimes d}(\tilde ht\mathsf v)]}{t^{\alpha+2}}
\end{split}\]
and using the bound in \eqref{eq:unifbound}, we have
\[\begin{split}
\|\hat f_{\mathsf v}^{(1)}\|_2^2 &=\int_{\mathscr D}\bigg|\frac{\di }{\di t}\bigg(\frac{[1-\hat K^{\otimes d}(\tilde ht\mathsf v)]}{t^{\alpha+1}}\bigg)\bigg|^2
\,\di t\\
&\lesssim \int_{|t|>(\tilde h\|\mathsf v\|_\infty)^{-1}}\pt{h^2\frac{|(\hat K^{\otimes d})^{(1)}(\tilde ht\mathsf v)|^2}{|t|^{2(\alpha+1)}}+
\frac{|1-\hat K^{\otimes d}(\tilde ht\mathsf v)|^2}{|t|^{2(\alpha+2)}}}\,\di t \lesssim  h^{2(\alpha+3/2)}
\end{split}\]
and the bound is uniform over $\mathbb S^{d-1}$ so that
$\max_{\mathsf v\in\mathbb S^{d-1}}\|\hat f_{\mathsf v}^{(1)}\|_2=O(h^{\alpha+3/2})$.
\noindent
For $f_{\mathsf v}(\cdot):=(2\pi)^{-1}\int_{\mathscr D}e^{-\imath t \cdot}\hat f_{\mathsf v}(t)\,\di t$, which is well defined because
$\hat f_{\mathsf v}\in L^1(\mathbb R)$, by identity \eqref{eq:identity21} and the Cauchy–Schwarz inequality, we have that
\[
\begin{split}
B_{2,\mathsf v}:=  \int_{|x|>h}|f_{\mathsf v}(x)|\,\di x &= \int_{|x|>h}\frac{1}{|x|}\abs{\frac{1}{2\pi}\int_{\mathscr D}
e^{-\imath t x}\hat f_{\mathsf v}^{(1)}(t)\,\di t
}\di x\\ 
&\lesssim
\pt{ \int_{\mathbb R}\frac{1}{x^2}\1_{(|x|>h)}\,\di x}^{1/2}
\|\hat f_{\mathsf v}^{(1)}\|_2\lesssim h^{-1/2}h^{\alpha+3/2}\lesssim  h^{\alpha+1}
\end{split}
\]
uniformly over $\mathbb S^{d-1}$. 
Thus, $\max_{\mathsf v\in\mathbb S^{d-1}}B_{2,\mathsf v}=O( h^{\alpha+1})$. Consequently, $D_3=O( h^{\alpha+1})$ and the proof is complete.
\qed

\subsection{Proof of Theorem \ref{thm:22}}\label{sec:prth22}
By the conditions in \eqref{con1}, Theorem 2.1 of \cite{ghosal2000}, p. 503, implies that, for sufficiently large $\,\xbar{M}>0$,
$$\mathbb E_{0Y}^n[\Pi_n(\mu_X:\,\|f_Y-f_{0Y}\|_1>\xbar{M}\,\tilde\epsilon_n \mid \mathsf Y^{(n)})]\rightarrow 0.$$
Since $\mu_{0X}\in\mathcal P_{4+\delta}(\mathbb R^d)$ and
$\mu_\varepsilon\in\mathcal P_{4+\delta}(\mathbb R)$, we have 
$M_{4+\delta}(\mu_{0Y}) < \infty$.  Also, since 
$$\mathbb E_{0Y}^n[\Pi_n(\mu_X:\,M_{4+\delta}(\mu_Y) > K'\tilde \epsilon_n^{-2}\mid\mathsf Y^{(n)})]\rightarrow0,$$
for $M>0$ and $\mathscr S_n :=\{\mu_X:\,W_1(\mu_Y,\,\mu_{0Y})\leq M \tilde \epsilon_n\log(1/\tilde\epsilon_n)\}$, by Theorem \ref{lem:1}
we have that
$\mathbb E_{0Y}^n[\Pi_n(\mathscr S_n^c\mid\mathsf Y^{(n)})] \rightarrow 0$.

The case where Assumption \ref{ass:smoothXXX} is in force is treated in details. 
By the bound in \eqref{eq:ass1}, Theorem \ref{theo:1} implies that, uniformly over $\mathscr P_n \cap \mathscr S_n$, 
$$W_1(\mu_X,\,\mu_{0X})\lesssim h_n^{\alpha+1}+\tilde \epsilon_n(\log n)+h_n^{-(\beta d-1)_+}(\log n)^{1+\1_{(\beta d\leq 1)}}\|f_Y-f_{0Y}\|_1.$$
Replacing $h_n$ with $[\tilde\epsilon_n(\log n)^{1+\1_{(\beta d\leq 1)}}]^{1/[\alpha+(\beta d\vee 1)]}$ and $\|f_Y-f_{0Y}\|_1$ with $\tilde \epsilon_n$ leads to 
$$W_1(\mu_X,\,\mu_{0X})\lesssim [\tilde\epsilon_n(\log n)^{1+\1_{(\beta d\leq 1)}}]^{(\alpha+1)/[\alpha+(\beta d\vee 1)]}.$$
There thus exists $C_\alpha>0$ such that
$$
\mathbb E_{0Y}^n[\Pi_n(\mu_X:\,W_1(\mu_X,\,\mu_{0X})>C_\alpha\epsilon_{n,\alpha}\mid \mathsf Y^{(n)})]
\rightarrow 0.
$$

The case where no regularity assumption on $\mu_{0X}$ is considered, except for the first moment condition $M_1(\mu_{0X})< \infty$, 
follows similarly from the inversion inequality of Theorem \ref{theo:1}, with $h_n$ in place of $h_n^{\alpha+1}$, choosing 
$h_n =[ \tilde \epsilon_n (\log n)^{1+\1_{(\beta d\leq 1)}}]^{1/(\beta d\vee 1)}$. \qed

\smallskip

\subsection{Proof of Lemma \ref{lem:contapprox}} \label{sec:rth3}
We begin by obtaining an equivalent expression for the $L^2$-norm in \eqref{bound:tildehm12}.
Denoting by $\mathcal F$ the Fourier transform operator, for any $f\in L^1(\mathbb R)$ we have $\mathcal F\{f\}:=\hat f$.
Recall that, given $f_\varepsilon(u)=e^{-|u|}/2$, $u\in\mathbb R$, for $b=\mp\frac{1}{2}$ we have $\mathcal F\{e^{b\cdot}f_\varepsilon\}(t)=[1/\varrho_b(t)]$, where
$\varrho_b(t):=[1-\psi_b^2(t)]$ and $\psi_b(t):=-(\imath t+b)$, $t\in\mathbb R$. Note that, as a consequence of the identity in \eqref{eq:ident}, we have that
$$\frac{1}{M_{0X}(b)}\mathcal F\{e^{b\cdot}(T_{m,b,\sigma}f_{0X})\}=\mathcal F\{\bar h_{0,b}\}+\gamma \mathcal F\{ h_{m,b,\sigma}\},$$
where $M_{0X}(b)<\infty$ by the assumption that $(e^{|\cdot|/2}f_{0X})\in L^1(\mathbb R)$. Then,
\begin{eqnarray}\label{eq:equivalent}
\Delta_0&:=&\sum_{b=\mp 1/2}\|e^{b\cdot}\{f_\varepsilon\ast [\phi_\sigma\ast (T_{m,b,\sigma}f_{0X})-f_{0X}]\}\|_2^2\nonumber\\
&=&
\sum_{b=\mp 1/2}M^2_{0X}(b)\|(e^{b\cdot}f_\varepsilon)\ast
[(e^{b\cdot}\phi_\sigma)\ast \{[M_{0X}(b)]^{-1}e^{b\cdot}(T_{m,b,\sigma}f_{0X})\}-\bar h_{0,b}]\|_2^2\nonumber\\
&=&
\frac{1}{2\pi}
\sum_{b=\mp 1/2}M^2_{0X}(b)\Big\|\frac{e^{\sigma^2 \psi_b^2/2}}{\varrho_b}
[(1-e^{-\sigma^2 \psi_b^2/2})\mathcal F\{\bar h_{0,b}\}+\gamma\mathcal F\{  h_{m,b,\sigma}\}]\Big\|_2^2.
\end{eqnarray}
Some facts are highlighted for later use. 
For every $\delta>0$, the function $\mathcal F\{H\}(\delta \cdot)$ is well defined because 
$\|H\|_1=(2\pi)^{-1}\|\hat \tau \reallywidehat{\phi_h}\|_1 <\infty$. Besides,
as $0\leq \tau\leq 1$,
\begin{equation}\label{eq:bilaplace1}
|\mathcal F\{H\}(\delta t)|=
|(\tau \ast \phi_{h})(-\delta t)|  \leq \|\phi_{h}(-\delta t-\cdot)\|_1=\|\phi_{-\delta t,h}\|_1=1,\quad t\in\mathbb R.
\end{equation}
Let $Z$ be a standard normal random variable. 
For constants $0<c_\delta,\,c_h<1$, take $\delta:=c_\delta\sigma$ and $h:=c_h|\log \sigma|^{-1/2}$.
Fix $u_0$ such that $0<c_\delta<u_0<1$. Then, for $\omega>0$ and 
$c_h$ such that $(1-u_0)\geq c_h\sqrt{2\omega}$, we have, for every $|t|\leq (u_0/\delta)$,
\begin{align}\label{eq:complement1}
|1-\mathcal F\{H\}(\delta t)|
\leq2\int_{|u|\geq1}\phi_{-\delta t,h}(u)\,\di u  &\leq2P(|Z|\geq (1-\delta |t|)/h) \nonumber\\
&\leq 2P(|Z|\geq(1-u_0) |\log \sigma|^{1/2}/c_h)\lesssim \sigma^{\omega}
\end{align}
as soon as $\sigma$ is small enough.
For every $j\in\mathbb N_0$, we have
$\mathcal F\{D^j H_\delta\}(t)=(-\imath t)^j \mathcal F \{H\} (\delta t)$, $t\in\mathbb R$. Then,
\[
\mathcal F\{h_{m,b,\sigma}\} (t)
=\frac{1}{\gamma}\mathcal F\{\bar h_{0,b}\}(t)
\mathcal F\{H\}(\delta t) \sum_{k=1}^{m-1}
\frac{\{-[\sigma\psi_b(t)]^2/2\}^k}{k!},\quad t\in\mathbb R.
\] 
Decomposing $\mathcal F\{\bar h_{0,b}\}(t)$ by means of $\mathcal F\{H\}(\delta t)$ 
and $[1-\mathcal F\{H\}(\delta t)]$, 
the numerator of the integrand of $\Delta_0$ in \eqref{eq:equivalent} can be bounded above by
\[\begin{split}
\mathcal J_b^2(t)&:=|e^{\sigma^2 \psi_b^2(t) /2}|^2|(1-e^{-\sigma^2 \psi_b^2(t) /2})\mathcal F\{\bar h_{0,b}\}(t)
\mathcal F\{H\}(\delta t)  + \gamma\mathcal F\{h_{m,b,\sigma}\}(t)|^2\\ 
&\hspace*{3.6cm} + |e^{\sigma^2 \psi_b^2(t) /2}-1|^2\, |\mathcal F\{\bar h_{0,b}\}(t)|^2 \abs{1 - \mathcal F\{H\} (\delta t)}^2, \quad t\in\mathbb R.
\end{split}\]
Set 
\[
\begin{split}
\Delta_{01}&:= \sum_{b=\mp1/2}M^2_{0X}(b)\int_{\delta|t|\leq u_0}[\mathcal J_b^2(t)/|\varrho_b(t)|^2]\,\di t,\\
\Delta_{02}&:=\sum_{b=\mp1/2}M^2_{0X}(b)\int_{\delta|t|> u_0}[\mathcal J_b^2(t)/|\varrho_b(t)|^2]\,\di t,
\end{split}
\]
we have $\Delta_0\lesssim\Delta_{01}+\Delta_{02}$. We prove that $\Delta_{0j}\lesssim \sigma^{2(\alpha+2)}$, for $j\in[2]$. 
Taking into account that $|e^{\sigma^2\psi_b^2(t)/2}|^2=e^{-\sigma^2(t^2-b^2)}=e^{-\sigma^2(t^2-1/4)}$, 
for $\omega\geq 2m\geq(\alpha+2)$ and $\sigma>0$ small enough, by Lemma \ref{lem:diseg}, 
relationships \eqref{eq:bilaplace1} and \eqref{eq:complement1}, we have
\[\begin{split}
\Delta_{01}&\lesssim  \sum_{b=\mp1/2}M^2_{0X}(b)\int_{\delta|t|\leq 
u_0} \frac{1}{|\varrho_b(t)|^2}
(
[\sigma^2(t^2+1/4)]^{2m}\\&\hspace*{5cm}+\sigma^{2\omega}\min\{4,\,\sigma^4(t^2+1/4)^2/4\}
)
|\mathcal F\{\bar h_{0,b}\}(t)|^2\,\di t\\
&\lesssim \sigma^{2(\alpha+2)} \sum_{b=\mp 1/2}\int_{\delta|t|\leq u_0} (|t|^{2\alpha}+1)|\widehat{(e^{b\cdot}f_{0X})}(t)|^2\,\di t
\lesssim \sigma^{2(\alpha+2)}
\end{split}
\]
because $\mathcal F\{\bar h_{0,b}\}(t)=[M_{0X}(b)]^{-1}\widehat{(e^{b\cdot}f_{0X})}(t)$, $t\in\mathbb R$, and $\int_{\mathbb R}(|t|^{2\alpha}\vee 1)|\widehat{(e^{b\cdot}f_{0X}})(t)|^2\,\di t<\infty$
by Assumption \ref{ass:sobolevcond} and the hypothesis that $(e^{|\cdot|/2}f_{0X})\in L^1(\mathbb R)$. 
Analogously, for $\sigma|t|>(u_0/c_\delta)>1$,
\[\begin{split}
\Delta_{02}&\lesssim \sum_{b=\mp 1/2}M^2_{0X}(b)\int_{\delta|t|>u_0} \frac{1}{ |\varrho_{b}(t)|^2}
\bigg(
\bigg|e^{\sigma^2 \psi_b^2(t) /2}\sum_{k=0}^{m-1}\frac{\{-[\sigma\psi_b(t)]^2/2\}^k}{k!}-1\bigg|^2\\[-4pt]
&\hspace*{6cm}+|e^{\sigma^2 \psi_b^2(t) /2}-1|^2\bigg)
|\mathcal F\{\bar h_{0,b}\}(t)|^2\,\di t\\
&\lesssim\sum_{b=\mp 1/2}M^2_{0X}(b)\int_{\delta|t|>u_0} \frac{1}{|\varrho_{b}(t)|^2}
\{e^{-\sigma^2(t^2-1/4)/2} (\sigma|t|)^{2m}+1\\[-4pt]
&\hspace*{6cm}+\min\{2,\,\sigma^2(t^2+1/4)/2\}\}^2 |\mathcal F\{\bar h_{0,b}\}(t)|^2 \,\di t \\ 
&\lesssim \sigma^{2(\alpha+2)}\sum_{b=\mp1/2}M^2_{0X}(b)\int_{\delta|t|>u_0}\frac{t^4}{|\varrho_{b}(t)|^2}
[e^{-(\sigma t)^2}(\sigma|t|)^{4m}+1]
|t|^{2\alpha} |\mathcal F\{\bar h_{0,b}\}(t)|^2  \,\di t\\
&\lesssim \sigma^{2(\alpha+2)}
\sum_{b=\mp1/2}\int_{\delta|t|>u_0}
|t|^{2\alpha}|\widehat{(e^{b\cdot}f_{0X}})(t)|^2 \,\di t \lesssim \sigma^{2(\alpha+2)}.
\end{split}\]
We prove relationship \eqref{eq:bound}.
Since $\mathcal F\{\bar h_{0,b}\}(0)=1$,
$(1-e^{-\sigma^2/8})/\gamma=-1$ and $\sigma^2/8\leq e^{\sigma^2/8}|\gamma|$,
from previous computations for the term $\Delta_{01}$ we have
$$\frac{\sigma^2}{8}|\mathcal F\{h_{m,b,\sigma}\}(0)-1|\leq
e^{\sigma^2/8}|\gamma|\Big|\mathcal F \{h_{m,b,\sigma}\}(0)+\frac{(1-e^{-\sigma^2/8})}{\gamma}\Big|
\lesssim \mathcal J_b(0)\lesssim \sigma^{2m},$$
whence $\int_{\mathbb R}h_{m,b,\sigma}(x)\,\di x=\mathcal F \{h_{m,b,\sigma}\}(0)=1+O(\sigma^{2(m-1)})$ and the proof is complete. 
\qed

\section{Final remarks}\label{sec:frmks}
In this paper, we have studied the problem of multivariate deconvolution with known ordinary smooth error distributions having independent coordinates, with respect to the $1$-Wasserstein loss. Prior to this work, optimal lower and upper bounds on the rates of convergence were derived only in \cite{dedecker2015} when $d=1$, under no smoothness assumption on the signal, leading to the minimax-optimal rate $n^{-1/(2\beta+1)}$ when the exponent $\beta$ of the Fourier transform of the noise distribution is such that $\beta\geq\frac{1}{2}$. The contributions of this work are four-fold: (1) propose an inversion inequality between $W_1(\mu_X,\,\mu_{0X})$ and $\|f_Y - f_{0Y}\|_1 $ (or $\|f_{Y,\mathsf v} - f_{0Y,\mathsf v}\|_1$ in the case where $d>1$), which can also be used in other contexts than those herein considered, for instance, as a first step to obtain Bernstein-von Mises type results for linear functionals of $\mu_{0X}$; (2) use this inversion inequality in a Bayesian framework under the Laplace noise to derive $\alpha$-adaptive minimax-optimal posterior contraction rates for any $\alpha>0$ when $d=1$; (3) prove that a kernel type deconvolution estimator achieves the minimax convergence rate under the Laplace noise for any $d\geq 1$ and (4) derive lower bounds on the $W_1$-convergence rates for any $\beta>0$ and $d\geq 1$. 
Note that the rate obtained for the kernel type deconvolution estimator easily extends to any other ordinary smooth noise distribution under additional moment assumptions. 
Along the way, we have obtained intermediate results which we believe are themselves of interest: a new approximation of a convolution between a Sobolev regular density and a Laplace distribution by the convolution of a mixture of Gaussian densities with a Laplace. This construction is different from (and significantly more involved than) the approximation of H\"older densities by mixtures of Gaussian densities constructed in \cite{kruijer2010}, which would not lead to the correct error rate in the present context.
Our method is validated by deriving lower bounds that match with the upper bounds in the case where the error coordinates are independent and homogeneous, in the sense that they are all ordinary smooth, possibly of different orders. These results pave the way to the study of the inhomogeneous case where there are mixed components, some ordinary smooth and some others supersmooth. Furthermore, the case where the error components are not independent remains to be completely investigated.


\section*{Acknowledgements}
The authors gratefully acknowledge financial support from Institut Henri Poincaré (IHP), Sorbonne Université (Paris), 
within the RIP program on \vir{Bayesian Wasserstein deconvolution} that has taken place in 2019 at the IHP-Centre Émile Borel, 
where part of this work was written. 

Catia Scricciolo has also been partially supported by Università di Verona and by MUR, PRIN project 2022CLTYP4.
She is a member of the Gruppo Nazionale per l’Analisi Matematica, la Probabilità e le loro Applicazioni (GNAMPA) of the Istituto Nazionale di Alta Matematica (INdAM). She wishes to dedicate this work to her mother and sister Emilia, with deep love and immense gratitude.

The project leading to this work has received funding from the European Research Council
(ERC) under the European Union’s Horizon 2020 research and innovation programme
(grant agreement No. 834175).


\newpage

\appendix

\begin{frontmatter}
\title{Wasserstein convergence in Bayesian deconvolution models: Supplementary Material}
\runtitle{Wasserstein convergence in Bayesian deconvolution models: Supplement}

\begin{aug}
\author{\fnms{Judith} \snm{Rousseau} and \fnms{Catia} \snm{Scricciolo}}
\address{University of Oxford and University of Verona}
\end{aug}

\begin{abstract}
This supplement contains auxiliary results for proving 
Theorems \ref{theo:1}, \ref{thm:22},  \ref{thm:31}, \ref{thm:4}, \ref{thm:lower bound} and \ref{th:comtelacour}
of the main document \cite{rousseau:scricciolo:main}. 
\end{abstract}

\end{frontmatter}



\section{Lemmas for Theorem \ref{theo:1} on the inversion inequality}\label{sec:lem:theo:1}

The following lemma provides the order of the $L^1$-norm
of the function $K_{1,h}$ that arises when controlling the term $T_1$ in Theorem \ref{theo:1}.
We recall the notation. The function $\chi:\,\mathbb R\rightarrow \mathbb R$ 
is symmetric, continuously differentiable, equal to $1$ on $[-1,\,1]$ and to $0$ 
outside $[-2,\,2]$. The kernel $K$ is defined in Section \ref{subsec:Ass} and has
Fourier transform $\hat K$ with compact support.
For $h>0$, we defined $w_{1,h}(\cdot):=\hat K(h\cdot)\chi(\cdot)r_\varepsilon(\cdot)$, with 
$r_\varepsilon$ as in \eqref{eq:re} satisfying Assumption \ref{ass:identifiability+error}.
The function $K_{1,h}(\cdot):=(2\pi)^{-1}\int_{\mathbb R}e^{-\imath t\cdot}w_{1,h}(t)\,\di t$ is the inverse 
Fourier transform of $w_{1,h}$.

\begin{lem}\label{lem:bigO}
If the single coordinate error distribution $\mu_\varepsilon$ satisfies Assumption \ref{ass:identifiability+error} for some $\beta>0$, then, for sufficiently small $h>0$,
$$\|K_{1,h}\|_1=O(1).$$
\end{lem}
\begin{proof}
Denoted by $w_{1,h}^{(1)}$ the derivative of $w_{1,h}$,
we have $\|K_{1,h}\|_1 \leq 2^{-1/2}(
\|w_{1,h}\|_2^2 +\|w^{(1)}_{1,h}\|_2^2)^{1/2}$,
see the proof of Theorem 4.2 in \cite{Bobkov:2016}, pp. 1030--1031.
For $h\leq\frac{1}{2}$, by condition \eqref{eq:deriv} with $l=0$, we have
$\|w_{1,h} \|_2^2\lesssim\int_{|t|\leq2}|\hat K(ht)|^2|\chi(t)|^2(1+|t|)^{2\beta}\,\di t\lesssim 
\|\chi\|_2^2<\infty$ as $\hat K $ is bounded on any compact set.  Analogously, for 
$w^{(1)}_{1,h}(t)=[h\hat K^{(1)}(ht)\chi(t) +\hat K(ht)\chi^{(1)}(t)] r_\varepsilon(t)+\hat K(ht)\chi(t) r^{(1)}_\varepsilon(t)$, for $t\in\mathbb R$,
using condition \eqref{eq:deriv} with $l=1$, we have
\[\begin{split}
\|w^{(1)}_{1,h}\|_2^2&\lesssim\int_{|t|\leq2}[h|\hat K^{(1)}(ht)||\chi(t)|+|\hat K(ht)||\chi^{(1)}(t)|]^2(1+|t|)^{2\beta}\,\di t \\
&\hspace*{6cm} +\int_{|t|\leq2} |\hat K(ht)|^2|\chi(t)|^2 (1+|t|)^{2(\beta-1)}\,\di t 
\\[3pt]
&\lesssim \|\chi\|_2^2 + \|\chi^{(1)}\|_2^2<\infty
\end{split}\]
because also $\hat K^{(1)}$ is bounded on any compact set by continuity. 
The assertion follows.
\end{proof}

The following lemma gives the order, in terms of the kernel bandwidth $h$, 
of the $L^1$-norm of the 
\vir{distribution function} $F_{2,h,\mathsf v}$ associated to $K_{2,h,\mathsf v}$, which is
the inverse Fourier transform of
$$w_{2,h,\mathsf v}(t):=\hat K^{\otimes d}(ht\mathsf v)[1-\chi^{\otimes d}(t\mathsf v)]r^{\otimes d}_{\varepsilon}(t\mathsf v)=
\bigg[1-\prod_{j=1}^d\chi(v_jt)\bigg]\prod_{k=1}^d\hat K(v_kht)r_\varepsilon(v_kt),\quad t\in\mathbb R.$$

\begin{lem}\label{lem:F3}
If the error distribution $\mu_\varepsilon^{\otimes d}$, $d\geq1$, has single
coordinate measure $\mu_\varepsilon$ satisfying Assumption \ref{ass:identifiability+error} for some $\beta>0$, then, 
for $h>0$ small enough, defined, for every $\mathsf v \in \mathbb S^{d-1}$, the set
$I_h^\ast(\mathsf v):=\{ j\in[d]:\,|v_j| >h\}$, we have
\begin{eqnarray}\label{eq:vineq}
\|F_{2,h,\mathsf v}\|_1
&\leq& C |\log h|\nonumber\\
&&\qquad\,\,\,\times \left(  |\log h|\1_{(\beta|I_h^\ast(\mathsf v)|\leq 1)}+ h^{-\beta |I_h^\ast(\mathsf v)|+1} 
\prod_{j\in I_h^\ast(\mathsf v)} |v_j|^\beta \1_{(\beta |I_h^\ast(\mathsf v)|> 1)}\right),
\end{eqnarray}
where $|I_h^\ast(\mathsf v)|$ denotes the cardinality of $I_h^\ast(\mathsf v)$ and $C$ does not depend on $\mathsf v$ nor on $h$. 
\end{lem}

\begin{proof}
For $\mathsf v\in\mathbb S^{d-1}$, let
$J_d^\ast(\mathsf v):=\{j\in [d]:\,v_j\neq0\}$. Note that $ \emptyset \neq I_h^\ast(\mathsf v) \subset  J_d^\ast(\mathsf v)$ because $|\mathsf v|=1$. 
Also, $|1-\chi^{\otimes d}(t\mathsf v)|\neq 0$ for all those $t\in\mathbb R$ 
for which there exists at least an index $j\in J_d^\ast(\mathsf v)$ so that $|v_jt|>1$. 
Besides, $|\hat K^{\otimes d}(h t\mathsf v)|\neq 0$
if and only if $|v_j t|\leq 1/h$ for all $j\in [d]$ because $\hat K$ is compactly supported on $[-1,\,1]$. 
Indeed, $\hat K$ is supported on $[-2,\,2]$, but, for ease of exposition and without loss of generality, 
we can assume that $\hat K$ has support on $[-1,\,1]$. 
For $h<1$ and $\mathsf v\in\mathbb S^{d-1}$, let
\[\begin{split}
\mathscr D_0&:=\cap_{j\in [d]}\{t\in\mathbb R:\,|v_jt|\leq 1/h\}\cap \{t\in\mathbb R:\,\exists\,j\in J_d^\ast(\mathsf v) \mbox{ so that }|v_jt|>1\}\\
&\,\,=
\{t\in\mathbb R:\,\|\mathsf v\|_\infty^{-1}<|t|\leq (h \|\mathsf v\|_\infty)^{-1}\}\\
&\,\,=
\{t\in\mathbb R:\,1<(\|\mathsf v\|_\infty|t|)\leq h^{-1}\},
\end{split}
\]
where $\|\mathsf v\|_\infty:=\max_{j\in [d]}|v_j|\leq1$. Note that $\mathscr D_0$ depends on $h$ and $\mathsf v$, \emph{i.e.}, $\mathscr D_0\equiv \mathscr D_{0,h,\mathsf v}$,
nevertheless, we shall not emphasize this dependence in what follows and simply write $\mathscr D_0$. 
By the same arguments used for the function $G_{2,h}$ in \cite{dedecker2015}, pp. 251--252, 
we have
\begin{equation*}\label{eq:F2}
F_{2,h,\mathsf v}(z)=
\frac{1}{2\pi}
\int_{\mathbb R}e^{-\imath t z}\frac{w_{2,h,\mathsf v}(t)}{(-\imath t)}\,\di t,\quad z\in\mathbb R,
\end{equation*}
where $t\mapsto [w_{2,h,\mathsf v}(t)/t]$ is in $L^1(\mathbb R)$ because 
$\int_{\mathscr D_0}[|w_{2,h,\mathsf v}(t)|/|t|]\,\di t
<\|w_{2,h,\mathsf v}\|_1<\infty$. Consider the integral decomposition
\[
\|F_{2,h,\mathsf v}\|_1=
\bigg(\int_{|z|\leq h}+\int_{h<|z|\leq 1}+\int_{|z|>1}\bigg)
|F_{2,h,\mathsf v}(z)|\,\di z=:F_2^{(1)}+F_2^{(2)}+F_2^{(3)}.
\]
We highlight some useful facts to study the terms $F_2^{(1)}$, $F_2^{(2)}$ and $F_2^{(3)}$.
By condition \eqref{eq:deriv} with $l=0,\,1$, over the set $\mathscr D_0$, we have
\begin{equation}\label{eq:ratio}
|r^{\otimes d}_{\varepsilon} (t\mathsf v)|\leq
\prod_{j\in J_d^\ast(\mathsf v)}(1+ |v_jt|)^\beta
\leq (1+\sqrt{d})^{ \beta (d- |I_h^\ast(\mathsf v)|)}  \prod_{j\in I_h^\ast(\mathsf v)}(1+ |v_jt|)^\beta 
\end{equation}
because $1=|\mathsf v|^2\leq d\|\mathsf v\|^2_\infty$, which implies that $\|\mathsf v\|_\infty\geq 1/\sqrt{d}$,
and
\begin{equation}\label{eq:ratioder}
\begin{split}
|(r^{\otimes d}_{\varepsilon})^{(1)} (t\mathsf v)|&\leq |r^{\otimes d}_{\varepsilon} (t\mathsf v)| \sum_{j=1}^d|v_j| \frac{|r^{(1)}_{\varepsilon}(v_ jt)|}{|r_{\varepsilon}( v_j t)|}\\
&\leq\sum_{j=1}^d|v_j| (1+|v_ jt|)^{\beta-1}\prod_{\substack{k\in[d]\\  k\neq j}}
(1+|v_k t|)^\beta < \sqrt{d}2^{\beta d}(\|\mathsf v\|_\infty| t|)^{\beta d}.
\end{split}
\end{equation}
We study $F_2^{(1)}$. By inequality \eqref{eq:ratio}, since
$\hat K\in C_b([-1,\,1])$, we have
\[\begin{split}
F_2^{(1)}:=
\int_{|z|\leq h}
\abs{F_{2,h,\mathsf v}(z)}\,\di z &\leq \frac{h}{\pi}
\int_{\mathscr D_0}
|\hat K^{\otimes d}(ht\mathsf v)| |1-\chi^{\otimes d}(t\mathsf v)|\frac{|r^{\otimes d}_{\varepsilon}(t\mathsf v)|}{|t|}\,\di t \\
&\lesssim h
\int_{\mathscr D_0}
\frac{1}{|t|}\prod_{j\in I_h^\ast(\mathsf v)}(1 +  |v_jt|)^\beta \,\di t.
\end{split}
\]  
If $d=1$, then $v_1=1$ and the above term is bounded above by $h^{-\beta+1}$. 
If $d>1$, without loss of generality, we can assume that $1\equiv v_0 \geq |v_1| \geq \ldots \geq |v_{d_h}| > h$ and, 
with abuse of notation, we can write $v_{d_h+1} \equiv h$,  where $1\leq d_h:=|I_h^\ast(\mathsf v)|\leq d$. Then,
\begin{equation*}
\begin{split}
\int_{\mathscr D_0}
\frac{1}{|t|} \prod_{j\in I_h^\ast(\mathsf v)}(1 +  |v_jt|)^\beta \,\di t & 
\lesssim \log (1/|v_1|)  + \sum_{l=1}^{d_h}\prod_{j=1}^l|v_j|^\beta\int_{1/|v_l|}^{1/|v_{l+1}|} |t|^{ \beta l-1}\,\di t \\
& \lesssim |\log h|  + \sum_{l=1}^{d_h}\frac{1}{\beta l} (|v_{l+1}|^{-\beta l } - |v_{l}|^{-\beta l} ) \prod_{j=1}^l |v_j|^\beta\\
&\lesssim |\log h|+h^{-\beta d_h}\prod_{j=1}^{d_h} |v_j|^\beta \lesssim h^{-\beta d_h}\prod_{j=1}^{d_h} |v_j|^\beta 
\end{split}
\end{equation*}
so that 
\[\begin{split}
F_2^{(1)}\lesssim h^{-\beta |I_h^\ast(\mathsf v)|+1}\prod_{j\in I_h^\ast(\mathsf v)}|v_j|^\beta . 
\end{split}
\] 
To bound $F_2^{(2)}$ and $F_2^{(3)}$, note that, by applying identity \eqref{eq:identity21} to $F_{2,h,\mathsf v}$ with $j=1$, 
we have
\begin{equation}\label{eq:F3der}
\mbox{for $z\neq0$,} \quad  F_{2,h,\mathsf v}(z)=\frac{1}{2\pi(\imath z)}\int_{\mathbb R}e^{-\imath tz}\,\pq{\frac{\di }{\di t }\pt{\frac{w_{2,h,\mathsf v}(t)}{-\imath t}}}\,\di t,
\end{equation} 
where
\begin{eqnarray}\label{eq:derw2}
&&\hspace*{-0.4cm}\frac{\di }{\di t }\pt{\frac{w_{2,h,\mathsf v}(t)}{t}}=h(\hat K^{\otimes d})^{(1)}(ht\mathsf v)[1 - \chi^{\otimes d}(t\mathsf v)] \frac{r_{\varepsilon}^{\otimes d}(t\mathsf v)}{t}\nonumber\\
&&\hspace*{3.5cm} -\hat K^{\otimes d}(ht\mathsf v) \bigg\{(\chi^{\otimes d})^{(1)}(t\mathsf v)\frac{ r^{\otimes d}_{\varepsilon}(t\mathsf v)}{t}\nonumber\\
&&\hspace*{5.6cm}\qquad-[1 - \chi^{\otimes d}(t\mathsf v)]
\pt{\frac{t(r^{\otimes d}_{\varepsilon})^{(1)}(t\mathsf v)-r^{\otimes d}_{\varepsilon}(t\mathsf v)}{t^2}}\bigg\},
\end{eqnarray}
with
\begin{eqnarray}\label{eq:unifbound}
 |(\hat K^{\otimes d})^{(1)}(ht\mathsf v)|&\leq&|\hat K^{\otimes d}(ht\mathsf v)|\sum_{j=1}^d|v_j|\frac{ |\hat K^{(1)}(v_j h t)|}{|\hat K(v_j h t)|}\nonumber\\
&=& \sum_{j=1}^d|v_j|| \hat K^{(1)}(v_j h t)|\prod_{\substack{k\in[d]\\  k\neq j}}|\hat K(v_k h t)|\leq\sqrt{d}\|K\|_1^{d-1}\int_{\mathbb R}|z||K(z)|\,\di z<\infty
\end{eqnarray}
because $K\in L^1(\mathbb R)$ as well as $zK(z)\in L^1(\mathbb R)$ by assumption, and 
\begin{equation}\label{eq:unifbchi}
 |(\chi^{\otimes d})^{(1)}(t\mathsf v)|\leq\sum_{j=1}^d|v_j||\chi^{(1)}(v_j t)|\prod_{\substack{k\in[d]\\  k\neq j}}|\chi(v_k t)|\leq \sqrt{d}\|\chi^{(1)}\|_\infty\|\chi\|_\infty^{d-1}<\infty
\end{equation}
since $\chi\in C_b([-2,\,2])$ and $\chi^{(1)}\in C_b([-1,\,1]^c\cap [-2,\,2])$. 
The bounds in \eqref{eq:unifbound} and \eqref{eq:unifbchi} are uniform over $\mathbb S^{d-1}$.
We prove below that
\begin{equation}\label{eq:normder}
\begin{split}
\int_{\mathscr D_0}\abs{\frac{\di }{\di t}\pt{\frac{w_{2,h,\mathsf v}(t)}{t}}}\,\di t & 
\lesssim  |\log h|\1_{(\beta|I_h^\ast(\mathsf v)|\leq 1)}\\
&\qquad\qquad  + h^{-\beta |I_h^\ast(\mathsf v)|+1} 
\prod_{j\in I_h^\ast(\mathsf v)} |v_j|^\beta \1_{(\beta |I_h^\ast(\mathsf v)|> 1)},\\
\pt{\int_{\mathscr D_0}\abs{\frac{\di }{\di t}\pt{\frac{w_{2,h,\mathsf v}(t)}{t}}}^2\,\di t}^{1/2} &\lesssim  |\log  h|^{1/2}\1_{(\beta |I_h^\ast(\mathsf v)|\leq3/2)} \\
&\qquad\qquad+ h^{-\beta  |I_h^\ast(\mathsf v)|+3/2}\prod_{j\in I_h^\ast(\mathsf v)} |v_j|^{\beta}\1_{(\beta  |I_h^\ast(\mathsf v)|>3/2)}.
\end{split}
\end{equation}
Then, in virtue of relationship \eqref{eq:F3der}, we have
\[
\begin{split}
F_{2,h,\mathsf v}^{(2)}&\leq\frac{1}{2\pi}\pt{\int_{h<|z|\leq 1}\frac{1}{|z|}\,\di z}\int_{\mathscr D_0}\abs{\frac{\di }{\di t}\pt{\frac{w_{2,h,\mathsf v}(t)}{t}}}\,\di t\\
&\lesssim |\log h|\left(  |\log h|\1_{(\beta|I_h^\ast(\mathsf v)|\leq 1)}+ h^{-\beta |I_h^\ast(\mathsf v)|+1} 
\prod_{j\in I_h^\ast(\mathsf v)} |v_j|^\beta \1_{(\beta |I_h^\ast(\mathsf v)|> 1)}\right)
\end{split}
\]
and 
\[
\begin{split}
F_{2,h,\mathsf v}^{(3)}&\leq\frac{1}{2\pi}\pt{\int_{|z|>1}\frac{1}{z^2}\,\di z}^{1/2}
\pt{\int_{\mathscr D_0}\abs{\frac{\di }{\di t}\pt{\frac{w_{2,h,\mathsf v}(t)}{t}}}^2\,\di t}^{1/2}\\
&\lesssim 
 |\log  h|^{1/2}\1_{(\beta |I_h^\ast(\mathsf v)|\leq3/2)} + h^{-\beta  |I_h^\ast(\mathsf v)|+3/2}\prod_{j\in I_h^\ast(\mathsf v)} |v_j|^{\beta}\1_{(\beta  |I_h^\ast(\mathsf v)|>3/2)}.
\end{split}
\]
We prove \eqref{eq:normder}. 
Using relationships \eqref{eq:derw2}, \eqref{eq:ratio}, \eqref{eq:unifbound} and reasoning as for term $F_2^{(1)}$, we get that
\[
\begin{split}
S_{1,1}:=h\int_{\mathscr D_0} | (\hat K^{\otimes d})^{(1)}(ht\mathsf v)| |1-\chi^{\otimes d}(t\mathsf v)| \frac{|r_\varepsilon^{\otimes d}(t\mathsf v)|}{|t|}\,\di t&\lesssim
h \int_{\mathscr D_0}\frac{1}{|t|}
\prod_{j\in I_h^\ast(\mathsf v)}(1+ |v_jt| )^\beta\,\di t\\
&\lesssim h^{-\beta d_h+1} 
 \prod_{j=1}^{d_h} |v_j|^\beta, 
\end{split}
\]
while
\[
\begin{split}
S_{1,2}&:=h^2\int_{\mathscr D_0} |(\hat K^{\otimes d})^{(1)}(ht\mathsf v)|^2 |1-\chi^{\otimes d}(t\mathsf v)|^2 \frac{|r^{\otimes d}_{\varepsilon}(t\mathsf v)|^2}{t^2}\,\di t\\
&\,\,\lesssim h^2 \int_{\mathscr D_0}\frac{1}{t^2}
\prod_{j\in I_h^\ast(\mathsf v)}(1+ |v_jt| )^{2\beta}\,\di t\\
&\,\,\lesssim h^2  + h^2\sum_{l=1}^{d_h}\prod_{j=1}^l|v_j|^{2\beta}\int_{1/|v_l|}^{1/|v_{l+1}|} |t|^{ 2(\beta l-1)}\,\di t \\
&\,\,\lesssim  h^2+ h^2 \sum_{l=1}^{d_h}\prod_{j=1}^l|v_j|^{2\beta}\Bigg[\log (|v_l|/|v_{l+1}|)\1_{(2\beta l = 1)}\\&\qquad\qquad\qquad\qquad\qquad\qquad\qquad+\frac{\1_{(2\beta l \neq 1)}}{ 2\beta l-1 } (|v_{l+1}|^{-2\beta l+1 } - |v_{j}|^{-2\beta l+1} )\Bigg] \\
&\,\,\lesssim   h^2 +h^2\sum_{l=1}^{d_h}\pq{|\log h|\1_{(2\beta l = 1)}+\frac{\1_{(2\beta l< 1)}}{ 1-2\beta l} + h^{-2\beta l+1}\frac{\1_{(2\beta l > 1)}}{ 2\beta l-1 }\prod_{j=1}^l |v_j|^{2\beta}} \\
&\,\,\lesssim h^2|\log  h| + h^{-2\beta d_h+3}\prod_{j=1}^{ d_h} |v_j|^{2\beta}.
\end{split}
\]
It is easily seen that
\[S_{2,1}:=\int_{\mathscr D_0}  |\hat K^{\otimes d}(ht\mathsf v)| |(\chi^{\otimes d})^{(1)}(t\mathsf v)|\frac{ |r^{\otimes d}_{\varepsilon}(t\mathsf v)|}{|t|}\,\di t=O(1)\]
and
\[S_{2,2}:=\int_{\mathscr D_0}  |\hat K^{\otimes d}(ht\mathsf v)|^2 |(\chi^{\otimes d})^{(1)}(t\mathsf v)|^2\frac{ |r^{\otimes d}_{\varepsilon}(t\mathsf v)|^2}{t^2}\,\di t=O(1).\]
Using \eqref{eq:ratio} and \eqref{eq:ratioder},  we have 
\[
\begin{split}
S_{3,1}&:=\int_{\mathscr D_0}|\hat K^{\otimes d}(ht\mathsf v)||1-\chi^{\otimes d}(t\mathsf v)|
\frac{|t(r^{\otimes d}_\varepsilon)^{(1)}(t\mathsf v)-r^{\otimes d}_{\varepsilon}(t\mathsf v)|}{t^2}\,\di t \\
&\,\,\lesssim 
\int_{\mathscr D_0}\pt{\frac{|(r^{\otimes d}_{\varepsilon})^{(1)}(t\mathsf v)|}{|t|}+\frac{|r^{\otimes d}_{\varepsilon}(t\mathsf v)|}{t^2}}\,\di t\\
&\,\,\lesssim \sum_{j=1}^d|v_j|
\int_{\mathscr D_0}\frac{1}{|t|}
(1+|v_j t|)^{\beta-1}
\prod_{\substack{k\in[d]\\  k\neq j}}(1+ |v_kt| )^\beta\,\di t
+
\int_{\mathscr D_0}\frac{1}{t^2}
\prod_{j\in I_h^\ast(\mathsf v)}(1+ |v_jt| )^\beta\,\di t\\
& \,\,\lesssim 1+\sum_{j=1}^d|v_j|\sum_{l=1}^{d_h} \int_{1/|v_l|}^{1/|v_{l+1}|}\frac{1}{|t|}(1+|v_j t|)^{\beta-1}
\prod_{\substack{k\in[d]\\  k\neq j}}(1+ |v_kt| )^\beta\,\di t\\
&\qquad\qquad\qquad\qquad\,\,+\sum_{l=1}^{d_h}\pq{|\log h|\1_{(\beta l = 1)}+\frac{\1_{(\beta l< 1)}}{ 1-\beta l} + h^{-\beta l+1}
\frac{\1_{(\beta l> 1)}}{ \beta l-1}\prod_{j=1}^l |v_j|^{\beta}}\\
& \,\,\lesssim \sum_{j=1}^d\sum_{l=1}^{d_h}\prod_{k=1}^l|v_k|^\beta
\int_{1/|v_l|}^{1/|v_{l+1}|}|t|^{\beta l-2}\,\di t+
|\log h| + h^{-\beta d_h+1}\prod_{j=1}^{d_h} |v_j|^{\beta}\\
& \,\,
\lesssim  |\log h| + h^{-\beta d_h+1} \prod_{j=1}^{d_h} |v_j|^\beta.
\end{split}
\]
Similarly,
\[
\begin{split}
S_{3,2}&:=\int_{\mathscr D_0}|\hat K^{\otimes d}(ht\mathsf v)|^2|1-\chi^{\otimes d}(t\mathsf v)|^2
\frac{|t(r_{\varepsilon}^{\otimes d})^{(1)}(t\mathsf v)-r^{\otimes d}_\varepsilon(t\mathsf v)|^2}{t^4}\,\di t \\
&\,\,\lesssim 
\int_{\mathscr D_0}\pt{\frac{|(r^{\otimes d}_{\varepsilon})^{(1)}(t\mathsf v)|^2}{t^2}+\frac{|r^{\otimes d}_{\varepsilon}(t\mathsf v)|^2}{t^4}}\,\di t\\
&\,\,\lesssim1 +\sum_{l=1}^{d_h}\pq{|\log h|\1_{(2\beta l = 3)}+ \frac{\1_{(2\beta l< 3)}}{ 3-2\beta l} + h^{-2\beta l+3}\frac{\1_{(2\beta l > 3)}}{ 2\beta l-3 }\prod_{j=1}^l |v_j|^{2\beta}} \\
&\,\,\lesssim |\log h|+ h^{-2\beta d_h+3}\prod_{j=1}^{d_h} |v_j|^{2\beta}.
\end{split}
\]
It follows that $S_{1,1}+S_{2,1}+S_{3,1}\lesssim |\log h|\1_{(\beta d_h \leq 1)}  +     h^{-\beta d_h+1}\prod_{j=1}^{d_h}  |v_j|^{\beta } \1_{(\beta d_h> 1)}$ and  
$S_{1,2}+S_{2,2}+S_{3,2}\lesssim  |\log h|\1_{(\beta d_h \leq 3/2)} +     h^{-2\beta d_h+3} \prod_{j=1}^{d_h}  |v_j|^{2\beta }\1_{(\beta d_h> 3/2)}$,
thus implying the first and second bounds in \eqref{eq:normder}, respectively.
Inequality \eqref{eq:vineq} follows by combining the bounds on $F_2^{(1)}$,  $F_2^{(2)}$ and $F_2^{(3)}$.
\end{proof}

The next lemma assesses the order of magnitude of the bias, in terms of the kernel bandwidth $h$, of any distribution function 
$F_{0X}$ having derivatives up to a certain order, with locally H\"older continuous derivative of the highest degree.
An $(\lfloor\alpha\rfloor+1)$-order kernel is used when $f_{0X}$ verifies Assumption \ref{ass:smoothXXX1} as in Lemma \ref{lem:der}.

\begin{lem}\label{lem:der}
Let $F_{0X}$ be the distribution function of $\mu_{0X}\in\mathscr P_0(\mathbb R)$ satisfying  Assumption \ref{ass:smoothXXX1} 
for $\alpha>0$. Let $K$ be a kernel of order $(\lfloor \alpha\rfloor+1)$ 
satisfying $\int_{\mathbb R}|z|^{\alpha+1}|K(z)|\,\di z<\infty$.
Then, there exists a constant $C_1>0$ such that, for every $h>0$,
\begin{equation}\label{eq:derivative}
\|F_{0X}\ast K_h-F_{0X}\|_1\leq C_1h^{\alpha+1}.
\end{equation}
\end{lem}

\begin{proof}
Let $\ell=\lfloor \alpha\rfloor$. 
For any $x,\,u\in\mathbb R$ and $h>0$, by Taylor's expansion,
\[F_{0X}(x-hu)=F_{0X}(x)-huf_{0X}(x)+\,\ldots\,+\frac{(-hu)^{\ell+1}}{\ell!}\int_0^1(1-\tau)^{\ell}f_{0X}^{(\ell)}(x-\tau hu)\,\di \tau.\]
Since $K$ is a kernel of order $\ell+1=\lfloor\alpha\rfloor+1$, we have
\[
\begin{split}
(F_{0X}\ast K_h-F_{0X})(x)&=\int_{\mathbb R}[F_{0X}(x-hu)-F_{0X}(x)]K(u)\,\di u\\
&=\int_{\mathbb R} K(u)\frac{(-hu)^{\ell+1}}{\ell!}\int_0^1(1-\tau)^{\ell}
\big[f_{0X}^{(\ell)}(x-\tau hu)-f_{0X}^{(\ell)}(x)\big]\,\di \tau\,\di u.
\end{split}
\]
Recalling the notation $b_{F_{0X}}(h):=F_{0X}\ast K_h-F_{0X}$, Assumption \ref{ass:smoothXXX1} yields that
\[
\begin{split}
\|b_{F_{0X}}(h)\|_1&\leq\int_{\mathbb R}\int_{\mathbb R}|K(u)|\frac{(h|u|)^{\ell+1}}{\ell!}\int_0^1(1-\tau)^{\ell}
|f_{0X}^{(\ell)}(x-\tau hu)-f_{0X}^{(\ell)}(x)|\,\di \tau\,\di u\,\di x\\
&\leq h^{\alpha+1}
\|L_0\|_1\frac{1}{\ell!}\pt{\int_{\mathbb R}|u|^{\alpha+1}|K(u)|\,\di u}
\int_0^1(1-\tau)^{\ell}\tau^{\alpha-\ell}\,\di \tau.
\end{split}
\]
By the assumptions that $L_0\in L^1(\mathbb R)$ and $\int_{\mathbb R}|z|^{\alpha+1}|K(z)|\,\di z<\infty$,
we conclude that $\|b_{F_{0X}}(h)\|_1\leq C_1h^{\alpha+1}$.
\end{proof}

\begin{rmk}
\emph{
The constant $C_1$ appearing in \eqref{eq:derivative} 
depends only on the kernel $K$ and the distribution function $F_{0X}$.
}
\end{rmk}

\section{Auxiliary result for Theorem \ref{thm:22}}\label{sec:section3}


We state a theorem that gives sufficient conditions for the posterior distribution to concentrate
on $L^1$-Wasserstein neighborhoods of the sampling distribution on $\mathbb R^d$. 
The assertion extends Theorem 3.2 of \cite{walker21},
p. 3643, to the $L^1$-Wasserstein metric between probability measures on $\mathbb R^d$ 
and provides conditions in terms of the prior concentration rate $\tilde\epsilon_n$ on Kullback-Leibler
type neighborhoods of the sampling distribution and in terms of moments of the
probability measures in the support of the posterior distribution so that the latter contracts at a 
nearly $\tilde\epsilon_n$-rate (up to a log-factor) on $L^1$-Wasserstein neighborhoods of the truth.
The underlying idea is to exploit the equivalence between the Wasserstein metric $W_1$ and the 
max-sliced Wasserstein metric $\overline W_1$, valid in any dimension $d\geq1$, to construct tests 
for the projected uni-dimensional distributions so that they have exponentially decaying type I and type II error probabilities.

\begin{thm}\label{lem:1}
Let $\Pi_n$ be a prior distribution on $\mathscr P_0(\mathbb R^d)$, $d\geq1$. 
Suppose that, for $\delta>0$, we have $\mu_{0Y}\in\mathscr P_0(\mathbb R^d)\cap\mathcal P_{2+\delta}(\mathbb R^d)$.
If, for $C>0$ and a sequence $\tilde\epsilon_n\geq\sqrt{(\log n)/n}$ such that $\tilde\epsilon_n\rightarrow 0$,
\begin{equation}\label{eq:33}
\Pi_n(B_{\mathrm{KL}}(P_{0Y};\,\tilde\epsilon_n^2))\gtrsim
\exp{(-Cn\tilde\epsilon_n^2)}
\end{equation}
and there exists $K>0$ so  that
\begin{equation}\label{eq:37}
\Pi_n(\mu_Y:\,M_{2+\delta}(\mu_Y)> K\mid \mathsf Y^{(n)})\rightarrow 0\,\mbox{ in $P_{0Y}^n$-probability},
\end{equation}
then, for sufficiently large $M>0$,
\begin{equation}\label{eq:66}
\Pi_n(\mu_Y:\,W_1(\mu_Y,\,\mu_{0Y})>M
\tilde\epsilon_n\log (1/\tilde \epsilon_n)\mid\mathsf Y^{(n)})\rightarrow0\,\mbox{ in $P_{0Y}^n$-probability.}
\end{equation}

\smallskip

If, instead, for $\delta'>0$, we have $\mu_{0Y}\in\mathscr P_0(\mathbb R^d)\cap\mathcal P_{4+\delta'}(\mathbb R^d)$, 
the conditions in \eqref{con1} are satisfied and, for $K'>0$,
\begin{equation}\label{prior:moment}
\Pi_n(\mu_Y:\,M_{4+\delta'}(\mu_Y)> K' \tilde \epsilon_n^{-2}) \lesssim \exp{(-(C+4)n\tilde\epsilon_n^2)},
\end{equation}
then there exists a constant $K>0$ such that \eqref{eq:37} holds for $\delta=\delta'/2$.
Consequently, the convergence in \eqref{eq:66} takes place. 
\end{thm}

The first part of the  proof is based on Theorem 3.2 of \cite{walker21}, p. 3643, but extends it to the multivariate case 
exploiting the equivalence between the Wasserstein metric $W_1$ and the 
max-sliced Wasserstein metric $\overline W_1$. 
The second part serves to prove that condition \eqref{eq:37} holds, 
provided that the posterior contraction $L^1$-norm rate has been derived. 

\begin{proof}[Proof of Theorem \ref{lem:1}]
Because $M_{2+\delta}(\mu_{0Y})<\infty$ implies that $M_1(\mu_{0Y})<\infty$, 
the hypothesis $\mu_{0Y}\in\mathcal P_{2+\delta}(\mathbb R^d)$ yields that
$\mu_{0Y}\in\mathcal P_1(\mathbb R^d)$. 
Assumption \eqref{eq:37} implies that also $\mu_Y\in\mathcal P_1(\mathbb R^d)$ so that $W_1(\mu_Y,\,\mu_{0Y})<\infty$,
see, \emph{e.g.}, \cite{villani2009}, p. 94, 
with posterior probability tending to one, in $P_{0Y}^n$-probability. 

By the inequalities in \eqref{eq:equivalence},
to prove \eqref{eq:66} it is enough to show that 
\begin{equation}\label{eq:condequiv}
\mathbb E_{0Y}^n[\Pi_n(\mu_Y:\,\overline W_1(\mu_Y,\,\mu_{0Y})>(M/C_d)\tilde\epsilon_n\log(1/\tilde\epsilon_n)\mid \mathsf Y^{(n)})]\rightarrow0.
\end{equation}
We apply a chaining argument.
For a sequence $0<\delta_n\leq\tilde\epsilon_n$, we consider a $\delta_n$-net for $\mathbb S^{d-1}$.
Since $\mathbb S^{d-1}\subseteq \{\mathsf v\in\mathbb R^d:\,|\mathsf v|\leq 1\}$, then, for $0<\epsilon<1$, 
the $\epsilon$-covering number of $\mathbb S^{d-1}$, that is,
the minimal number of $|\cdot|$-balls of radius $\epsilon$ needed to cover $\mathbb S^{d-1}$, denoted by $N(\epsilon,\,\mathbb S^{d-1},\,|\cdot|)$, is such that
\[N(\epsilon,\,\mathbb S^{d-1},\,|\cdot|)\leq N(\epsilon,\, \{\mathsf v\in\mathbb R^d:\,|\mathsf v|\leq 1\},\,|\cdot|)\leq 
3\epsilon^{-d},\]
see Proposition C.2 in \cite{bookgvdv}, p. 530. Thus, 
$N_{\delta_n}:=N(\delta_n,\,\mathbb S^{d-1},\,|\cdot|)\leq 3\delta_n^{-d}$.
Let $(\mathsf v_j)_{j\in[N_{\delta_n}]}$ be a minimal $\delta_n$-net for $\mathbb S^{d-1}$. 
Because for all $\mathsf v,\,\mathsf v_j\in \mathbb S^{d-1}$ and $\mu\in\mathcal P_1(\mathbb R^d)$,
$$W_1(\mu_{\mathsf v},\,\mu_{\mathsf v_j})\leq |\mathsf v-\mathsf v_j|M_1(\mu),$$
for every $\mu_Y\in\mathcal P_1(\mathbb R^d)$, $\mathsf v\in \mathbb S^{d-1}$ and $\mathsf v_j$ in a minimal $\delta_n$-net for $\mathbb S^{d-1}$, 
we have 
\[
\begin{split}
W_1(\mu_{Y,\mathsf v},\,\mu_{0Y,\mathsf v}) &\leq 
W_1(\mu_{Y,\mathsf v},\,\mu_{Y,\mathsf v_j})
+ W_1(\mu_{Y,\mathsf v_j},\,\mu_{0Y,\mathsf v_j})
+ W_1(\mu_{0Y,\mathsf v_j},\,\mu_{0Y,\mathsf v})\\
& \leq \max_{j\in[N_{\delta_n}]} W_1(\mu_{Y,\mathsf v_j},\,\mu_{0Y,\mathsf v_j}) \\
&\qquad\qquad\qquad+
\max_{j\in[N_{\delta_n}]}\sup_{|\mathsf v_j-\mathsf v|\leq \delta_n} 
[
W_1(\mu_{Y,\mathsf v},\,\mu_{Y,\mathsf v_j})
+ W_1(\mu_{0Y,\mathsf v_j},\,\mu_{0Y,\mathsf v})
]\\
& \leq \max_{j\in[N_{\delta_n}]} W_1(\mu_{Y,\mathsf v_j},\,\mu_{0Y,\mathsf v_j}) 
+\delta_n [M_1(\mu_Y)+M_1(\mu_{0Y})]\\
&\leq \max_{j\in[N_{\delta_n}]} W_1(\mu_{Y,\mathsf v_j},\,\mu_{0Y,\mathsf v_j}) 
+\tilde\epsilon_n [M_1(\mu_Y)+M_1(\mu_{0Y})].
\end{split}
\]
Thus, 
$$\overline W_1(\mu_{Y},\,\mu_{0Y})\leq \max_{j\in[N_{\delta_n}]} W_1(\mu_{Y,\mathsf v_j},\,\mu_{0Y,\mathsf v_j}) 
+\tilde\epsilon_n [M_1(\mu_Y)+M_1(\mu_{0Y})].$$
For $0<M'< (M/C_d)-[K+M_1(\mu_{0Y})]/\log(1/\tilde \epsilon_n)$, defined the event 
\[A_n:=\pg{\mu_Y:\, \max_{j\in[N_{\delta_n}]} W_1(\mu_{Y,\mathsf v_j},\,\mu_{0Y,\mathsf v_j})\leq M'\tilde\epsilon_n\log(1/\tilde\epsilon_n)},
\]
if $\mathbb E_{0Y}^n[\Pi_n(A_n^c\mid \mathsf Y^{(n)})]\rightarrow 0$, 
then the convergence in \eqref{eq:condequiv} follows by condition \eqref{eq:37}.

We define (a sequence of) tests $(\Psi_n)_{n\in\mathbb N}$ for the hypothesis
$H_0:\,P=P_{0Y}\equiv \mu_{0Y}$ \emph{versus} $H_1:\,P=P_Y\equiv \mu_Y$, for 
$\mu_Y\in C_n:=A_n^c\cap \{\mu_Y:\, M_{2+\delta}(\mu_Y)\leq K\}$, such that
\[
\mathbb E_{0Y}^n[\Psi_n] = o(1)
\quad
\mbox{ and}\quad\sup_{\mu_Y\in C_n}
\mathbb E_{\mu_Y}^n[1-\Psi_n]\leq \exp{(-n\tilde\epsilon_n^2)} \, \mbox{ for $n$ large enough}.
\]
Let 
\[\Psi_n:=\max_{j\in[N_{\delta_n}]}\phi_{n,j},\]
where $\phi_{n,j}$ is the test associated to $\mu_{0Y,\mathsf v_j}$ defined on pp. 3668--3669 of \cite{walker21},
with $\mu_{0Y,\mathsf v_j}$ playing the role of $P_0$ in the definition of $\phi_{m,F,-}$ and $\phi_{m,F,+}$. 
It is known from the proof of Theorem 8.9 in \cite{walker21}, p. 3665, that there exists a constant $c>0$ such that, for all $j\in[N_{\delta_n}]$,
\begin{equation*}\label{eq:typeI}
\mathbb E_{0Y,\mathsf v_j}^n[\phi_{n,j}] = \exp{(-cn\tilde\epsilon_n^2)}\quad
\mbox{ and}\quad 
\sup_{\mu_Y \in C_{n,j}}\mathbb E_{\mu_{Y,\mathsf v_j}}^n[1 - \phi_{n,j}] \leq \exp{(-cn\tilde\epsilon_n^2)},
\end{equation*}
where 
\[ C_{n,j}:= \{\mu_Y:\, W_1(\mu_{Y,\mathsf v_j},\,\mu_{0Y,\mathsf v_j})> M'\tilde\epsilon_n\log(1/\tilde\epsilon_n)\} \cap \{\mu_Y:\, M_{2+\delta}(\mu_Y)\leq K\} .\]
Recalling that $N_{\delta_n}\leq 3\delta_n^{-d}$,
\[
\mathbb E_{0Y}^n[\Psi_n]\leq \sum_{j=1}^{N_{\delta_n}}\mathbb E_{0Y,\mathsf v_j}^n[\phi_{n,j}]
\leq
 N_{\delta_n}\exp{(-cn\tilde\epsilon_n^2)} \lesssim \exp{(-cn\tilde\epsilon_n^2/2)}
\]
and
\[
 \sup_{\mu_Y \in C_n} \mathbb E_{\mu_Y}^n[1-\Psi_n] \leq  \max_{j\in[N_{\delta_n}]}
\sup_{\mu_Y \in C_{n,j}}\mathbb E_{\mu_{Y,\mathsf v_j}}^n[1 - \phi_{n,j}] \leq \exp{(-cn\tilde\epsilon_n^2)}.
\]
Using Theorem 3 of \cite{ghosal:vdv:07}, p. 196, together with assumption \eqref{eq:33}, we have that
$\mathbb E_{0Y}^n[\Pi_n(C_n\mid \mathsf Y^{(n)})]\rightarrow0$.
Then, under condition \eqref{eq:37}, the convergence in \eqref{eq:66} holds.

We now show that, under \eqref{con1}, assumption \eqref{eq:37} holds. 
The conditions in \eqref{con1} imply that 
$\mathbb E_{0Y}^n [\Pi_n (\mu_Y:\,d_{\mathrm H}(f_Y,\,f_{0Y})>M_0\tilde \epsilon_n\mid \mathsf Y^{(n)})] \rightarrow0$.
Besides, condition \eqref{prior:moment} and the Kullback-Leibler prior mass condition in \eqref{eq:33} imply that 
$$ \mathbb E_{0Y}^n[\Pi_n ( \mu_Y:\,M_{4+\delta'}(\mu_{Y})> K' \tilde \epsilon_n^{-2}\mid \mathsf Y^{(n)})] \rightarrow 0.$$
Let $\mu_Y$ be such that $M_{4+\delta'}(\mu_{Y})\leq K' \tilde \epsilon_n^{-2}$ and $d_{\mathrm H}(f_Y,\,f_{0Y})\leq M_0 \tilde \epsilon_n$. 
Since we are now assuming that $M_{4+\delta'}(\mu_{0Y})<\infty$, by the Cauchy-Schwarz inequality we have
\begin{align*}
M_{2+\delta'/2}(\mu_{Y})&\leq M_{2+\delta'/2}(\mu_{0Y})+ \int_{\mathbb R^d} |\mathsf y|^{2+\delta'/2}[|\sqrt{f_Y}-\sqrt{f_{0Y}}|(\sqrt{f_Y}+\sqrt{f_{0Y}})](\mathsf y)\,\di \mathsf y \\
& \leq M_{2+\delta'/2}(\mu_{0Y})+[M_{4+\delta'}(\mu_{Y})+M_{4+\delta'}(\mu_{0Y})]^{1/2}M_0\tilde \epsilon_n \\
& < \{M_{2+\delta'/2}(\mu_{0Y})+ [K'+M_{4+\delta'}(\mu_{0Y})]^{1/2}M_0\}=:K.
\end{align*}
Therefore, $M_{2+\delta'/2}(\mu_{Y})< K$,
which implies condition \eqref{eq:37} with $\delta=\delta'/2$.
\end{proof}

\begin{rmk}\label{eq:remarktransf}
\emph{
For $d=1$, the first part of Theorem \ref{lem:1} reduces to Theorem 3.2 of \cite{walker21}, pp. 3643 and 3667--3669, for the 
$L^1$-Wasserstein distance on $\mathbb R$. The assertion holds for any probability measure
$P_{0Y}\equiv\mu_{0Y}\in\mathscr P_0(\mathbb R^d)\cap\mathcal P_{2+\delta}(\mathbb R^d)$, with $\delta>0$. The
probability measure $\mu_{0Y}$ need not be a convolution, but if this is the case with error distribution $\mu_\varepsilon^{\otimes d}$, then
the condition $\mu_{0Y}\in\mathscr P_0(\mathbb R^d)\cap\mathcal P_{2+\delta}(\mathbb R^d)$ is implied by $\mu_{0X}\in \mathcal P_{2+\delta}(\mathbb R^d)$ and $\mu_\varepsilon\in \mathscr P_0(\mathbb R)\cap \mathcal P_{2+\delta}(\mathbb R)$.
Under the latter assumption on $\mu_\varepsilon$, condition \eqref{eq:37} boils down to require that
there exists $K^\ast>0$ such that 
$\Pi_n(\mu_X:\,M_{2+\delta}(\mu_X)> K^\ast\mid \mathsf Y^{(n)})\rightarrow 0$ in $P_{0Y}^n$-probability.
}
\end{rmk}

\begin{rmk}\label{rmk:KLneigh}
\emph{
If condition \eqref{eq:33} is replaced by
\begin{equation*}\label{eq:34}
\Pi_n(N_{\mathrm{KL}}(P_{0Y};\,\tilde\epsilon_n^2))\gtrsim
\exp{(-Cn\tilde\epsilon_n^2)},
\end{equation*}
where $N_{\mathrm{KL}}(P_{0Y};\,\tilde\epsilon_n^2):=
\{P_Y:\,\mathrm{KL}(P_{0Y};\,P_Y)\leq\tilde\epsilon_n^2\}$
is a Kullback-Leibler neighborhood of $P_{0Y}$, 
then, by Lemma 6.26 of \cite{bookgvdv}, pp. 143--144, for any sequence $L_n\rightarrow\infty$, with $P_{0Y}^n$-probability at least equal to $(1-L_n^{-1})$, 
we have 
\begin{equation}\label{eq:LBKL}
\int\prod_{i=1}^n \frac{f_Y}{f_{0Y}}(\mathsf Y_i)\,\di\Pi_n(\mu_Y)\gtrsim
\exp{(-(C+2L_n)n\tilde\epsilon_n^2)}.
\end{equation}
Following the proof of Theorem \ref{lem:1} and applying the 
lower bound in \eqref{eq:LBKL}, the convergence in \eqref{eq:66} continues to
hold with $M\tilde\epsilon_n\log(1/\tilde\epsilon_n)$ replaced by $M_n\tilde\epsilon_n\log(1/\tilde\epsilon_n)$, where
$M_n>(C+2L_n)$. Therefore, Kullback-Leibler type neighborhoods can be replaced by Kullback-Leibler neighborhoods. 
}
\end{rmk}


\section{Lemmas for Theorem \ref{thm:31} on posterior contraction rates for Dirichlet Laplace-normal mixtures}\label{sec:th31:supp}

In Lemmas \ref{lem:discrete} and \ref{lem:KL} below we prove the existence of a compactly supported discrete 
mixing probability measure such that the corresponding Laplace-normal mixture
has Hellinger distance of the appropriate order from a Laplace mixture and
the prior law on Laplace-normal mixtures concentrates 
on Kullback-Leibler neighborhoods of the true density $f_{0Y}$ at optimal rate, up to a logarithmic factor.

\smallskip

The next lemma provides an upper bound on the remainder term (or truncation  error) 
associated  with  the $(r-1)$th order Taylor polynomial about zero of the complex exponential function, 
see, \emph{e.g.}, Lemma 10.1.5 in \cite{Athreya:2006}, pp. 320--321.

\begin{lem}\label{lem:diseg}
For every $r\in\mathbb{N}$, we have 
$$\abs{e^{\imath x}-\sum_{k=0}^{r-1}\frac{(\imath x)^k}{k!}}\leq
\min\pg{\frac{|x|^{r}}{r!},\, \frac{2|x|^{r-1}}{(r-1)!}}, \quad x\in\mathbb R.$$
\end{lem}
\smallskip

For later use, we recall that the bilateral Laplace transform of a function 
$f:\,\mathbb R\rightarrow \mathbb C$ is defined as 
$\mathcal B\{f\}(s): =\int_{\mathbb R}e^{-sx} f(x)\,\di x$ for all $s\in\mathbb C$ such that
$\int_{\mathbb R}|e^{-sx}f(x)|\,\di x=\int_{\mathbb R}e^{-\mathrm{Re}(s)x}|f(x)|\,\di x<\infty$, 
where $\mathrm{Re}(s)$ denotes the real part of $s$. 
With abuse of notation, for a probability measure $\mu$ on $\mathbb R$, 
we define $\mathcal B\{\mu\}(s):=\int_{\mathbb R}e^{-sx}\mu(\di x)$, $s\in\mathbb C$.
For all $t\in\mathbb R$ such that $\int_{\mathbb R}e^{tx}\mu(\di x)<\infty$, the mapping $t\mapsto
M_{\mu}(t):=\int_{\mathbb R}e^{tx}\mu(\di x)$
is the moment generating function of $\mu$ and $M_\mu(t) = \mathcal B\{\mu\}(-t)$, $t\in\mathbb R$.

\smallskip


In the following lemma we prove the existence of a compactly supported discrete mixing
probability measure, with a sufficiently small number of support points, 
such that the corresponding Laplace-normal mixture
has Hellinger distance of the order $O(\sigma^{\beta})$, with $\beta=2$, for $\sigma>0$ small enough, 
from the sampling density $f_{0Y}$.


\begin{lem}\label{lem:discrete}
Let $f_\varepsilon$ be the standard Laplace density. 
Let  $\mu_{0X}\in\mathscr P_0(\mathbb R)$ be a probability measure supported on $[-a,\,a]$, with density $f_{0X}$
such that $(e^{|\cdot|/2}f_{0X})\in L^2(\mathbb R)$.
For $\sigma>0$ small enough, there exists a discrete probability measure $\mu_H$ on $[-a,\,a]$, with at most 
$N=O((a/\sigma)|\log\sigma|^{1/2})$ 
support points, such that, for $f_Y:=f_\varepsilon\ast (\phi_\sigma\ast \mu_H)$ and 
$f_{0Y}:=f_\varepsilon\ast f_{0X}$,
\begin{equation*}\label{eq:discrete}
d_{\mathrm{H}}(f_Y,\,f_{0Y})\lesssim \delta_0^{-1/2}e^{a_0/2}\sigma^\beta, \mbox{ with $\beta=2$,}
\end{equation*}
as soon as $\mu_{0X}(\{x\in\mathbb R:\,|x|\leq a_0\})\geq \delta_0$ for some $0<a_0<a$ and $0<\delta_0<1$.
\end{lem}


\begin{proof}
For $a_0,\,\delta_0$ as in the statement, we have
\[\begin{split}
f_{0Y}(y)\geq\int_{|x|\leq a_0}f_\varepsilon(y-x)
f_{0X}(x)\,\di x\geq \frac{\delta_0}{2}e^{-(|y|+a_0)},\quad y\in\mathbb R.
\end{split}
\]
Define 
\begin{equation}\label{eq:defU}
U(y):=e^{-y/2}+e^{y/2},\quad y\in\mathbb R.
\end{equation}
By the inequality $e^{|y|/2}\leq U(y)$, $y\in\mathbb R$,
we have 
$$d_{\mathrm H}^2(f_Y,\,f_{0Y}) \leq 
2\delta_0^{-1}e^{a_0}\int_{\mathbb R}[e^{|y|/2}(f_Y-f_{0Y})(y)]^2\,\di y\leq 2\delta_0^{-1}e^{a_0}
\|g_Y-g_{0Y}\|_2^2,$$
where $g_Y:=Uf_Y$ and $g_{0Y}:=Uf_{0Y}$.
For $b=\mp \frac{1}{2}$, we have $e^{b\cdot}f_{0Y}=(e^{b\cdot}f_\varepsilon)\ast(e^{b\cdot}f_{0X})$, 
where $e^{b\cdot}f_{0X}\in L^1(\mathbb R)$ for compactly supported $f_{0X}$ and
$e^{b\cdot}f_\varepsilon\in L^p(\mathbb R)$ for every $1\leq p\leq \infty$. Hence, 
$\|e^{b\cdot}f_{0Y}\|_p\leq \|e^{b\cdot}f_\varepsilon\|_p\times\|e^{b\cdot}f_{0X}\|_1<\infty$.
Analogously, since $e^{b\cdot}f_Y=(e^{b\cdot}(f_\varepsilon\ast \phi_\sigma))\ast(e^{b\cdot}\mu_H)$, 
where 
$M_{\mu_H}(b)<\infty
$ 
for compactly supported $\mu_H$ and
$e^{b\cdot}(f_\varepsilon\ast \phi_\sigma)\in L^p(\mathbb R)$
for $1\leq p\leq\infty$, we have 
$\|e^{b\cdot}f_Y\|_p\leq \|e^{b\cdot}(f_\varepsilon\ast \phi_\sigma)\|_p\times M_{\mu_H}(b)<\infty$.
Consequently, $g_Y,\,g_{0Y}\in L^1(\mathbb R)\cap L^2(\mathbb R)$ and 
the corresponding Fourier transforms $\hat g_Y(t):=\int_{\mathbb R}e^{\imath t y}g_{Y}(y)\,\di y$ and 
$\hat g_{0Y}(t):=\int_{\mathbb R}e^{\imath t y}g_{0Y}(y)\,\di y$, $t\in\mathbb R$, are well defined.
Also, $\|g_{0Y}\|_2^2=(2\pi)^{-1}\|\hat g_{0Y}\|_2^2$ and $\|g_Y\|_2^2=(2\pi)^{-1}\|\hat g_Y\|_2^2$.
For $\psi_b(t):=-(\imath t+b)$, let $\varrho_b(t):=[1-\psi_b^2(t)]$, $t\in\mathbb R$. Note that 
$\varrho_{-1/2}(t)=\overline{\varrho_{1/2}(t)}$ and $|\varrho_{-1/2}(t)|^2=|\varrho_{1/2}(t)|^2=(t^4+5t^2/2+9/16)$. 
Since
$$\mathcal B\{f_\varepsilon(\cdot-x)\}(\psi_b(t))=
\frac{e^{-\psi_b(t)x}}{\varrho_b(t)},\quad t,\,x\in\mathbb R,$$
we have
\[
\begin{split}
r(t;\,x)&:=\int_{\mathbb R}e^{\imath t y}U(y)f_\varepsilon(y-x)\,\di y\\&=\sum_{b=\mp 1/2}\mathcal B\{f_\varepsilon(y-x)\}(\psi_b(t))
=\sum_{b=\mp 1/2}\frac{e^{-\psi_b(t)x}}{\varrho_{b}(t)}
,\quad t,\,x\in\mathbb R.
\end{split}
\]
Then,
$\hat g_{0Y}(t)
=\int_{|x|\leq a}r(t;\,x)f_{0X}(x)\,\di x
=\sum_{b=\mp 1/2}\mathcal B\{f_{0X}\}(\psi_b(t))/\varrho_{b}(t)$,
$t\in\mathbb R$.   
We derive the expression of $\hat g_Y$.
Since
\begin{equation*}\label{eq:lapnorm}
\mathcal B\{\phi_\sigma(\cdot-u)\}(\psi_b(t))
=\exp{(-\psi_b(t)u+\sigma^2\psi_b^2(t)/2)}, \quad t,\,u\in\mathbb R,
\end{equation*} 
we have
\[
\begin{split}
\hat g_Y(t)
&=\int_{|u|\leq a}\pt{\int_{\mathbb R}r(t;\,x)\phi_\sigma(x-u)\,\di x}\mu_H(\di u)\\
&=\sum_{b=\mp 1/2}\frac{e^{\sigma^2\psi_b^2(t)/2}}{\varrho_{b}(t)}
\mathcal B\{\mu_H\}(\psi_b(t)), 
\quad t\in\mathbb R.
\end{split}
\] 
For ease of notation, we introduce the integrals
\[I_b:=\int_{\mathbb R}\frac{1}{|\varrho_b(t)|^2}
|e^{\sigma^2\psi_b^2(t)/2}
\mathcal B\{\mu_H\}(\psi_b(t))-\mathcal B\{f_{0X}\}(\psi_b(t))
|^2\,\di t,\quad b=\mp\frac{1}{2}.\]
By Plancherel's theorem and the triangular inequality, 
$2\pi\|g_Y-g_{0Y}\|_2^2=\|\hat g_Y-\hat g_{0Y}\|_2^2\leq 2(I_{-1/2}+I_{1/2})$.
Both terms $I_{-1/2}$ and $I_{1/2}$ can be controlled using the same arguments, we therefore 
consider a unified treatment for $I_b$. For $M>0$, we have
\[
\begin{split}  
I_b
&\leq
\pt{\int_{|t|\leq M}+\int_{|t|>M}}\frac{|e^{\sigma^2\psi_b^2(t)/2}|^2}{{|\varrho_b(t)|^2}}
|(\mathcal B\{\mu_H\}-\mathcal B\{f_{0X}\})(\psi_b(t))|^2\,\di t\\
&\hspace*{3.5cm}+
\int_{\mathbb R}\frac{1}{{|\varrho_b(t)|^2}}
|e^{\sigma^2\psi_b^2(t)/2}-1|^2
\abs{\mathcal B\{f_{0X}\}(\psi_b(t))
}^2\,\di t=:\sum_{k=1}^3 I_b^{(k)}.
\end{split}
\]  
\noindent
\emph{Study of the term $I_b^{(1)}$}\\[5pt]
The term $I_b^{(1)}$ can be bounded similarly to $I_1$ in Lemma 2 of \cite{gao2016}, p. 616. 
Preliminarily note that,
for $\sigma<1/|b|=2$, we have
$|e^{\sigma^2\psi_b^2(t)/2}|^2=|e^{\sigma^2(-t^2+2\imath bt+b^2)/2}|^2
=e^{-\sigma^2(t^2-b^2)}=e^{-\sigma^2(t^2-1/4)}< e$.
Let $\mu_H$ be a discrete probability measure on $[-a,\,a]$ satisfying the constraints
\begin{equation}\label{eq:c2}
\begin{split}
\int u^j\mu_H(\di u) &=\int u^j f_{0X}(u)\,\di u, \quad j=0,\,\ldots,\,J-1, \\
\int e^{bu}\mu_H(\di u)& =\int e^{bu}f_{0X}(u)\,\di u, \quad  b= \mp \frac{1}{2},
\end{split}
\end{equation}
where $J=\lceil\eta ea M\rceil$ for some $\eta>1$, with $\lceil x \rceil:=\min\,\{k\in\mathbb{Z} : k>x\}$ the upper integer part of $x$. 
Note that the second set of 
constraints in \eqref{eq:c2} can be written as
$M_{\mu_H}(b)=M_{0X}(b)$, with $b=\mp\frac{1}{2}$.
Using Lemma \ref{lem:diseg} with $r=J$,
by the inequality $J!\geq (J/e)^J$, we have 
\[\begin{split}
I_b^{(1)}&\lesssim \int_{|t|\leq M}\frac{1}{{|\varrho_b(t)|^2}}
\bigg|\int_{|u|\leq a}e^{bu}  \bigg[e^{\imath t u}-\sum_{j=0}^{J-1}\frac{(\imath tu)^j}{j!}\bigg](\mu_H-\mu_{0X})(\di u)
\bigg|^2\,\di t\\
&\lesssim[M_{\mu_H}(b)+M_{0X}(b)]^2\frac{1}{(J!)^2}
\int_{|t|\leq M}\frac{(a|t|)^{2J}}{{|\varrho_b(t)|^2}}\,\di t\\
&\lesssim \frac{a^{2J}}{(J!)^2}
\int_{0}^M t^{2(J-2)}\,\di t \\
&\lesssim \frac{a^{2J}}{(J!)^2}\times
\frac{M^{2J-3}}{2J-3}\lesssim M^{-4} \frac{a^{2J}}{(J!)^2}\times
\frac{M^{2J+1}}{2J-3}\lesssim  M^{-4} \pt{\frac{eaM}{J}}^{2J+1}\lesssim M^{-4}.
\end{split}\]   
\noindent
\emph{Study of the term $I_b^{(2)}$}\\[5pt]
Note that, for $b=\mp\frac{1}{2}$, 
\begin{equation}\label{eq:disff}
\mathcal B\{f_{0X}\}(\psi_b(t))=(\widehat{e^{b\cdot}f_{0X}})(t), \quad t\in\mathbb R,
\end{equation} 
so that $|\mathcal B\{f_{0X}\}(\psi_b(t))|\leq M_{0X}(b)$. Similarly, 
$|\mathcal B\{\mu_H\}(\psi_b(t))|\leq M_{\mu_H}(b)= M_{0X}(b)$.
Choosing $M$ so that  
$(\sigma M)^2\geq|\log\sigma|$, equivalently, $M\geq \sigma^{-1}|\log\sigma|^{1/2}
$, and using the fact that $|e^{\sigma^2\psi_b^2(t)/2}|^2=O(e^{-(\sigma t)^2})$, we have
$$I_b^{(2)}
\lesssim M^2_{0X}(b)
e^{-(\sigma M)^2}
\int_{|t|>M}\frac{1}{t^4}
\,\di t
\lesssim
e^{-(\sigma M)^2}M^{-3}\lesssim 
\sigma M^{-3}
\lesssim \sigma^4.$$
\noindent
\emph{Study of the term $I_b^{(3)}$}\\[5pt]
By Lemma \ref{lem:diseg},
\begin{equation*}\label{eq:disf}
|e^{\sigma^2\psi_b^2(t)/2}-1|\leq \min\{2,\,\sigma^2(t^2+b^2)/2\}\leq\sigma^2(t^2+1/4)/2,
\end{equation*}
which combined with \eqref{eq:disff} gives
\[  
\begin{split}
I_b^{(3)}
&\leq \frac{\sigma^4}{2}\int_{\mathbb R}\frac{(t^4+b^4)}{{|\varrho_b(t)|^2}}
\abs{\mathcal B\{f_{0X}\}(\psi_b(t))
}^2\,\di t
\lesssim
\sigma^4
\int_{\mathbb R}
|(\widehat{e^{b\cdot}f_{0X}})(t)|^2\,\di t\lesssim
\sigma^4,
\end{split}   
\]  
where, by Plancherel's theorem,
$(2\pi)^{-1}\|\widehat{e^{b\cdot}f_{0X}}\|_2^2=\|e^{b\cdot}f_{0X}\|_2^2<\infty
$ by the assumption that $(e^{|\cdot|/2}f_{0X})\in L^2(\mathbb R)$.

\vspace*{-0.1cm}
\smallskip
The existence of a discrete probability measure $\mu_H$ supported on $[-a,\,a]$, with at most 
$[(J+2)+1]\propto (a M)\gtrsim (a/\sigma)|\log\sigma|^{1/2}$ support 
points, is guaranteed by Lemma A.1 of \cite{ghosal2001}, p. 1260.
Combining the bounds on $I_b^{(k)}$, $k\in[3]$, 
we conclude that 
$\|g_Y-g_{0Y}\|_2^2\lesssim 
\sigma^4$. It follows that $d^2_{\mathrm H}(f_Y,\,f_{0Y})n\lesssim \delta_0^{-1} e^{a_0}\sigma^{4}$, which completes the proof.
\end{proof}


The next lemma gives sufficient conditions on the distribution $\mathscr D_{H_0}\otimes \Pi_\sigma$ so that 
the induced prior probability measure $\Pi$ on Laplace-normal mixtures 
$f_Y=f_\varepsilon\ast(\phi_\sigma\ast\mu_H)$ concentrates on Kullback-Leibler neighborhoods 
of a Laplace mixture $f_{0Y}=f_\varepsilon\ast f_{0X}$, with mixing density $f_{0X}$ having exponentially decaying tails, at a rate of the order $O(n^{-2/5}(\log n)^\tau)$ for suitable $\tau>0$.

\begin{lem}\label{lem:KL}
Let $f_{0Y}:=f_\varepsilon\ast f_{0X}$, where $f_\varepsilon$ is the density of a standard 
Laplace distribution and $f_{0X}$ satisfies Assumption \ref{ass:twicwtailcond}. 
Consider the model $f_Y:=f_\varepsilon\ast (\phi_\sigma\ast \mu_H)$, with $\mu_H\in\mathscr P(\mathbb R)$.
If the base measure $H_0$ of the Dirichlet process prior $\mathscr D_{H_0}$ for $\mu_H$ 
satisfies Assumption \ref{ass:basemeasure1} and the prior $\Pi_\sigma$ for $\sigma$ satisfies 
Assumption \ref{ass:priorscale1} with $0<\gamma\leq1$, then
$\Pi
(N_{\mathrm{KL}}(P_{0Y};\,\tilde\epsilon_n^2)
)\gtrsim \exp{(-C n\tilde\epsilon_n^2)}$, for $\tilde\epsilon_n=n^{-2/5}(\log n)^{1/2+(2t_1\vee 3)/5}$.
\end{lem}

\begin{proof} 
We use the generic exponent $\beta>0$ of the Fourier transform of the error density in those steps of the proof that do not depend on the specific form of the Laplace density.
We show that, for some constant $C>0$, the prior probability of a 
Kullback-Leibler neighborhood of $P_{0Y}$ of radius $\tilde\epsilon_n^2$ is at least $\exp{(-C n\tilde\epsilon_n^2)}$.
We apply Lemma B2 of \cite{ghosal:shen:tokdar}, pp. 638--639, to relate 
$N_{\mathrm{KL}}(P_{0Y};\,\xi^2)$ to a Hellinger ball of appropriate radius. 
By Assumption \ref{ass:twicwtailcond}, there exists $C_0 >0$ such that  
$\mu_{0X}([-a,\,a]^c) \lesssim  e^{-(1+C_0)a}$ for $a$ large enough. 
Set $a_\eta:=a_0|\log\eta|$, with $a_0\geq[2/(1+C_0)]$ and $\eta>0$ small enough,
we have $\mu_{0X}([-a_\eta,\,a_\eta]^c) \lesssim\eta^2$. Then,
Lemma A.3 of \cite{ghosal2001}, p. 1261, shows that the $L^1$-distance between $f_{0Y}$ and 
$f_{0Y}^\ast:=f_\varepsilon\ast f_{0X}^\ast$, where $f_{0X}^\ast$ is the density of the renormalized restriction of $\mu_{0X}$ to $[-a_\eta,\,a_
\eta]$, denoted by $\mu_{0X}^\ast$, is bounded above by $2\eta^2$. From
$d_{\mathrm{H}}^2(f_{0Y},\,f^\ast_{0Y})\leq\|f_{0Y}-f^\ast_{0Y}\|_1\leq 2\eta^2$, we have $d_{\mathrm{H}}(f_{0Y},\,f^\ast_{0Y})\lesssim\eta$.
Lemma \ref{lem:discrete} applied to $\mu^\ast_{0X}$ (which plays the role of $\mu_{0X}$ in the statement) shows that, for $\sigma>0$ small enough, there exists a discrete probability measure $\mu_H^\ast$ supported on $[-a_\eta,\,a_\eta]$, with at most $N=O((a_\eta/\sigma)|\log\sigma|^{1/2})$ support points, such that 
$f_Y^\ast:=f_\varepsilon\ast (\phi_\sigma\ast\mu_H^\ast)$ satisfies 
$$d_{\mathrm{H}}(f_Y^\ast,\,f^\ast_{0Y})\lesssim \sigma^\beta.$$
An analogue of Corollary B1 in \cite{ghosal:shen:tokdar}, p. 16, shows that 
$\mu_H^\ast=\sum_{j=1}^Np_j\delta_{u_j}$ has support points inside $[-a_\eta,\,a_\eta]$, with at least 
$\sigma^{1+2\beta}$-separation between every pair of points $u_i\neq u_j$, 
and that $d_{\mathrm{H}}(f_Y^\ast,\,f^\ast_{0Y})\lesssim \sigma^{\beta}$.
Consider disjoint intervals $U_j$, for $j\in[N]$, centred at $u_1,\,\ldots,\,u_N$, with length $\sigma^{1+2\beta}$ each.
Extend $\{U_1,\,\ldots,\,U_N\}$ to a partition $\{U_1,\,\ldots,\,U_K\}$ of $[-a_\eta,\,a_\eta]$ such that each 
$U_j$, for $j=N+1,\,\ldots,\,K$, has length at most $\sigma$.
Further extend this to a partition $U_1,\,\ldots,\, U_M$ of $\mathbb R$ such that, for some constant $a_1>0$, 
we have $\sigma^{a_1}\leq H_0(U_j)\leq 1$, 
for $j\in[M]$. 
The whole process can be done with a total number $M$ of intervals of the same order as $N$.
Define $p_j=0$, for $j=N+1,\,\ldots,\,M$. Let $\mathscr P_\sigma$ be the set of probability measures $\mu_H\in\mathscr P(\mathbb R)$ 
with 
\begin{equation}\label{eq:diffprob}
\sum_{j=1}^K|\mu_H(U_j)-p_j|\leq 2\sigma^{2\beta+1} \quad \mbox{and}\quad \min_{j\in[K]}\mu_H(U_j)\geq \sigma^{2(2\beta+1)}/2.
\end{equation} 
Note that $\sigma^{2\beta+1}K<1$.
By Lemma 5 in \cite{ghosal2007}, p. 711, or 
Lemma B1 in \cite{ghosal:shen:tokdar}, p. 16, with $V_0:=\cup _{j>N}U_j$ and $V_j\equiv U_j$, for $j\in[N]$, for any $\mu_H\in\mathscr P_\sigma$ 
we have $d_{\mathrm{H}}^2(f_Y,\,f^\ast_Y)\leq\|f_Y-f^\ast_Y\|_1\lesssim \sigma^{2\beta}$. Then, for 
$\eta=O(\sigma^\beta)$, we have
$d_{\mathrm{H}}^2(f_Y,\,f_{0Y}) \lesssim
d_{\mathrm{H}}^2(f_Y,\,f_Y^\ast) +d_{\mathrm{H}}^2(f_Y^\ast,\,f_{0Y}^*) + d_{\mathrm{H}}^2(f_{0Y}^*,\,f_{0Y})\lesssim \sigma^{2\beta}$.
To apply Lemma B2 of \cite{ghosal:shen:tokdar}, pp. 16--17, we study the ratio $(f_Y/f_{0Y})$.
Let $\mu_H\in\mathscr P_\sigma$. 
For a standard Laplace error distribution ($\beta=2$), since $\|f_{0Y}\|_\infty\leq\frac{1}{2}$, for $|y|<a_\eta$,
\begin{align*}
\frac{f_Y}{f_{0Y}}(y)
&\gtrsim \int_{|x|\leq a_\eta} f_\varepsilon (y-x)\int_{|x-u|\leq\sigma}\phi_\sigma(x-u)\mu_H(\di u)\,\di x\\
&\gtrsim \frac{1}{\sigma}\int_{|x|\leq a_\eta} f_\varepsilon (y-x)\mu_H(U_{J(x)})\,\di x\gtrsim \sigma^{4\beta+1}a_\eta e^{-2a_\eta},
\end{align*}
while, for $|y|\geq a_\eta$,
\[\begin{split}
\frac{f_Y}{f_{0Y}}(y)
&\gtrsim \int_{|x|\leq a_\eta} f_\varepsilon(y-x)\int_{|u|\leq a_\eta}\phi_\sigma(x-u)\mu_H(\di u)\,\di x
\gtrsim \frac{a_\eta}{\sigma}e^{-|y|}e^{-a_\eta}e^{-2(a_\eta/\sigma)^2},
\end{split}\]
where 
$\mu_H([-a_\eta,\,a_\eta])\geq 1-2\sigma^{2\beta+1}$ 
because of the first condition in \eqref{eq:diffprob}.
For $\lambda=\sigma^{4\beta+1}a_\eta e^{-2a_\eta}$, we have $\log(1/\lambda)\lesssim |\log\sigma|$.
Since $\{y\in\mathbb R:\,(f_Y/f_{0Y})(y)\leq\lambda\}\subseteq \{y\in\mathbb R:\,|y|\geq a_\eta\}$, 
\[
\begin{split}
P_{0Y}\hspace*{-0.1cm}\pt{\hspace*{-0.1cm}\pt{\log\frac{f_{0Y}}{f_Y}}\1_{\big(\frac{f_Y}{f_{0Y}}\leq \lambda\big)}}
&\lesssim\int_{|y|\geq a_\eta}\pt{\log\frac{f_{0Y}}{f_Y}(y)}f_{0Y}(y)\,\di y
\lesssim\frac{1}{\sigma^2}\int_{|y|\geq a_\eta}y^2f_{0Y}(y)\,\di y, 
\end{split}
\]
where
\[\begin{split}
\int_{|y|\geq a_\eta}y^2f_{0Y}(y)\,\di y
&\leq\pt{\int_{\mathbb R}e^{|x|}f_{0X}(x)\,\di x}
\int_{|y|\geq a_\eta}y^2f_\varepsilon(y)\,\di y\\
&\lesssim
\int_{|y|\geq a_\eta}e^{-|y|/2}\,\di y\lesssim e^{-a_\eta/2}\lesssim \sigma^6,
\end{split}
\]
see, \emph{e.g.}, Lemma A.7 in \cite{scricciolo:2011}, pp. 303--304,
provided that $a_0\geq\max\{6,\,2/(1+C_0)\}$. Consequently, 
$P_{0Y}(\log(f_{0Y}/f_Y)\1_{((f_Y/f_{0Y})\leq\lambda)})\lesssim \sigma^4$.

\smallskip

\noindent
Thus, $P_{0Y}(\log(f_{0Y}/f_Y)\1_{((f_Y/f_{0Y})\leq\lambda)})\lesssim
\sigma^{2\beta}$, with $\beta=2$.
Lemma B2 of \cite{ghosal:shen:tokdar}, pp. 16--17, implies that $P_{0Y}\log(f_{0Y}/f_Y)$ is bounded above by 
$\sigma^{2\beta}|\log \sigma|$. By Lemma 10 of \cite{ghosal2007}, p. 714, we have $\mathscr D_{H_0}(\mathscr P_\sigma)\gtrsim\exp{(-c_1K|\log\sigma|)}
\gtrsim \exp{(-c_2(a_\eta/\sigma)|\log\sigma|^{3/2})}$ for constants $c_1,\,c_2>0$ that depend on $H_0(\mathbb R)$ and $a_1$.
Given $\sigma>0$, define $\mathscr S_\sigma:=\{\sigma'>0:\,\sigma(1+\sigma^d)^{-1}\leq\sigma'\leq\sigma\}$ for a constant 
$0<d\leq s_1-1$. Then, $\Pi_\sigma(\mathscr S_\sigma)\gtrsim \exp{(-D_1\sigma^{-\gamma}|\log\sigma|^{t_1})}$.
Replace $\sigma$ at every occurrence with $\sigma'\in\mathscr S_\sigma$.
For $\xi:=\sigma^{\beta}|\log\sigma|^{1/2}$, noting that $|\log\sigma|\lesssim|\log\xi|$, 
since $\gamma\leq1$ we have
\[
\begin{split}
\Pi(N_{\textrm{KL}}(P_{0Y};\,\xi^2))
&\gtrsim\exp{(-c_2(a_\eta/\sigma)|\log\sigma|^{3/2})}\times\exp{(-D_1\sigma^{-\gamma}|\log\sigma|^{t_1})}\\
&\gtrsim 
\exp{(-c_3(a_\eta/\sigma)|\log\sigma|^{(t_1\vee 3/2)})}\\
&\gtrsim 
\exp{(-c_4\xi^{-1/\beta}|\log\xi|^{1/(2\beta)+1+(t_1\vee 3/2)})}.
\end{split}
\]
Replacing $\xi$ with $\tilde\epsilon_n=n^{-\beta/(2\beta+1)}(\log n)^{1/2+\beta(t_1\vee 3/2)/(2\beta+1)}$,
for a suitable constant $C>0$, we have
$
\Pi(N_{\textrm{KL}}(P_{0Y};\,\tilde\epsilon_n^2))\gtrsim \exp{(-C n\tilde\epsilon_n^2)}
$ and the proof is complete.
\end{proof} 


\section{Lemmas for Theorem \ref{thm:4} on adaptive
posterior contraction rates for Dirichlet Laplace-normal mixtures}


The following lemma assesses the order of the bias of the distribution function
corresponding to a Gaussian mixture, where the mixing distribution 
is any probability measure on $\mathbb R$ and the scale 
parameter is bounded below by a multiple of the kernel bandwidth $h$, times a logarithmic factor.
It shows that, when $d=1$, condition \eqref{eq:ass1} of Theorem \ref{theo:1} is verified 
for a universal positive constant $C_1$.

\begin{lem}\label{lem:biasmixgaus}
Let $F_{X}$ be the distribution function of $\mu_X=\phi_\sigma\ast \mu_H$, 
with $\sigma>0$ and $\mu_{H}\in\mathscr P(\mathbb R)$. 
Let $K\in L^1(\mathbb R)\cap L^2(\mathbb R)$ be symmetric, with 
$\int_{\mathbb R}|z||K(z)|\,\di z<\infty$ and $\hat K\in L^1(\mathbb R)$ such that
$\hat K\equiv 1$ on $[-1,\,1]$.
Then, there exists $C_1>0$, depending only on $K$, such that, 
for $\alpha>0$, $0<h<1$ and $h\sqrt{(2\alpha+1)|\log h|}\leq \sigma <1$, we have
\begin{equation*}\label{eq:derivative21}
\|F_{X}-F_{X}\ast K_h\|_1 \leq C_1 h^{\alpha+1}.
\end{equation*}
\end{lem}

\begin{proof}
Defined the function $\hat f(t):=[1-\hat K(ht)][\hat\phi(\sigma t)/t]$, $t\in\mathbb R$,
since $t\mapsto \hat\mu_H(t)\hat f(t)$ is in $L^1(\mathbb R)$, 
arguing as for $G_{2,h}$ in \cite{dedecker2015}, pp. 251--252, we have
\[\begin{split}
\|b_{F_X}(h)\|_1=\|F_{X}-F_{X}\ast K_h\|_1&=
\int_{\mathbb R}\bigg|\frac{1}{2\pi}\int_{|t|>1/h}
e^{-\imath t x}\hat \mu_H(t)\hat f(t)\,\di t
\bigg|\,\di x\\
&=\|\mu_H\ast f\|_1\leq \| f\|_1
\leq \frac{1}{\sqrt{2}}(\|\hat f\|_2^2+\|\hat f^{(1)}\|_2^2)^{1/2},
\end{split}\]
see, \emph{e.g.}, \cite{Bobkov:2016}, p. 1031, for the last inequality. 
Using the fact that $\|\hat K\|_\infty \leq \|K\|_1<\infty$, we have
$\|\hat f\|_2^2< \hat\phi(\sqrt{2}\sigma/h)(1+\|K\|_1)^2\int_{|t|>1/h}
t^{-2}\,\di t= 2(1+\|K\|_1)^2he^{-(\sigma/h)^2} \lesssim h^{2(\alpha+1)}$. Besides,
\[
\hat f^{(1)}(t)=-\pg{h\hat K^{(1)}(ht)+\pt{\frac{1}{t
}+\sigma^2 t
}[1-\hat K(ht)]
}
\frac{\hat\phi(\sigma t)}{t}\1_{[-1,\,1]^c}(ht),\quad t\in\mathbb R.
\]
Since $K\in L^1(\mathbb R)$ and $zK(z)\in L^1(\mathbb R)$ jointly imply that $\hat K$ is 
continuously differentiable with $|\hat K^{(1)}(t)|\rightarrow 0$, as $|t|\rightarrow \infty$, 
so that $\hat K^{(1)}\in C_b(\mathbb R)$, we have
\[\begin{split}
\|\hat f^{(1)}\|_2^2 &< 2 e^{-(\sigma/h)^2}
\int_{|t|>1/h}
\pq{\frac{h^2}{t^2}\|\hat K^{(1)}\|_\infty^2+\frac{2}{t^4}(1+\|K\|_1)^2}\,\di t\\
&\qquad\qquad\qquad\qquad\qquad\qquad\qquad\qquad\quad +4\sigma^4(1+\|K\|_1)^2\int_{|t|>1/h}e^{-(\sigma t)^2}\,\di t \\
&\lesssim (h^2+\sigma^2)he^{-(\sigma/h)^2}\lesssim he^{-(\sigma/h)^2}\lesssim h^{2(\alpha+1)}. 
\end{split}
\]
The assertion follows for a suitable constant $C_1>0$ depending only on $K$.
\end{proof}

\begin{rmk}
\emph{Due to the exponentially decaying tails of the Gaussian density and to
a suitable choice of the scale parameter $\sigma$ greater than or equal to a  
multiple of the kernel bandwidth $h$, times a logarithmic factor, 
a different argument than that used in the proof of Theorem \ref{theo:1} for the case 
when the smoothness Assumption \ref{ass:smoothXXX} on $\mu_{0X}$ holds 
is used to bound the bias of the distribution function $F_X$ associated to $\mu_X=\phi_\sigma\ast \mu_H$. 
}
\end{rmk}

We introduce some more notation. For $h=o(1)$, let $\delta=o(h)$. 
For $m\in\mathbb N$, $b=\mp\frac{1}{2}$ and $\sigma=o(1)$, we define the set
\begin{equation}\label{def:set}
A_{b,\sigma}:=\pg{x\in\mathbb R:\,\gamma h_{m,b,\sigma}(x)>-\frac{1}{2}\bar h_{0,b}(x)},
\end{equation}
with $\bar h_{0,b}$ and $h_{m,b,\sigma}$ as defined in \eqref{eq:barh} and \eqref{hm}, respectively,
and the function
\begin{equation}\label{def:functiong}
g_{b,\sigma}:=M_{0X}(b)e^{-b\cdot}\gamma h_{m,b,\sigma}\1_{A_{b,\sigma}}-\frac{1}{2}f_{0X}\1_{A_{b,\sigma}^c}.
\end{equation}

\smallskip


In the following lemma we prove the existence of a compactly supported discrete mixing
probability measure, with a sufficiently small number of support points, 
such that the corresponding Laplace-normal mixture
has Hellinger distance of the order $O(\sigma^{\alpha+2})$ from the Laplace mixture sampling density $f_{0Y}=f_\varepsilon\ast f_{0X}$
having an $\alpha$-Sobolev regular mixing density $f_{0X}$ with exponentially decaying tails.

\begin{lem}\label{lem:deltabound1}
Let $f_\varepsilon$ be the standard Laplace density.
Let $f_{0X}$ be a density 
satisfying Assumption \ref{ass:twicwtailcond}, Assumption \ref{ass:sobolevcond} for $\alpha>0$ and 
Assumption \ref{ass:smoothf0}.
For $\sigma>0$ small enough, there exist a constant $A_0>0$ and a discrete probability measure on $[-a_\sigma,\,a_\sigma]$, with 
$a_\sigma:=A_0|\log\sigma|$, having at most $N=O((a_\sigma/\sigma)|\log \sigma|^{1/2})$ support points, such that, for $f_Y:=f_\varepsilon\ast (\phi_\sigma\ast\mu_H)$ 
and $f_{0Y}:=f_\varepsilon\ast f_{0X}$,
$$d_{\mathrm H}(f_Y,\,f_{0Y})\lesssim \delta_0^{-1/2}e^{a_0/2}\sigma^{\alpha+2}$$
as soon as $\mu_{0X}(\{x\in\mathbb R:\,|x|\leq a_0\})\geq \delta_0$ for some $0<a_0<a_\sigma$ and $0<\delta_0<1$.
\end{lem}

\begin{proof}
Reasoning as in Lemma \ref{lem:discrete}, for $a_0,\,\delta_0$ as in the statement, 
$d_{\mathrm H}^2(f_{Y},\,f_{0Y}) \leq 2\delta_0^{-1}e^{a_0}\|g_{Y}-g_{0Y}\|_2^2$, 
where $g_{Y}:=Uf_{Y}$ and $g_{0Y}:=Uf_{0Y}$, with $U$ defined in \eqref{eq:defU}.
Note that $(e^{|\cdot|/2}f_{0X})\in L^1(\mathbb R)\cap L^2(\mathbb R)$ by Assumption \ref{ass:twicwtailcond}. 
Also, $g_Y,\,g_{0Y}\in L^1(\mathbb R)\cap L^2(\mathbb R)$ so that, not only are 
the corresponding Fourier transforms
$\hat g_{Y},\,\hat g_{0Y}$ well defined, but $\|g_{Y}\|_2^2=(2\pi)^{-1}\|\hat g_{Y}\|_2^2$ and  
$\|g_{0Y}\|_2^2=(2\pi)^{-1}\|\hat g_{0Y}\|_2^2$. In order to bound $\|g_{Y}-g_{0Y}\|_2^2$, 
some definitions and preliminary facts are exposed. For $T\geq[(\alpha+2)/\vartheta]$, with $\vartheta\in(0,\,1)$, 
we define the set $E_\sigma:=\{x\in\mathbb R:\,f_{0X}(x)>\sigma^T\}$. The tail condition on $f_{0X}$ of Assumption \ref{ass:twicwtailcond} implies that $E_\sigma \subset \{x\in\mathbb R:\, |x| \leq  A_0 
|\log\sigma|\} $ for some $A_0>0$. Note that $A_0$ can be chosen arbitrarily large by choosing $T$ large enough because 
$A_0$ is proportional to $T/(1+C_0)$. Set $B_0:=\int_{\mathbb R}[f_{0X}(x)]^{1-\vartheta}\,\di x<\infty$, then 
\begin{equation}\label{eq:Ecompl}
\mu_{0X}(E_\sigma^c)\leq B_0 \sigma^{\vartheta T} \lesssim \sigma^{\alpha+2} 
\end{equation}
by definition of $T$. For $b=\mp\frac{1}{2}$, introduced the densities
$$\bar h_{b,\sigma}:=\frac{f_{0X}+g_{b,\sigma}}{\|f_{0X}+g_{b,\sigma}\|_1} \qquad \text{and} \qquad \frac{ \bar h_{b,\sigma}\1_{E_\sigma}}{\|\bar h_{b,\sigma}\1_{E_\sigma}\|_1},$$
where $g_{b,\sigma}$ is as defined in \eqref{def:functiong}, we consider the decomposition
\[\begin{split}
\|g_{0Y}-g_Y\|_2^2&\lesssim
\sum_{b=\mp 1/2}\|e^{b\cdot}\{f_\varepsilon\ast [f_{0X}-\phi_\sigma\ast (T_{m,b,\sigma}f_{0X})]\}\|_2^2
\\ &\qquad\qquad\quad +\sum_{b=\mp 1/2}\|e^{b\cdot}\{f_\varepsilon\ast \phi_\sigma\ast [(T_{m,b,\sigma}f_{0X})-(f_{0X}+g_{b,\sigma})]\}\|_2^2\\
&\qquad\qquad\quad+\sum_{b=\mp 1/2}\|e^{b\cdot}\{f_\varepsilon\ast \phi_\sigma\ast [(f_{0X}+g_{b,\sigma})-\bar h_{b,\sigma}]\}\|_2^2\\
&\qquad\qquad\quad+\sum_{b=\mp 1/2}\|e^{b\cdot}\{f_\varepsilon\ast \phi_\sigma\ast[\bar h_{b,\sigma}-(\bar h_{b,\sigma}\1_{E_\sigma}/\|\bar h_{b,\sigma}\1_{E_\sigma}\|_1)]\}\|_2^2\\
&\qquad\qquad\quad+\sum_{b=\mp 1/2}
\|e^{b\cdot}\{f_\varepsilon\ast \phi_\sigma\ast[(\bar h_{b,\sigma}\1_{E_\sigma}/\|\bar h_{b,\sigma}\1_{E_\sigma}\|_1)-\mu_H]\}\|_2^2\\
&=:\sum_{r=1}^5 V_r.
\end{split}
\]
We show that each term $V_1,\,\ldots,\,V_5$ is of order $O(\sigma^{2(\alpha+2)})$.
By inequality \eqref{bound:tildehm12} of Lemma \ref{lem:contapprox}, we have $V_1\lesssim \sigma^{2(\alpha+2)}$.

\smallskip

\noindent
\emph{Study of the term $V_2$}\\[5pt]
We recall that $\varrho_b(t):=[1-\psi_b^2(t)]$, with $\psi_b(t):=-(\imath t+b)$, $t\in\mathbb R$, and
\begin{equation*}
h_{m,b,\sigma}:=\frac{1}{\gamma} \sum_{k=1}^{m-1}\frac{(-\sigma^2/2)^k}{k!}
\sum_{j=0}^{2k} \binom{2k}{j}(-b)^{2k-j}(\bar h_{0,b}\ast D^jH_\delta).
\end{equation*}
As in Lemma \ref{lem:contapprox}, for constants $0<c_\delta,\,c_h<1$, 
we take $\delta:=c_\delta\sigma$ and $h:=c_h|\log \sigma|^{-1/2}$.
We write 
$$e^{b\cdot}\Big\{f_\varepsilon\ast \phi_\sigma\ast 
\Big[M_{0X}(b)e^{-b\cdot}\gamma h_{m,b,\sigma}\1_{A_{b,\sigma}^c}+\frac{1}{2}f_{0X}\1_{A_{b,\sigma}^c}\Big]\Big\} $$ 
as 
$$ M_{0X}(b) \Big\{ e^{b\cdot}[f_\varepsilon\ast \phi_\sigma]\ast 
\Big[\Big(\gamma h_{m,b,\sigma}+\frac{1}{2}\bar h_{0,b}\Big)\1_{A_{b,\sigma}^c}\Big]\Big\}$$
so that,
using the definition of $g_{b,\sigma}$ in \eqref{def:functiong},
\[\begin{split}
V_2&=\sum_{b=\mp 1/2}\Big\|e^{b\cdot}\Big\{f_\varepsilon\ast \phi_\sigma\ast 
\Big[M_{0X}(b)e^{-b\cdot}\gamma h_{m,b,\sigma}\1_{A_{b,\sigma}^c}+\frac{1}{2}f_{0X}\1_{A_{b,\sigma}^c}\Big]\Big\}\Big\|_2^2\\
&\lesssim\sum_{b=\mp 1/2}M^2_{0X}(b)
\int_{\mathbb R}  \frac{|e^{\sigma^2 \psi_b(t)^2/2}|^2}{|\varrho_b(t)|^2 } 
| \mathcal F\{[2\gamma h_{m,b,\sigma}+\bar h_{0,b}]\1_{A_{b,\sigma}^c}\}(t)|^2\,\di t\\&\lesssim 
\sum_{b=\mp 1/2}M^2_{0X}(b)(\|\gamma h_{m,b,\sigma}\1_{A_{b,\sigma}^c}\|_1+\|\bar h_{0,b}\1_{A_{b,\sigma}^c}\|_1 )^2,
\end{split}
\] 
where we have used the facts that 
$$\|\mathcal F\{[2\gamma h_{m,b,\sigma}+\bar h_{0,b}]\1_{A_{b,\sigma}^c}\}\|_\infty  \leq 2\|\gamma h_{m,b,\sigma}\1_{A_{b,\sigma}^c}\|_1+\|\bar h_{0,b}\1_{A_{b,\sigma}^c}\|_1 $$
and
$$  \frac{|e^{\sigma^2 \psi_b(t)^2/2}|^2}{|\varrho_b(t)|^2 } \lesssim \frac{ 1 }{1+t^4}.$$ 
Finally, using the inequalities in \eqref{cm12} of Lemma \ref{lem:Asigma}, we obtain that
\[
V_2 \lesssim \sigma^{2\upsilon R}\lesssim \sigma^{2(\alpha+2)}.
\] 

\smallskip

\noindent
\emph{Study of the term $V_3$}\\[5pt]
By the inequalities in \eqref{cm12} and \eqref{eq:norm1} of Lemma \ref{lem:Asigma}, noting that $\|\mathcal F\{\bar h_{0,b}\}\|_\infty\leq 1$, we have
\[\begin{split}
V_3&=
\sum_{b=\mp 1/2}\pt{1-\frac{1}{\|f_{0X}+g_{b,\sigma}\|_1}}^2
\|e^{b\cdot}[f_\varepsilon\ast \phi_\sigma\ast (f_{0X}+g_{b,\sigma})]\|_2^2\\
&\lesssim \sigma^{2\upsilon R}\sum_{b=\mp 1/2}M_{0X}^2(b)\int_{\mathbb R}  \frac{  |e^{\sigma^2 \psi_b(t)^2/2}|^2 }{|\varrho_b(t)|^2} 
[|\mathcal F\{\bar h_{0,b}\}(t)|^2+|\mathcal F\{\gamma h_{m,b,\sigma}\}(t)|^2]\,\di t.
\end{split}
\]
Recalling that, by inequality \eqref{eq:bilaplace1}, we have $|\mathcal F\{D^j H_\delta\}(t)|\leq |t|^j |\mathcal F \{H\} (\delta t)|\leq |t|^j$, for $j=0,\,\ldots,\,2k$, we find
\[\begin{split}
\hspace*{-1cm}\int_{\mathbb R}  \frac{  |e^{\sigma^2 \psi_b(t)^2/2}|^2 }{|\varrho_b(t)|^2} 
|\mathcal F\{\gamma h_{m,b,\sigma}\}(t)|^2\,\di t &\lesssim
\sum_{k=1}^{m-1} \frac{e^{\sigma^2/4}}{(2^k k!)^2}
\int_{\mathbb R}\frac{[\sigma^2(t^2+1/4)]^{2k}}{e^{(\sigma t)^2}|\varrho_b(t)|^2}
|\mathcal F\{\bar h_{0,b}\}(t)|^2 \,\di t\\&
\lesssim 
\|\widehat{(e^{b\cdot}f_{0X}})\|_2^2\lesssim\|e^{|\cdot|/2}f_{0X}\|_2^2<\infty, 
\end{split}
\]
which implies that 
\[
V_3
\lesssim \sigma^{2\upsilon R}\lesssim \sigma^{2(\alpha+2)}.
\]

\smallskip

\noindent
\emph{Study of the term $V_4$}\\[5pt]
Taking into account that $\|f_{0X}+g_{b,\sigma}\|_1\geq 1$ (see \eqref{eq:norm1} of Lemma \ref{lem:Asigma}), we have
\[\begin{split}
V_4&\lesssim\sum_{b=\mp 1/2}\Big(\|\bar h_{b,\sigma}\1_{E_\sigma^c}\|_1^2\times
\|e^{b\cdot}\{f_\varepsilon\ast \phi_\sigma\ast(\bar h_{b,\sigma}\1_{E_\sigma}/\|\bar h_{b,\sigma}\1_{E_\sigma}\|_1)\}\|_2^2\\
&\hspace*{5.2cm}+
\|e^{b\cdot}\{f_\varepsilon\ast \phi_\sigma\ast(\bar h_{b,\sigma}\1_{E_\sigma^c})\}\|_2^2\Big)
\\
&\lesssim \sum_{b=\mp 1/2}\|(f_{0X}+g_{b,\sigma})\1_{E^c_\sigma}\|_1^2 
\int_{\mathbb R}  \frac{  |e^{\sigma^2 \psi_b(t)^2/2}|^2 }{|\varrho_b(t)|^2}|
\mathcal F\{e^{b\cdot}\bar h_{b,\sigma}\1_{E_\sigma}/\|\bar h_{b,\sigma}\1_{E_\sigma}\|_1\}(t)|^2 \,\di t\\
&\hspace*{5.2cm} +\sum_{b=\mp 1/2}
\int_{\mathbb R}  \frac{|e^{\sigma^2 \psi_b(t)^2/2}|^2 }{|\varrho_b(t)|^2}|\mathcal F\{e^{b\cdot}\bar h_{b,\sigma}\1_{E_\sigma^c}\}(t)|^2 \,\di t.
\end{split}
\]
Note that, by \eqref{eq:Ecompl},
\[\begin{split}
\|(f_{0X}+g_{b,\sigma})\1_{E^c_\sigma}\|_1&\leq \frac{3}{2}\mu_{0X}(E_\sigma^c)+M_{0X}(b) \int_{A_{b,\sigma}\cap E_\sigma^c} e^{-bx}|\gamma h_{m,b,\sigma}(x)|\,\di x\\&
\leq \frac{3B_0}{2}\sigma^{\vartheta T}+M_{0X}(b) \int_{A_{b,\sigma}\cap E_\sigma^c} e^{-bx}|\gamma h_{m,b,\sigma}(x)|\,\di x,
\end{split}\] where, as hereafter shown, 
\begin{equation}\label{eq:setcomplem}
\int_{A_{b,\sigma}\cap E_\sigma^c} e^{-bx}|\gamma h_{m,b,\sigma}(x)|\,\di x\lesssim\sigma^{\alpha+2}
\end{equation}
and 
\begin{equation}\label{eq:fourierEcompl}
\|\mathcal F\{e^{b\cdot}\bar h_{b,\sigma}\1_{E_\sigma^c}\}\|_\infty\lesssim \sigma^{\alpha+2}.
\end{equation}
It follows that 
$$V_4\lesssim \sigma^{2\vartheta T}+\sigma^{2(\alpha+2)}\lesssim \sigma^{2(\alpha+2)}.$$
We prove inequality \eqref{eq:setcomplem}. 
By H\"older's inequality, Lemma \ref{lem:HolderHm} and inequality \eqref{eq:Ecompl}, 
for $j=0,\,\ldots,\,2k$, we have
\[\begin{split}
&\int_{A_{b,\sigma}\cap E_\sigma^c}\abs{\int_{\mathbb R}e^{-b u}f_{0X}(x-u)D^jH_\delta(u)\,\di u}\di x\\
&\hspace*{0.5cm}\leq
\int_{A_{b,\sigma}\cap E_\sigma^c}\abs{\int_{\mathbb R}[e^{-b u}f_{0X}(x-u)-f_{0X}(x)]D^jH_\delta(u)\,\di u}\di x \\
&\hspace*{7cm}+\pt{\int_{\mathbb R}|D^jH_\delta(u)|\,\di u}\int_{A_{b,\sigma}\cap E_\sigma^c}f_{0X}(x)\,\di x \\
&\hspace*{0.5cm}\lesssim C_j\delta^{-j+\upsilon}\int_{A_{b,\sigma}\cap E_\sigma^c}[L_0(x)+f_{0X}(x)]\,\di x+
C_{0,j}\mu_{0X}(A_{b,\sigma}\cap E_\sigma^c)\\
&\hspace*{0.5cm}\lesssim C_j\delta^{-j+\upsilon}\int_{A_{b,\sigma}\cap E_\sigma^c}[f_{0X}(x)]^{1/R+(1-1/R)}
\pt{\frac{L_0}{f_{0X}}(x)}\,\di x+(C_{0,j}+C_j\delta^{-j+\upsilon}) \mu_{0X}(E_\sigma^c)\\
&\hspace*{0.5cm}\lesssim \delta^{-j+\upsilon}
\pt{\int_{A_{b,\sigma}\cap E_\sigma^c}f_{0X}(x)\pt{\frac{L_0}{f_{0X}}(x)}^R\,\di x}^{1/R}
[\mu_{0X}(E_\sigma^c)]^{1-1/R}+ \sigma^{\vartheta T-j+\upsilon}\\
&\hspace*{0.5cm}\lesssim \delta^{-j+\upsilon}
\pt{\int_{\mathbb R}f_{0X}(x)\pt{\frac{L_0}{f_{0X}}(x)+1}^R\,\di x}^{1/R}\sigma^{\vartheta T(1-1/R)}
+\sigma^{\vartheta T-j+\upsilon}.
\end{split}\]
Consequently,
\[
\int_{A_{b,\sigma}\cap E_\sigma^c} e^{-bx} |\gamma h_{m,b,\sigma}(x)|\,\di x\lesssim
\sigma^{\upsilon+\vartheta T(1-1/R)}+\sigma^{\upsilon+\vartheta T}\lesssim \sigma^{\vartheta T(1-1/R)}+\sigma^{\vartheta T}\lesssim \sigma^{\alpha +2}
\]
by choosing $T\geq [(\alpha+2)R]/[\vartheta(R-1)]$. With this choice, the condition $T\geq[(\alpha+2)/\vartheta]$ is satisfied.
Thus, $\|(f_{0X}+g_{b,\sigma})\1_{E^c_\sigma}\|_1\lesssim\sigma^{\alpha+2}$.
Inequality \eqref{eq:fourierEcompl} follows from the tail condition on $f_{0X}$ of Assumption \ref{ass:twicwtailcond} 
and the definition of the set $E_\sigma^c$.

\smallskip

\noindent
\emph{Study of the term $V_5$}\\[5pt]
Recalling that $\mathcal B$ stands for the bilateral Laplace transform operator, we have
\[
V_5\lesssim\sum_{b=\mp 1/2}
\int_{\mathbb R}  \frac{|e^{\sigma^2 \psi_b(t)^2/2}|^2 }{|\varrho_b(t)|^2} |[\mathcal B\{\bar h_{b,\sigma}\1_{E_\sigma}/\|\bar h_{b,\sigma}\1_{E_\sigma}\|_1\}-\mathcal B\{\mu_H\}](\psi_b(t))|^2\,\di t.
\]
For $M>0$, split the integral domain into $|t|\leq M$ and $|t|>M$
and let the corresponding terms be denoted by $V_5^{(1)}$ and $V_5^{(2)}$.
Let $\mu_H$ be a discrete probability measure on $E_\sigma\subseteq [-a_\sigma,\,a_\sigma]$ satisfying, for $b=\mp\frac{1}{2}$, the constraints
\begin{equation*}\label{eq:c22}
\int_{E_\sigma} u^j\mu_H(\di u)=
\int_{E_\sigma} u^j\frac{\bar h_{b,\sigma}(u)}{\|\bar h_{b,\sigma}\1_{E_\sigma}\|_1}\,\di u, \quad\mbox{$j=0,\,\ldots,\,J-1$,}
\end{equation*}
with $J=\lceil\eta ea_\sigma M\rceil$ for some $\eta>1$, together with
\begin{equation}\label{eq:c23}
\int_{E_\sigma}e^{bu}\mu_H(\di u)=\int_{E_\sigma} e^{bu}
\frac{\bar h_{b,\sigma}(u)}{\|\bar h_{b,\sigma}\1_{E_\sigma}\|_1}\,\di u,
\end{equation}
where the integral on the right-hand side of \eqref{eq:c23} is finite because 
$\int_{E_\sigma} e^{bu}
\bar h_{b,\sigma}(u)\,\di u\leq \int_{\mathbb R} e^{bu}
\bar h_{b,\sigma}(u)\,\di u\lesssim 
M_{0X}(b)[1+\int_{\mathbb R}\gamma h_{m,b,\sigma}(u)\,\di u]=
M_{0X}(b)\{1+\gamma [1+O(\sigma^{2(m-1)})]\}$
by relationship \eqref{eq:bound} of Lemma \ref{lem:contapprox}. Thus, 
$$\int_{E_\sigma} e^{bu}
\bar h_{b,\sigma}(u)\,\di u=O(1).$$
By the lower bound inequality in \eqref{eq:norm1} of Lemma \ref{lem:Asigma} and the previously proven fact that 
$\|(f_{0X}+g_{b,\sigma})\1_{E_\sigma^c}\|_1\lesssim \sigma^{\alpha+2}$, we have
$\|\bar h_{b,\sigma}\1_{E_\sigma}\|_1=1-\|\bar h_{b,\sigma}\1_{E_\sigma^c}\|_1\gtrsim1-\|(f_{0X}+g_{b,\sigma})\1_{E_\sigma^c}\|_1
\gtrsim1-\sigma^{\alpha+2}$. Therefore,
$$\|\mathcal B\{\bar h_{b,\sigma}\1_{E_\sigma}/\|\bar h_{b,\sigma}\1_{E_\sigma}\|_1\}(\psi_b)\|_\infty\leq
\|e^{b\cdot}\bar h_{b,\sigma}\1_{E_\sigma}\|_1/\|\bar h_{b,\sigma}\1_{E_\sigma}\|_1
\lesssim \int_{E_\sigma} e^{bu}
\bar h_{b,\sigma}(u)\,\di u.$$
Then, using Lemma \ref{lem:diseg} with $r=J$, by the inequality $J!\geq (J/e)^J$, we have 
\[\begin{split}
V_5^{(1)}&:=\sum_{b=\mp 1/2}
\int_{|t|\leq M}  \frac{|e^{\sigma^2 \psi_b(t)^2/2}|^2 }{|\varrho_b(t)|^2} |[\mathcal B\{\bar h_{b,\sigma}\1_{E_\sigma}/\|\bar h_{b,\sigma}\1_{E_\sigma}\|_1\}-\mathcal B\{\mu_H\}](\psi_b(t))|^2\,\di t\\
&\lesssim \frac{a_\sigma^{2J}}{(J!)^2}
\int_{0}^M t^{2(J-2)}\,\di t \\
&\lesssim M^{-2(\alpha+2)} \times \frac{a_\sigma^{2J}}{(J!)^2}\times
\frac{M^{2(J+\alpha)+1}}{2J-3}\lesssim  M^{-2(\alpha+2)} \pt{\frac{ea_\sigma M}{J}}^{2J+1} M^{2\alpha}\lesssim M^{-2(\alpha+2)}
\end{split}\]  
because $(ea_\sigma M/J)^{2J+1} M^{2\alpha}<e^{-2J(\log\eta)}M^{2\alpha}<1$. Choosing $M$ so that  
$(\sigma M)^2\geq (2\alpha+1)|\log\sigma|$, equivalently, $M\geq\sigma^{-1}[(2\alpha+1)|\log\sigma|]^{1/2}$, and recalling that $|e^{\sigma^2\psi_b^2(t)/2}|^2=O(e^{-(\sigma t)^2})$, we have
\[\begin{split}
V_5^{(2)}&:=\sum_{b=\mp 1/2}
\int_{|t|>M}  \frac{|e^{\sigma^2 \psi_b(t)^2/2}|^2 }{|\varrho_b(t)|^2} |[\mathcal B\{\bar h_{b,\sigma}\1_{E_\sigma}/\|\bar h_{b,\sigma}\1_{E_\sigma}\|_1\}-\mathcal B\{\mu_H\}](\psi_b(t))|^2\,\di t\\
&\lesssim
e^{-(\sigma M)^2}
\int_{|t|>M}t^{-4}
\,\di t
\lesssim
e^{-(\sigma M)^2}M^{-3}\lesssim 
\sigma^{2\alpha+1} M^{-3}
\lesssim \sigma^{2(\alpha+2)}.
\end{split}
\]
Therefore, $$V_5\lesssim V_5^{(1)}+V_5^{(2)}\lesssim\sigma^{2(\alpha+2)}.$$
Conclude that $\|g_Y-g_{0Y}\|_2^2\lesssim\sum_{r=1}^5 V_r\lesssim \sigma^{2(\alpha+2)}$. The assertion follows.
\end{proof}

\smallskip


\section{Technical lemmas for adaptive posterior contraction rates for Dirichlet Laplace-normal mixtures}

\begin{lem}\label{lem:H}
For $r\geq 0$, $a\in \mathbb R$ and $j \in \mathbb N_0$, there exists 
a constant $C_{r,j}< \infty$ such that, for $h=o(1)$  and $\delta=o(h)$,
\begin{equation}\label{eq:constant}
  \int_{\mathbb R} |x|^r e^{a\delta x} |H^{(j)}(x) |\,\di x \leq C_{r,j} .
\end{equation}
\end{lem}

\begin{proof}
Recalling that $H(x)=(2\pi)^{-1}\hat \tau(x) \reallywidehat{\phi_h}(x)=(2\pi)^{-1}\hat \tau(x) e^{-(h x)^2/2}$, $x\in\mathbb R$,
we have 
$$H^{(j)}(x) =\frac{1}{2\pi}\sum_{i=0}^j\binom{j}{i}\hat\tau^{(i)}(x)D^{j-i}
\reallywidehat{\phi_h}(x),\quad x\in\mathbb R,$$
where $D^{j-i}
\reallywidehat{\phi_h}(x)=D^{j-i}e^{-(h x)^2/2}$ is a linear combination of terms of the form 
$$\reallywidehat{\phi_h}(x)(-1)^{j_1}h^{2j_2}x^{j_3},$$ where $0\leq j_1,\,j_2,\,j_3\leq (j-i)$.
Note that $ e^{a\delta x}\reallywidehat{\phi_h}(x)e^{ -(hx-a\delta/h)^2/2}e^{(a\delta/h)^2/2} \leq e^{(a\delta/h)^2/2}$,
where $e^{(a\delta/h)^2/2}= 1 + o(1)$ because $(\delta/h)=o(1)$. 
Then, by condition \eqref{eq:poldec}, for $\nu>(r+j+1)$ and $0\leq j_1,\,j_2,\,j_3\leq (j-i)$,
$$|x|^{r} e^{a\delta x}|\hat\tau^{(i)}(x)|\reallywidehat{\phi_h}(x)|x|^{j_3}\lesssim
|x|^{r+j_3} |\hat\tau^{(i)}(x)|,$$
where the function on the right-hand side of the last inequality is integrable. The assertion follows.
\end{proof}

\begin{lem}\label{lem:HolderHm}
Suppose that $f_{0X}$ satisfies the local H\"{o}lder condition \eqref{eq:holderf0} of Assumption \ref{ass:smoothf0} 
with $0<\upsilon\leq1$ and $L_0\in L^1(\mathbb R)$.
For every $b\in\mathbb R$ and $j\in\mathbb N_0$, if $h=o(1)$ and $\delta=o(h)$, then
\begin{equation}\label{eq:f+L}
\int_{\mathbb R}|[e^{-bu}f_{0X}(x-u) - f_{0X}(x) ] D^jH_\delta(u)|\,\di u
\leq \frac{C_j}{\delta^j}\delta^\upsilon [L_0(x)+f_{0X}(x)],\quad x\in\mathbb R,
\end{equation}
where $C_j:=(3C_{\upsilon,j}\vee C_{1,j})>0$, with $C_{\upsilon,j}$ as in \eqref{eq:constant}.
\end{lem}

\begin{proof}
Let $x\in\mathbb R$ be fixed.
By Lemma \ref{lem:diseg} and the local H\"{o}lder condition \eqref{eq:holderf0} of Assumption \ref{ass:smoothf0}, 
we have
\[
\begin{split}
&\delta^j\int_{\mathbb R} |[e^{-bu}f_{0X}(x-u) - f_{0X}(x) ] D^jH_\delta(u)|\,\di u\\
&\hspace*{1cm}=
\int_{\mathbb R} |[e^{-b\delta z}f_{0X}(x-\delta z) - f_{0X}(x) ] H^{(j)}(z)|\,\di z\\
&\hspace*{1cm}\leq 
\int_{\mathbb R}  |e^{-b\delta z}-1||f_{0X}(x-\delta z)-f_{0X}(x)||H^{(j)}(z)|\,\di z \\
&\hspace*{7.1cm}+ f_{0X}(x)
\int_{\mathbb R}  |e^{-b\delta z}-1| |H^{(j)}(z)|\,\di z\\
&\hspace*{7.1cm} + \int_{\mathbb R}|f_{0X}(x-\delta z) - f_{0X}(x)||H^{(j)}(z)|\,\di z\\
 & \hspace*{1cm}\leq 
3 \int_{\mathbb R}|f_{0X}(x-\delta z) - f_{0X}(x)||H^{(j)}(z)|\,\di z +
f_{0X}(x)
\int_{\mathbb R}  |e^{-b\delta z}-1| |H^{(j)}(z)|\,\di z \\
&\hspace*{1cm}\leq3\delta^{\upsilon}L_0(x) \int_{\mathbb R}|z|^{\upsilon}|H^{(j)}(z)|\,\di z
+
b\delta f_{0X}(x) \int_{\mathbb R}|z||H^{(j)}(z)|\,\di z\\
 & \hspace*{1cm}
\leq 
3\delta^{\upsilon} C_{\upsilon,j}L_0(x) + b\delta C_{1,j} f_{0X}(x)
<C_j\delta^{\upsilon} [L_0(x)+f_{0X}(x)].
\end{split}
\]
Inequality \eqref{eq:f+L} follows.
\end{proof}

\begin{lem} \label{lem:Asigma}
For $m\in\mathbb N$, $b=\mp\frac{1}{2}$ and $\sigma=o(1)$, let the set $A_{b,\sigma}$ 
be defined as in \eqref{def:set}.
Under Assumptions \ref{ass:twicwtailcond} and 
\ref{ass:smoothf0} on $f_{0X}$, the latter with $0<\upsilon\leq1$, $L_0\in L^1(\mathbb R)$ and 
any $R\geq1$,
there exists a constant $\bar C_{m}>0$, depending on $m$ and $\upsilon$, such that, for $\sigma$ small enough,
\begin{equation}\label{eq:inclus}
\forall\,b=\mp\frac{1}{2},\quad A_{b,\sigma}^c\subseteq B_\sigma,
\end{equation}
with $B_\sigma:=
\{x\in\mathbb R:\,[L_0(x)+f_{0X}(x)]> \bar C_{m}^{-1}\sigma^{-\upsilon} f_{0X}(x)\}$.
Furthermore, there exist constants $C_R,\,D_R,\,S_R>0$, depending on $m$, $\upsilon$ and $R$, so that
\begin{equation}\label{cm12}
\|\bar h_{0,b}\1_{A_{b,\sigma}^c}\|_1< C_R \sigma^{\upsilon R},\quad
\|\gamma h_{m,b,\sigma}\1_{A_{b,\sigma}^c}\|_1 <D_R\sigma^{\upsilon R}
\end{equation}
and the function $f_{0X}+g_{b,\sigma}$, with $g_{b,\sigma}$ as defined in \eqref{def:functiong}, which is non-negative, has 
\begin{equation}\label{eq:norm1}
1\leq\|f_{0X}+g_{b,\sigma}\|_1\leq1+S_R\sigma^{\upsilon R}.
\end{equation}
\end{lem}

\begin{proof}
Let $b$ be fixed. We begin by proving the inclusion in \eqref{eq:inclus}. 
Assume that $x \in  A_{b,\sigma}^c$, \emph{i.e.},  $\gamma h_{m,b,\sigma}(x) \leq - \bar h_{0,b}(x)/2$. Recall that 
$$\gamma h_{m,b,\sigma}= \sum_{k=1}^{m-1}\frac{ (-\sigma^2/2)^k}{k!}  \sum_{j=0}^{2k} \binom{2k}{j}(-b)^{2k-j}
(\bar h_{0,b}\ast D^jH_\delta),$$
where, as in Lemma \ref{lem:contapprox}, for constants $0<c_\delta,\,c_h<1$, we 
take $\delta:=c_\delta\sigma$ and $h:=c_h|\log \sigma|^{-1/2}$.
Note that, by relationship \eqref{eq:bilaplace1},
\[
\begin{split}
\int_{\mathbb R}D^jH_\delta(u)\,\di u&=\pt{\int_{\mathbb R}H(x)\,\di x}\1_{\{0\}}(j)\\
&=
\pt{\frac{1}{2\pi}\int_{\mathbb R}\reallywidehat{(\tau\ast\phi_h)}(x)\,\di x}\1_{\{0\}}(j)=(\tau\ast \phi_h)(0)\1_{\{0\}}(j)\leq 1.
\end{split}
\]
Then,
 \begin{equation*}\label{eq:sufcond}
\begin{split}
 \sum_{k=1}^{m-1} & \frac{ (-\sigma^2/2)^k}{k!}  \sum_{j=0}^{2k} \binom{2k}{j}(-b)^{2k-j}
\int_{\mathbb R} [ \bar h_{0,b}(x-u) -  \bar h_{0,b}(x)] D^jH_\delta(u)\, \di u \\
&\qquad = \gamma h_{m,b,\sigma}(x) - \bar h_{0,b} (x)\sum_{k=1}^{m-1}\frac{ (-\sigma^2/2)^k}{k!} \sum_{j=0}^{2k} \binom{2k}{j}(-b)^{2k-j}\int_{\mathbb R}  D^jH_\delta(u)\, \di u \\
&\qquad\leq - \bar h_{0,b}(x)\left( \frac{ 1 }{ 2 } - (\tau\ast\phi_h)(0)\frac{(b\sigma)^2  }{  2 } \sum_{k=0}^{m-2}\frac{ [-(b\sigma)^2/2]^k}{  (k+1)!}\right)\\ &\qquad \leq  - \frac{ \bar h_{0,b}(x)  }{ 2 }[ 1-(b\sigma)^2]<  - \frac{ \bar h_{0,b} (x) }{ 4 }.
\end{split}
\end{equation*}
For $\sigma$ small enough, by Lemma \ref{lem:HolderHm},
\[\begin{split}
&\hspace*{-0.1cm}\abs{\sum_{k=1}^{m-1}\frac{ (-\sigma^2/2)^k}{k!}  \sum_{j=0}^{2k} \binom{2k}{j}(-b)^{2k-j} \int_{\mathbb R} [ \bar h_{0,b}(x-u) -  \bar h_{0,b}(x)]  D^jH_\delta(u)\, \di u}\\
&\hspace*{0.3cm}\leq\frac{1}{M_{0X}(b)}\sum_{k=1}^{m-1}\frac{(\sigma^2/2)^k}{k!}  \sum_{j=0}^{2k} \binom{2k}{j}|b|^{2k-j} e^{bx}\int_{\mathbb R} |e^{-bu} f_{0X}(x-u) - f_{0X}(x)||  D^jH_\delta(u)|\, \di u\\
&\hspace*{0.3cm}<\frac{1}{M_{0X}(b)}\pt{\sum_{k=1}^{m-1}\frac{[(1+|b|c_\delta\sigma)^2/(2c_\delta^2)]^k}{k!}\max_{0\leq j\leq 2k}C_j}
\sigma^{\upsilon}e^{bx}[L_0(x)+f_{0X}(x)]\\
&\hspace*{0.3cm}<\frac{1}{M_{0X}(b)}\underbrace{
\pt{ \sum_{k=1}^{m-1}\frac{(2/c_\delta^2)^k}{k!}\max_{0\leq j\leq 2k}C_j}
}_{=:\tilde C_{m}}
\sigma^{\upsilon}e^{bx}[L_0(x)+f_{0X}(x)],
\end{split}\]
where $0<\tilde C_{m}<\infty$. 
Then, for $\bar C_m:=4\tilde C_m$, we have $A_{b,\sigma}^c  \subseteq B_\sigma$.

We prove the inequalities in \eqref{cm12}. Concerning the first one,
$$\int_{A_{b,\sigma}^c} \bar h_{0,b}(x)\,\di x < \sigma^{\upsilon R}
\frac{\bar C_{m}^R}{M_{0X}(b)} \int_{B_\sigma} e^{bx} f_{0X}(x) \pt{\frac{L_0}{f_{0X}}(x)+1}^R\di x \leq C_R\sigma^{\upsilon R},$$
where 
$$C_R:=\frac{\bar C_{m}^R}{[M_{0X}(-1/2)\wedge M_{0X}(1/2)]}
\int_{\mathbb R} e^{|x|/2} f_{0X}(x)\pt{\frac{L_0}{f_{0X}}(x)+1}^R\di x< \infty$$
by condition \eqref{eq:ratiofo} and Assumption \ref{ass:twicwtailcond}.
As for the second inequality, from previous computations, for every $j\in\mathbb N_0$, we have
\begin{equation*}
\begin{split} 
&\int_{A_{b,\sigma}^c}| (\bar h_{0,b}\ast D^jH_\delta)(x)|\,\di x\\
& \hspace*{2cm}<\delta^{-j}
\pt{\int_{A_{b,\sigma}^c}\int_{\mathbb R}|\bar h_{0,b}(x-\delta u) -\bar h_{0,b}(x)||H^{(j)}(u)|\,\di u\,\di x
+ C_{0,j} C_R \sigma^{\upsilon R}}\\
&\hspace*{2cm} \leq \delta^{-j}
\pt{\frac{C_j\delta^{\upsilon}}{M_{0X}(b)}\int_{A_{b,\sigma}^c} e^{bx}[L_0(x)+f_{0X}(x)]\, \di x+C_{0,j}C_R\sigma^{\upsilon R}} \\
&\hspace*{2cm}<  \delta^{-j}\pt{\frac{C_jc_\delta^\upsilon\bar C_{m}^{R-1}}{M_{0X}(b)}
\int_{A_{b,\sigma}^c}e^{bx}f_{0X}(x)\pt{\frac{L_0}{f_{0X}}(x)+1}^R\di x+ C_{0,j}C_R} \sigma^{\upsilon R}\\
&\hspace*{2cm}\leq  \delta^{-j}C_R\pt{\frac{C_jc_\delta^\upsilon}{\bar C_{m}}
+ C_{0,j}} \sigma^{\upsilon R},
\end{split}
\end{equation*}
which, defined the constant
$$D_R:=C_R\pt{\frac{c_\delta^\upsilon}{\bar C_{m}}
+1} \sum_{k=1}^{m-1}\frac{(2/c_\delta^2)^k}{k!}\max_{0\leq j\leq 2k}(C_j\vee C_{0,j}),$$
implies that 
$\|\gamma h_{m,b,\sigma}\1_{A_{b,\sigma}^c}\|_1<D_R\sigma^{\upsilon R}$.

To prove the last part of the lemma, we begin by noting that 
$$f_{0X}+g_{b,\sigma}=[f_{0X}+M_{0X}(b)e^{-bx}\gamma h_{m,b,\sigma}]\1_{A_{b,\sigma}}+\frac{1}{2}f_{0X}\1_{A_{b,\sigma}^c}>
\frac{1}{2}f_{0X}\geq0$$
and
\[M_{0X}(b)\int_{\mathbb R}e^{-bx}\gamma h_{m,b,\sigma}(x)\,\di x=0.\]
In fact, since by Lemma \ref{lem:H} we have $\int_{\mathbb R}e^{-b\delta x}H(x)\,\di x<\infty$ because $(\delta/h)=o(1)$, it holds that
\[\begin{split}
&M_{0X}(b)\int_{\mathbb R}e^{-bx}\gamma h_{m,b,\sigma}(x)\,\di x\\
&\hspace*{2.5cm}
=\sum_{k=1}^{m-1}
\frac{(-\sigma^2/2)^k}{k!}\sum_{j=0}^{2k}\binom{2k}{j}(-b)^{2k-j}\int_{\mathbb R}e^{-bu}D^jH_\delta(u)\,\di u\\
&\hspace*{2.5cm}=\pt{\int_{\mathbb R}e^{-b\delta x}H(x)\,\di x}\sum_{k=1}^{m-1}
\frac{(-\sigma^2/2)^k}{k!}\underbrace{\sum_{j=0}^{2k}\binom{2k}{j}(-b)^{2k-j}b^j}_{(-b+b)^{2k}=0}=0.
\end{split}\]
Then, since
$g_{b,\sigma}=M_{0X}(b)e^{-b\cdot}\gamma h_{m,b,\sigma}-[M_{0X}(b)e^{-b\cdot}\gamma h_{m,b,\sigma}+(f_{0X}/2)]\1_{A_{b,\sigma}^c}$,
we have
\[
\begin{split}
\int_{\mathbb R}(f_{0X}+g_{b,\sigma})(x)\,\di x&=1+\underbrace{\int_{\mathbb R}M_{0X}(b)e^{-bx}\gamma h_{m,b,\sigma}(x)\,\di x}_{=0}\\
&\hspace*{2cm}
-\int_{A_{b,\sigma}^c}
\pq{M_{0X}(b)e^{-bx}\gamma h_{m,b,\sigma}(x)+\frac{1}{2}f_{0X}(x)}\,\di x\\&= 1-\int_{A_{b,\sigma}^c}
\Big[\underbrace{M_{0X}(b)e^{-bx}\gamma h_{m,b,\sigma}(x)}_{\leq -\frac{1}{2}f_{0X}(x)}+\frac{1}{2}f_{0X}(x)\Big]\,\di x\geq 1.
\end{split}
\]
On the other side, using Lemma \ref{lem:HolderHm} and reasoning as in the first part of the present lemma,
\[
\begin{split}
\int_{\mathbb R}(f_{0X}+g_{b,\sigma})(x)\,\di x&=1
-\int_{A_{b,\sigma}^c}
\pq{M_{0X}(b)e^{-bx}\gamma h_{m,b,\sigma}(x)+\frac{1}{2}f_{0X}(x)}\,\di x\\
&=1-\frac{1}{2}\mu_{0X}(A_{b,\sigma}^c)
-M_{0X}(b)\int_{A_{b,\sigma}^c}
e^{-bx}\gamma h_{m,b,\sigma}(x)\,\di x\\
&\leq
1+M_{0X}(b)
\int_{A_{b,\sigma}^c}
e^{-bx}|\gamma h_{m,b,\sigma}(x)|\,\di x\\
&\leq 1+\underbrace{[M_{0X}(-1/2)\vee M_{0X}(1/2)]D_R}_{=:S_R}\sigma^{\upsilon R}.
\end{split}
\]
Conclude that $1\leq\|f_{0X}+g_{b,\sigma}\|_1\leq1+S_R\sigma^{\upsilon R}$. The proof is thus complete.
\end{proof}

\begin{rmk}
\emph{
Although in condition \eqref{eq:holderf0} of Assumption \ref{ass:smoothf0} the constant 
$R\geq(2m/\upsilon)$, for the smallest integer $m\geq[2\vee(\alpha+2)/2]$, in Lemma \ref{lem:Asigma}
we have that $R$ can be any real greater than or equal to $1$.
}
\end{rmk}


\section{Proof of Theorem \ref{thm:lower bound} on rate lower bounds}\label{sec:lower bounds}
\begin{proof}[Proof of Theorem \ref{thm:lower bound}]
For clarity of exposition, we distinguish the cases where $d=1$ and $d\geq2$.

\vspace{0.2cm}  

\noindent 
$\bullet$ \emph{Case $d=1$}\\[4pt]
The proof develops along the lines of Theorem 3 in \cite{dedecker:michel:2013}, pp. 281--285, and of Theorem 4.1 in \cite{dedecker2015}, pp. 246--248. It uses intermediate results from \cite{fan1991}, pp. 1267--1271, from Theorem 1 in \cite{comtelacourihp}, pp. 577 and 590--594, and 
from Theorem 2.10 in \cite{MR3449779}, p. 10 and pp. 34--36.

We consider mixing distributions belonging to the class $\mathcal D_1=\mathcal P_1(\mathbb R,\,M)\cap \mathcal S(\alpha,\,L)$.
We begin by defining a finite family of Lebesgue absolutely continuous probability measures on $\mathbb R$, 
with uniformly bounded first moments, whose densities belong to $\mathcal S(\alpha,\,L)$.
For $1<r<\frac{3}{2}$, we define the density
\begin{equation}\label{eq:densf0}
f_{0,r}(x):=C_r(1+x^2)^{-r}, \quad x\in\mathbb R.
\end{equation}
Let $H(\cdot)$ be the kernel function on $\mathbb R$ defined in \cite{fan1991}, p. 1268, which is such that, in particular,\\[-17pt]
\begin{itemize}
\item $H$ is real, bounded and continuous,
\item $H(0)\neq 0$, $\int_{\mathbb R}H(x)\,\di x=0$ and $\int_0^1|H^{(-1)}(x)|\,\di x>0$, where $H^{(-1)}(x):=\int_{-\infty}^xH(u)\,\di u$ is a primitive of $H$,
\item $|H(x)|\leq c(1+x^2)^{-\delta}$, $x\in\mathbb R$, with $\delta>\frac{3}{2}$,
\item $\hat H(t)=0$ (hence also $\hat H^{(1)}(t)=\hat H^{(2)}(t)=0$), $|t|\notin[1,\,2]$.
\end{itemize}
Let $b_n:=([n^{1/[2(\alpha+\beta)+1]}]\vee 1)$, where $[\cdot]$ denotes the integer part. 
For $\boldsymbol{\theta}\in\{0,\,1\}^{b_n}$ and $C>0$, let
\begin{equation}\label{eq:targetf}
f_{\boldsymbol{\theta}}(x):=f_{0,r}(x)+C b_n^{-\alpha}\sum_{s=1}^{b_n}\theta_s H(b_n(x-x_{s,n})),\quad x\in\mathbb R,
\end{equation}
where $x_{s,n}:=(s-1)/b_n$. Defined the measure 
$\mu_{\boldsymbol{\theta}}:=f_{\boldsymbol{\theta}}\,\di\lambda$, 
we show that 
$$\{\mu_{\boldsymbol{\theta}}:\,{\boldsymbol{\theta}}\in\{0,\,1\}^{b_n}\}\subseteq\mathcal D_1.$$

\smallskip
\noindent\emph{The function $f_{\boldsymbol{\theta}}$ is a density}\\[5pt]
Since $f_{0,r}$ is a density and $\int_{\mathbb R}H(x)\,\di x=0$, we have $\int_{\mathbb R}f_{\boldsymbol{\theta}}(x)\,\di x=1$.
To ensure that $f_{\boldsymbol{\theta}}\geq0$, it is sufficient to show that $|f_{\boldsymbol{\theta}}-f_{0,r}|\leq f_{0,r}$. 
In fact, by Lemma 7 in \cite{fan-truong}, pp. 1923--1924, since $\delta>(r\vee \frac{1}{2})$, there exists $\tilde C>0$ such that, for $n$ large enough, we have
\[
f_{0,r}^{-1}(x)|f_{\boldsymbol{\theta}}(x)-f_{0,r}(x)|=CC_r^{-1}(1+x^2)^r b_n^{-\alpha}\sum_{s=1}^{b_n}\theta_s |H(b_n(x-x_{s,n}))|\leq C C_r^{-1} \tilde C b_n^{-\alpha}\leq 1.
\]
\noindent\emph{The probability measure $\mu_{\boldsymbol{\theta}}\in\mathcal P_1(\mathbb R,\,M)$}\\[5pt]
Let $\mu_{0,r}:=f_{0,r}\,\di\lambda$. Since $r>1$ we have $M_1(\mu_{0,r}):=\int_{\mathbb R}|x| f_{0,r}(x)\,\di x<\infty$. For $n$ large enough, since $\delta>1$, we have
\[\begin{split}
M_1(\mu_{\boldsymbol{\theta}})&=\int_{\mathbb R}|x| f_{\boldsymbol{\theta}}(x)\,\di x\\&\leq\int_{\mathbb R}|x| f_{0,r}(x)\,\di x+
C b_n^{-\alpha}\sum_{s=1}^{b_n}\int_{\mathbb R}|x| |H(b_n(x-x_{s,n}))|\,\di x\\
&\leq M_1(\mu_{0,r})+cC b_n^{-(\alpha+1)}\sum_{s=1}^{b_n} 
\int_{\mathbb R}\frac{1}{(1+u^2)^{\delta}}\pt{\frac{|u|}{b_n}+x_{s,n}}\,\di u\\
&= M_1(\mu_{0,r})+cC b_n^{-(\alpha+1)}\int_{\mathbb R}\frac{1}{(1+u^2)^{\delta}}\pt{|u|+\sum_{s=1}^{b_n}x_{s,n}}\,\di u\\
&<M_1(\mu_{0,r})+\int_{\mathbb R}\frac{|u|+1}{(1+u^2)^{\delta}}\,\di x=:M_{0,r,\delta}<\infty.
\end{split}
\]
Thus, $\mu_{\boldsymbol{\theta}}\in\mathcal P_1(\mathbb R,\,M)$ for every $M\geq M_{0,r,\delta}$.\\

\smallskip

\noindent\emph{The density $f_{\boldsymbol{\theta}}\in\mathcal S(\alpha,\,L)$}\\[5pt]
By Lemma 4 in \cite{comtelacourihp}, p. 590, we have 
$\hat f_{0,r}(t)=\exp{(-|t|^{2r-1})}$, $t\in\mathbb R$, where $0<(2r-1)<2$.
Let $L_{0,r,\alpha}:=\int_{\mathbb R}(1+t^2)^{\alpha}|\hat f_{0,r}(t) |^2\,\di t<\infty$.
For $n$ large enough, we have
\[
\begin{split}
\int_{\mathbb R}(1+t^2)^{\alpha}|\hat f_{\boldsymbol{\theta}}(t)|^2\,\di t
&\leq 2 
\int_{\mathbb R}(1+t^2)^{\alpha}\big[|\hat f_{0,r}(t) |^2 +
C^2b_n^{-2(1+\alpha)}|\hat H(t/b_n)|^2\big]\,\di t\\
&\leq 2\big[L_{0,r,\alpha} +C^2b_n^{-1}(4+b_n^{-2})^{\alpha}
\|\hat H\|_2^2\big]< 2 L_{0,r,\alpha}+1=:L_0.
\end{split}
\]
Thus, $f_{\boldsymbol{\theta}}\in\mathcal S(\alpha,\,L)$ for every $L\geq L_0$.


\smallskip

The rest of the proof proceeds along the lines of Theorem 3 in \cite{dedecker:michel:2013}, pp. 281--285. Therefore, we only sketch it. Let $\hat \mu_n$ be an estimator of $\mu$ based on the sample $\Data$. 
Let $\tilde{\boldsymbol{\theta}}$ be a random vector whose components are i.i.d. Bernoulli random variables $\tilde\theta_1,\,\ldots,\,\tilde\theta_{b_n}$, with $P(\tilde\theta_s=1)=P(\tilde\theta_s=0)=\frac{1}{2}$, for $s\in[b_n]$. We have
\[\begin{split}
\sup_{\mu\in\mathcal D_1}\mathbb E^n_{(\mu\ast\mu_\varepsilon)}W_1(\hat\mu_n,\,\mu)
&\geq \sup_{\boldsymbol{\theta}\in\{0,\,1\}^{b_n}}
\mathbb E^n_{(\mu_{\boldsymbol{\theta}}\ast\mu_\varepsilon)}W_1(\hat\mu_n,\,\mu_{\boldsymbol{\theta}})\\
&\geq \inf_{{\hathat{\mu}}_n}\sup_{\boldsymbol{\theta}\in\{0,\,1\}^{b_n}}\mathbb E^n_{(\mu_{\boldsymbol{\theta}}\ast\mu_\varepsilon)}W_1(\hathat{\mu}_n,\,\mu_{\boldsymbol{\theta}})\\
&\geq \inf_{{\hathat{\mu}}_n}\mathbb E\mathbb 
E^n_{(\mu_{\tilde{\boldsymbol{\theta}}}\ast\mu_\varepsilon)}W_1(\hathat{\mu}_n,\,\mu_{\tilde{\boldsymbol{\theta}}})\\
&= \inf_{{\hathat{\mu}}_n}\int_{\mathbb R}\mathbb E\mathbb
E^n_{(\mu_{\tilde{\boldsymbol{\theta}}}\ast\mu_\varepsilon)}|\hathat{F}_n(x)-F_{\tilde{\boldsymbol{\theta}}}(x)|\,\di x,
\end{split}\]
where the infimum is taken over all estimators ${\hathat{\mu}}_n$ of $\mu_{\boldsymbol{\theta}}$, 
the expectation $\mathbb E$ is taken with respect to the distribution of $\tilde{\boldsymbol{\theta}}$ 
and $\hathat{F}_n,\,F_{\tilde{\boldsymbol{\theta}}}$ are the distribution functions of 
$\hathat{\mu}_n$ and $\mu_{\tilde{\boldsymbol{\theta}}}$, respectively. For $\boldsymbol{\theta}\in\{0,\,1\}^{b_n}$ and $s\in[b_n]$, we define the densities
$$f_{\boldsymbol{\theta},s,u}:=f_{(\theta_1,\ldots,\theta_{s-1},u,\theta_{s+1},\ldots,\theta_{b_n})},\quad u=0,\,1,$$
and let $\mu_{\boldsymbol{\theta},s,u}:=f_{\boldsymbol{\theta},s,u}\,\di\lambda$, for $u=0,\,1$, be the corresponding probability measures on $\mathbb R$. Let $h_{\boldsymbol{\theta},s,u}$ be the density of 
$\mu_{\boldsymbol{\theta},s,u}\ast\mu_\varepsilon$, for $u=0,\,1$.
Using a standard randomization argument, 
it can be shown that there exists a constant $C_1>0$ such that 
\[
\sup_{\mu\in\mathcal D_1}\mathbb E^n_{(\mu\ast\mu_\varepsilon)}W_1(\hat\mu_n,\,\mu)\geq C_1 b_n^{-(\alpha+1)}\int_0^1|H^{(-1)}(u)|\,\di u\gtrsim n^{-(\alpha+1)/[2(\alpha+\beta)+1]},\]
provided that $[1-\chi^2(h_{\boldsymbol{\theta}, s,0}; \,h_{\boldsymbol{\theta}, s,1})/2]^{2n}$ is bounded below by a constant, where the $\chi^2$-divergence between two densities $h_0$ and $h_1$ on $\mathbb R$ is defined as
$$\chi^2(h_0;\,h_1)=\int_{\mathbb R}\frac{[h_0(x)-h_1(x)]^2}{h_0(x)}\,\di x.$$
Using standard arguments from \cite{fan1991}, see the proof of Theorem 5, pp. 1269--1271, under assumption \eqref{eq:condderiv}, it can be shown that there exists a constant $C_2>0$ so that 
$\chi^2(h_{\boldsymbol{\theta}, s,0}; \,h_{\boldsymbol{\theta}, s,1})\leq C_2b_n^{-[2(\alpha+\beta)+1]}\leq C_2n^{-1}$. 

\smallskip

We now show that also the sequence $n^{-1/2}$ is a lower bound on $\sup_{\mu\in\mathcal D_1}\mathbb E^n_{(\mu\ast\mu_\varepsilon)}W_1(\hat\mu_n,\,\mu)$.
In fact, replacing the function in \eqref{eq:targetf} with 
$$f_{\theta}(x):=f_{0,r}(x)+C a_n^{-1}\theta H(x),\quad x\in\mathbb R,\quad \theta\in\{0,1\},$$
and taking $a_n=([n^{1/2}]\vee 1)$, all previous steps go through and we find that, for a constant $C_3>0$,
\[
\sup_{\mu\in\mathcal D_1}\mathbb E^n_{(\mu\ast\mu_\varepsilon)}W_1(\hat\mu_n,\,\mu)\geq C_3 a_n^{-1}\gtrsim n^{-1/2}\]
and $\chi^2(f_0\ast\mu_\varepsilon; \,f_1\ast\mu_\varepsilon)\lesssim a_n^{-2}\lesssim n^{-1}$. 
Combining the two previously obtained bounds, we conclude that
\[
\sup_{\mu\in\mathcal D_1}\mathbb E^n_{(\mu\ast\mu_\varepsilon)}W_1(\hat\mu_n,\,\mu)\gtrsim 
\max\{n^{-(\alpha+1)/[2(\alpha+\beta)+1]},\,n^{-1/2}\}=n^{-(\alpha+1)/[2\alpha+(2\beta\vee 1)+1]}.\]

\vspace{0.5cm}  

\noindent 
$\bullet$ \emph{Case $d\geq2$}\\[4pt]
We define the $d\times d$ matrix
\[
\mathbf{A}=
\left(
\begin{array}{l@{}l}
 1 & \quad\mathbf{1}^t\\
 \mathbf{0} &\quad\mathbf{I}_{d-1}\\
\end{array}\right),
\]    
where $\mathbf{0}$ is a $(d-1)\times1$-column vector with all elements equal to $0$,  
while $\mathbf{1}^t$ is a $1\times(d-1)$-row vector with all elements equal to $1$ and
$\mathbf{I}_{d-1}$ is the identity matrix of size $(d-1)$.
The matrix $\mathbf{A}$ is invertible and, being upper triangular, the determinant is the product of the main diagonal entries, therefore $\mathrm{det}(\mathbf{A})=1$. 
For each observation $\mathsf{Y}_i$, we consider the transformation
\[\mathsf{Z}_i:=\mathbf{A}\mathsf{Y}_i=\mathbf{A}(\mathsf{X}_i+\boldsymbol{\varepsilon}_i)=\mathbf{A}\mathsf{X}_i+\mathbf{A}\boldsymbol{\varepsilon}_i,\quad i\in[n].\]
Set the position
\[\boldsymbol{\eta}_i:=\mathbf{A}\boldsymbol\varepsilon_i,\]
the $d\times 1$-column vector $\boldsymbol{\eta}_i$ has elements
\[
\eta_{i,j}=
\begin{dcases*}
\sum_{k=1}^d\varepsilon_{i,k},
   & if  $j=1$,\\
\varepsilon_{i,j},
   & if $j=2,\,\ldots,\,d$\,.
\end{dcases*}
\]
The random variables $\eta_{1,1},\,\ldots,\,\eta_{n,1}$ are i.i.d. according to the $d$-fold convolution measure $\mu_\varepsilon^{\ast d}$, that is,
\[\mu_{\eta_{i,1}}=\mu_\varepsilon^{\ast d},\quad i\in[n].\]
We now show that condition \eqref{eq:condderiv} implies that, for every $l=0,\,1,\,2$,
\begin{equation}\label{eq:derivbd}
|\hat\mu_{\eta_{i,1}}^{(l)}(t)|\lesssim (1+|t|)^{-(\beta d+l)},\quad t\in\mathbb R,\quad\mbox{for } i\in[n].
\end{equation}
In fact, by condition \eqref{eq:condderiv} with $l=0$,
$$|\hat\mu_{\eta_{i,1}}(t)|=|[\hat\mu_\varepsilon(t)]^{d}|\lesssim (1+|t|)^{-\beta d}, \quad t\in\mathbb R,\quad\mbox{for } i\in[n].$$
By the same condition, with $l=1$,
\[\begin{split}
|\hat\mu_{\eta_{i,1}}^{(1)}(t)|&
=d|\hat\mu_\varepsilon(t)|^{d-1}|\hat\mu_\varepsilon^{(1)}(t)|\lesssim (1+|t|)^{-(\beta d+1)},\quad t\in\mathbb R,\quad\mbox{for } i\in[n],
\end{split}
\]
and, with $l=2$,
\[\begin{split}
|\hat\mu_{\eta_{i,1}}^{(2)}(t)|&=|d(d-1)[\hat\mu_\varepsilon(t)]^{d-2}[\hat\mu_\varepsilon^{(1)}(t)]^2+
d[\hat\mu_\varepsilon(t)]^{d-1}\hat\mu_\varepsilon^{(2)}(t)|\\
&\lesssim (1+|t|)^{-\beta(d-2)}
(1+|t|)^{-2(\beta +1)}+(1+|t|)^{-\beta(d-1)}(1+|t|)^{-(\beta +2)}\\
&\lesssim (1+|t|)^{-(\beta d+2)},\quad t\in\mathbb R,\quad\mbox{for } i\in[n].
\end{split}
\]
This proves that condition \eqref{eq:derivbd} holds.

We make a preliminary remark for bounding below the supremum of the $L^1$-Wasserstein risk.
For a random vector $\mathsf X$ in $\mathbb R^d$ with distribution $\mu\in\mathcal P_1(\mathbb R^d,\,M)$, we denote by $\mu^{\mathbf A}$ the distribution of the transformation $\mathbf A\mathsf X$, which is the image measure of $\mu$ by ${\mathbf A}$. Let $\hat\mu_n$ be any estimator of $\mu$ based on the observations 
$\mathsf Y^{(n)}=(\mathsf{Y}_1,\,\ldots,\,\mathsf{Y}_n)$.
We denote by $\hat\mu_n^{\mathbf A}$ the corresponding estimator of $\mu^{\mathbf A}$,
which is a function of $\mathsf Z^{(n)}=(\mathsf{Z}_1,\,\ldots,\,\mathsf{Z}_n)$, with 
$\mathsf Z_i={\mathbf A}\mathsf Y_i$, for $i\in[n]$. Then,
\[\begin{split}
W_1(\hat\mu_n^{\mathbf A},\,\mu^{\mathbf A})&=\inf_{\tau\in\Gamma(\hat \mu_n^{\mathbf A},\,\mu^{\mathbf A})}
\int_{\mathbb R^d\times\mathbb R^d}|\mathsf z-\mathsf z'|\,\tau(\di\mathsf z,\,\di\mathsf z')\\
&=\inf_{\gamma\in\Gamma(\hat \mu_n,\,\mu)}\int_{\mathbb R^d\times\mathbb R^d}|{\mathbf A}\mathsf y-{\mathbf A}\mathsf y'|\,\gamma(\di\mathsf y,\,\di\mathsf y')\\
&= \inf_{\gamma\in\Gamma(\hat \mu_n,\,\mu)}\int_{\mathbb R^d\times\mathbb R^d}|\mathbf A(\mathsf y-\mathsf y')|\,\gamma(\di\mathsf y,\,\di\mathsf y')\leq |{\mathbf A}|W_1(\hat\mu_n,\,\mu),
\end{split}\]
where $|{\mathbf A}|:=(\sum_{i=1}^d\sum_{j=1}^d a_{ij}^2)^{1/2}=\sqrt{2d-1}$. Therefore, 
\begin{equation}\label{eq:reltransf}
|{\mathbf A}|W_1(\hat\mu_n,\,\mu)\geq W_1(\hat\mu_n^{\mathbf A},\,\mu^{\mathbf A}).
\end{equation}

\smallskip

We now study the supremum of the $L^1$-Wasserstein risk over the class $\mathcal D_d=\mathcal P_1(\mathbb R^d,\,M)\cap \mathcal S_d(\alpha,\,L)$. We begin by defining a finite family of Lebesgue absolutely continuous probability measures on $\mathbb R^d$, with uniformly bounded first moments, whose densities belong to $\mathcal S_d(\alpha,\,L)$.
Let $b_n:=([n^{1/[2(\alpha+\beta d)+1]}]\vee 1)$.
Let $f_{0,r}$ be the density defined in \eqref{eq:densf0} and $\mu_{0,r}=f_{0,r}\,\di\lambda$ the corresponding probability measure. 
For $\boldsymbol{\theta}\in\{0,\,1\}^{b_n}$, let $\mu_{\boldsymbol{\theta}}=f_{\boldsymbol{\theta}}\,\di\lambda$ be the probability measure corresponding to the density $f_{\boldsymbol{\theta}}$ defined in \eqref{eq:targetf}.
Define the product probability measure on $\mathbb R^d$
\[\bar\mu_{\boldsymbol{\theta}}^{\mathbf A}:=\mu_{\boldsymbol{\theta}}\otimes\mu_{0,r}^{\otimes (d-1)}=
(f_{\boldsymbol{\theta}}\,\di\lambda)\otimes (f_{0,r}\,\di\lambda)\otimes\ldots\otimes(f_{0,r}\,\di\lambda)\]
having Lebesgue density $\bar f_{\boldsymbol{\theta}}^{\mathbf A}(\tilde {\mathsf x})=f_{\boldsymbol{\theta}}(\tilde x_1)\times\prod_{j=2}^d f_{0,r}(\tilde x_j)$, $\tilde {\mathsf x} \in \mathbb R^d$. Define $\bar\mu_{\boldsymbol{\theta}}$ to be the distribution of $\mathsf X:= \mathbf A^{-1} \tilde {\mathsf X}$ when $\tilde {\mathsf X} \sim \bar\mu_{\boldsymbol{\theta}}^{\mathbf A}$. In other words, $\bar\mu_{\boldsymbol{\theta}}$ has density 
$\bar f_{\boldsymbol{\theta}}(\mathsf x) = \bar f_{\boldsymbol{\theta}}^{\mathbf A}(\mathbf A \mathsf x)$. 
We show that 
$$\{\bar\mu_{\boldsymbol{\theta}}:\,{\boldsymbol{\theta}}\in\{0,\,1\}^{b_n}\}\subseteq\mathcal D_d.$$

\smallskip
\noindent\emph{The probability measure $\bar \mu_{\boldsymbol{\theta}}\in\mathcal P_1(\mathbb R^d,\,M)$}\\[5pt]
Taking into account that the Euclidean norm of any vector is bounded above by its $1$-norm, that is, $|\mathsf x|\leq \sum_{j=1}^d|x_j|$,
for $n$ large enough we have
\[
\begin{split}
M_1(\bar\mu_{\boldsymbol{\theta}}) &\leq |\mathbf A^{-1}|  M_1(\bar\mu_{\boldsymbol{\theta}}^{\mathbf A}) \leq  |\mathbf A^{-1}| \left[  
M_1(\mu_{\boldsymbol{\theta}})+\sum_{j=2}^dM_1(\mu_{0,r})\right] =:\bar M_{0,r,\delta}<\infty
\end{split}
\]
by the same arguments laid down for the case $d=1$.
Thus, $\bar \mu_{\boldsymbol{\theta}}\in\mathcal P_1(\mathbb R^d,\,M)$ for every $M\geq \bar M_{0,r,\delta}$.\\

\smallskip

\noindent\emph{The density $\bar f_{\boldsymbol{\theta}}\in\mathcal S_d(\alpha,\,L)$}\\[5pt]
First note that 
$\hat {\bar f}_{\boldsymbol{\theta}}(\mathsf t)  = \hat {\bar f}_{\boldsymbol{\theta}}^{\mathbf A}(\mathsf t \cdot \mathbf A^{-1})$.
By the same arguments exposed for the case $d=1$, for $n$ large enough we have
\[
\begin{split}
\sum_{j=1}^d
\int_{\mathbb R^d}|\hat {\bar f}_{\boldsymbol{\theta}}(\mathsf t)|^2(1+t_j^2)^{\alpha}\,\di \mathsf t
=&
\sum_{j=1}^d
\int_{\mathbb R^d}|\hat {\bar f}_{\boldsymbol{\theta}}^{\mathbf A}(\mathsf t \cdot\mathbf A^{-1})|^2(1+t_j^2)^{\alpha}\,\di \mathsf t\\ 
&\leq d 
\int_{\mathbb R^d}|\hat {\bar f}_{\boldsymbol{\theta}}^{\mathbf A}(\mathsf t\cdot\mathbf A^{-1})|^2(1+|\mathsf  t|^2)^{\alpha}\,\di \mathsf t \\ 
&< d |\mathbf A|^{2\alpha}
\int_{\mathbb R^d}|\hat {\bar f}_{\boldsymbol{\theta}}^{\mathbf A}(\mathsf t)|^2(1+|\mathsf  t|^2)^{\alpha}\,\di \mathsf t\\
&=d |\mathbf A|^{2\alpha}
\bigg(\int_{\mathbb R}|\hat f_{\boldsymbol{\theta}}(t_1)|^2(1+t_1^2)^{\alpha}\,\di t_1 \\
&\qquad\qquad\qquad\quad+
\sum_{j=2}^d\int_{\mathbb R}|\hat f_{0,r}(t_j)|^2(1+t_j^2)^{\alpha}\,\di t_j\bigg)=:L_0.
\end{split}
\]
Therefore, there exists a finite constant $\bar L_0>0$ such that $\sum_{j=1}^d
\int_{\mathbb R^d}|\hat {\bar f}_{\boldsymbol{\theta}}(\mathsf t)|^2(1+t_j^2)^{\alpha}\,\di \mathsf t \leq \bar L_0$.
Thus, $\bar f_{\boldsymbol{\theta}}\in\mathcal S_d(\alpha,\,L)$ for every $L\geq\bar L_0$.

\smallskip

Let $\tilde{\boldsymbol{\theta}}$ be a random vector whose components are i.i.d. Bernoulli random variables $\tilde\theta_1,\,\ldots,\,\tilde\theta_{b_n}$, with $P(\tilde\theta_s=1)=P(\tilde\theta_s=0)=\frac{1}{2}$, for $s\in[b_n]$. By the inequality in \eqref{eq:reltransf}, 
for any estimator $\hathat\mu_n$ that is a measurable function of the observations $\mathsf{Z}^{(n)}$ 
from $(\mathbb R^d)^n$ into the set of probability measures on $\mathbb R$, we have
\[\begin{split}
|{\mathbf A}|\sup_{\mu\in\mathcal D_d}\mathbb E^n_{(\mu\ast\mu_\varepsilon^{\otimes d})}W_1(\hat\mu_n,\,\mu)
&\geq |{\mathbf A}|
\sup_{\boldsymbol{\theta}\in\{0,\,1\}^{b_n}}
\mathbb E^n_{(\bar\mu_{\boldsymbol{\theta}}\ast\mu_\varepsilon^{\otimes d})}W_1(\hat\mu_n,\,\bar\mu_{\boldsymbol{\theta}})\\
&\geq
\sup_{\boldsymbol{\theta}\in\{0,\,1\}^{b_n}}
\mathbb E^n_{(\bar\mu_{\boldsymbol{\theta}}\ast\mu_\varepsilon^{\otimes d})}
W_1(\hat\mu_n^{\mathbf A},\,\bar\mu_{\boldsymbol{\theta}}^{\mathbf A})\\
&\geq
\sup_{\boldsymbol{\theta}\in\{0,\,1\}^{b_n}}
\mathbb E^n_{(\bar\mu_{\boldsymbol{\theta}}\ast\mu_\varepsilon^{\otimes d})}
W_1((\hat\mu_n^{\mathbf A})_1,\,(\bar\mu_{\boldsymbol{\theta}}^{\mathbf A})_1)\\
&\geq \inf_{\hathat\mu_n}\sup_{\boldsymbol{\theta}\in\{0,\,1\}^{b_n}}\mathbb E^n_{(\bar\mu_{\boldsymbol{\theta}}\ast\mu_\varepsilon^{\otimes d})}W_1(\hathat\mu_n,\,(\bar\mu_{\boldsymbol{\theta}}^{\mathbf A})_1)\\
&\geq \inf_{\hathat\mu_n}\mathbb E\mathbb 
E^n_{(\bar\mu_{\tilde{\boldsymbol{\theta}}}\ast\mu_\varepsilon^{\otimes d})}W_1(\hathat\mu_n,\,(\bar\mu_{\tilde{\boldsymbol{\theta}}}^{\mathbf A})_1)\\
&= \inf_{\hathat\mu_n}\int_{\mathbb R}\mathbb E\mathbb 
E^n_{(\bar\mu_{\tilde{\boldsymbol{\theta}}}\ast\mu_\varepsilon^{\otimes d})}|\hathat F_n(x)-F_{(\bar\mu_{\tilde{\boldsymbol{\theta}}}^{\mathbf A})_1}(x)|\,\di x,
\end{split}\]
where $\hathat F_n$ is the distribution function of $\hathat \mu_n$ and 
$F_{(\bar\mu_{\tilde{\boldsymbol{\theta}}}^{\mathbf A})_1}$ is the distribution function of 
$(\bar\mu_{\tilde{\boldsymbol{\theta}}}^{\mathbf A})_1$, which is the marginal distribution of $\bar\mu_{\tilde{\boldsymbol{\theta}}}^{\mathbf A}$ on the first coordinate, that is, $\mu_{\tilde{\boldsymbol{\theta}}}$, whose density is 
\[f_{\tilde{\boldsymbol{\theta}}}=f_{0,r}+
C b_n^{-\alpha}\sum_{s=1}^{b_n}\tilde \theta_s H(b_n(\cdot-x_{s,n})).\]
For $\boldsymbol{\theta}\in\{0,\,1\}^{b_n}$ and $s\in[b_n]$, we define the densities
$$\bar f_{\boldsymbol{\theta},s,u}:=\bar f_{(\theta_1,\ldots,\theta_{s-1},u,\theta_{s+1},\ldots,\theta_{b_n})},\quad \mbox{for }u=0,\,1,$$
and let $\bar\mu_{\boldsymbol{\theta},s,u}:=\bar f_{\boldsymbol{\theta},s,u}\,\di\lambda$, for $u=0,\,1$, be the corresponding probability measures on $\mathbb R$. 
For any $x\in[x_{s,n},\,x_{s+1,n}]$, taking the expected value with respect to $\tilde\theta_s$ and using the subscript 
$\tilde{\boldsymbol{\theta}}\mysetminus s:=(\tilde{\theta}_1,\,\ldots,\,\tilde{\theta}_{s-1},\,\tilde{\theta}_{s+1},\,\ldots,\,\tilde{\theta}_{b_n})$ to denote the expected value with respect to the remaining 
components of the vector $\tilde{\boldsymbol{\theta}}$, we have
\[\begin{split}
&\mathbb E\mathbb 
E^n_{(\bar\mu_{\tilde{\boldsymbol{\theta}}}\ast\mu_\varepsilon^{\otimes d})}|\hathat F_n(x)-F_{(\bar\mu_{\tilde{\boldsymbol{\theta}}}^{\mathbf A})_1}(x)|\\
&\quad=
\frac{1}{2}
\mathbb E_{\tilde{\boldsymbol{\theta}}\mysetminus s}
\bigg[\sum_{u=0}^1\mathbb 
E^n_{(\bar\mu_{\tilde{\boldsymbol{\theta}},s,u}\ast\mu_\varepsilon^{\otimes d})}|\hathat F_n(x)-
F_{(\bar\mu_{\tilde{\boldsymbol{\theta}},s,u}^{\mathbf A})_1}(x)|
\bigg]\\
&\quad\geq
\frac{1}{2}
\mathbb E_{\tilde{\boldsymbol{\theta}}\mysetminus s}
\int_{\mathbb R^d}\ldots\int_{\mathbb R^d}
|F_{(\bar\mu_{\tilde{\boldsymbol{\theta}},s,0}^{\mathbf A})_1}(x)-
F_{(\bar\mu_{\tilde{\boldsymbol{\theta}},s,1}^{\mathbf A})_1}(x)|\\
&\qquad\quad
\times \min\bigg\{\prod_{i=1}^n(\bar f_{\tilde{\boldsymbol{\theta}},s,0}\ast f_\varepsilon^{\otimes d})^{\mathbf A}(\mathsf z_i),\,\prod_{i=1}^n(\bar f_{\tilde{\boldsymbol{\theta}},s,1}\ast f_\varepsilon^{\otimes d})^{\mathbf A}(\mathsf z_i)\bigg\}\,\di \mathsf{z}_1\ldots\di \mathsf{z}_n\\
&\quad=
\frac{1}{2}b_n^{-(\alpha+1)}\big|H^{(-1)}(b_n(x-x_{s,n}))\big|\\
&\qquad\quad\times 
\mathbb E_{\tilde{\boldsymbol{\theta}}\mysetminus s}
\int_{\mathbb R^n}
\min\bigg\{\prod_{i=1}^n\big( f_{\tilde{\boldsymbol{\theta}},s,0}\ast f_{\eta_{1,1}}\big)(z_{i,1}),\,\prod_{i=1}^n\big(f_{\tilde{\boldsymbol{\theta}},s,1}\ast f_{\eta_{1,1}}\big)(z_{i,1})\bigg\}\,\di {z_{1,1}}\ldots\di {z_{n,1}}\\
&\quad\geq
\frac{1}{4}b_n^{-(\alpha+1)}\big|H^{(-1)}(b_n(x-x_{s,n}))\big|\\
&\qquad\quad\times 
\mathbb E_{\tilde{\boldsymbol{\theta}}\mysetminus s}\bigg[1-\frac{1}{2}
\chi^2\Big(f_{\tilde{\boldsymbol{\theta}},s,0}\ast f_{\eta_{1,1}};\,f_{\tilde{\boldsymbol{\theta}},s,1}\ast f_{\eta_{1,1}}\Big)\bigg]^{2n}
\end{split}
\]
because the following facts hold: 
\begin{itemize}
\item for any $\boldsymbol{\theta}\in\{0,\,1\}^{b_n}$,
\[
\begin{split}
|F_{(\bar\mu_{{\boldsymbol{\theta}},s,0}^{\mathbf A})_1}(x)-
F_{(\bar\mu_{{\boldsymbol{\theta}},s,1}^{\mathbf A})_1}(x)|&=
Cb_n^{-\alpha}\bigg|\int_{-\infty}^x (f_{\boldsymbol{\theta},s,0}-
f_{\boldsymbol{\theta},s,1})(u)\,\di u
\bigg|\\
&=Cb_n^{-(\alpha+1)}\abs{\int_{-\infty}^x
H(b_n(u-x_{s,n}))\,\di u}\\
&=C
b_n^{-(\alpha+1)}\big|H^{(-1)}(b_n(x-x_{s,n}))\big|;
\end{split}\]
\item for any $\boldsymbol{\theta}\in\{0,\,1\}^{b_n}$ and $u=0,\,1$, all observations ${z_{1,1}},\,\ldots,\,{z_{n,1}}$ are i.i.d. according to the probability measure
\[((\bar\mu_{{\boldsymbol{\theta}},s,u}\ast\mu_\varepsilon^{\otimes d})^{\mathbf A})_1=
((\bar\mu_{{\boldsymbol{\theta}},s,u})^{\mathbf A})_1\ast((\mu_\varepsilon^{\otimes d})^{\mathbf A})_1=
\mu_{\boldsymbol{\theta},s,u}\ast\mu_{\eta_{1,1}};
\]
\item by the same arguments as in Theorem 3 of \cite{dedecker:michel:2013}, p. 283, for any $\boldsymbol{\theta}\in\{0,\,1\}^{b_n}$,
\[\begin{split}
&\int_{\mathbb R^n}
\min\bigg\{\prod_{i=1}^n\big( f_{\boldsymbol{\theta},s,0}\ast f_{\eta_{1,1}}\big)(z_{i,1}),\,\prod_{i=1}^n\big(f_{\boldsymbol{\theta},s,1}\ast f_{\eta_{1,1}}\big)(z_{i,1})\bigg\}\,\di {z_{1,1}}\ldots\di {z_{n,1}}\\
&\qquad\qquad\qquad\qquad\qquad\qquad\qquad\qquad\,\geq
\frac{1}{2} \bigg[1-\frac{1}{2}
\chi^2\Big(f_{\boldsymbol{\theta},s,0}\ast f_{\eta_{1,1}}, f_{\boldsymbol{\theta},s,1}\ast f_{\eta_{1,1}}\Big)\bigg]^{2n}.
\end{split}\]
\end{itemize}
By applying the same arguments as in \cite{fan1991}, p. 1270, 
with the difference that the error density is ordinary smooth of order $\beta d$ instead of $\beta$ and condition \eqref{eq:derivbd} holds, we get that there exists a constant $c>0$ such that, for any $\boldsymbol{\theta}\in\{0,\,1\}^{b_n}$, we have 
\[\chi^2\Big(f_{\boldsymbol{\theta},s,0}\ast f_{\eta_{1,1}}, f_{\boldsymbol{\theta},s,1}\ast f_{\eta_{1,1}}\Big)\leq cb_n^{-[2(\alpha+\beta d)+1]}\lesssim n^{-1}.\]
Thus, for a suitable constant $C'>0$,
\[\begin{split}
|{\mathbf A}|\sup_{\mu\in\mathcal D_d}\mathbb E^n_{(\mu\ast\mu_\varepsilon^{\otimes d})}W_1(\hat\mu_n,\,\mu)
&\geq \inf_{\hathat\mu_n}\int_{\mathbb R}\mathbb E\mathbb 
E^n_{(\bar\mu_{\tilde{\boldsymbol{\theta}}}\ast\mu_\varepsilon^{\otimes d})}|\hathat F_n(x)-F_{(\bar\mu_{\tilde{\boldsymbol{\theta}}}^{\mathbf A})_1}(x)|\,\di x\\
&\geq C'b_n^{-(\alpha+1)}
\sum_{s=1}^{b_n}\int_{x_{s,n}}^{x_{s+1,n}}
\big|H^{(-1)}(b_n(x-x_{s,n}))\big|\,\di x\\
&= C'b_n^{-(\alpha+1)}
\int_0^1
\big|H^{(-1)}(u)\big|\,\di u\\
&\gtrsim n^{-(\alpha+1)/[2(\alpha+\beta d)+1]}.
\end{split}\]

To show that also the sequence $n^{-1/2}$ is a lower bound on $\sup_{\mu\in\mathcal D_d}\mathbb E^n_{(\mu\ast\mu_\varepsilon^{\otimes d})}W_1(\hat\mu_n,\,\mu)$ we can reason as in the case $d=1$.
Therefore, combining the two previously obtained bounds, we have that
\[
\sup_{\mu\in\mathcal D_d}\mathbb E^n_{(\mu\ast\mu_\varepsilon^{\otimes d})}W_1(\hat\mu_n,\,\mu)\gtrsim 
\max\{n^{-(\alpha+1)/[2(\alpha+\beta d)+1]},\,n^{-1/2}\}=n^{-(\alpha+1)/[2\alpha+(2\beta d\vee 1)+1]}\]
and the proof is complete.
\end{proof}


\section{Proof of Theorem \ref{th:comtelacour} and related results} \label{sec:prthcomte} 
\subsection{Proof of Theorem \ref{th:comtelacour}}
We first note that 
$$ \sup_{\mathsf v \in \mathbb S^{d-1}} \| F_{\tilde \mu_{1 n, \mathsf v}} -  F_{ \mu_{0X, \mathsf v}} \|_1 \leq \sup_{\mathsf v \in \mathbb S^{d-1} } \| F_{\tilde \mu_{n, \mathsf v}} -  F_{\mu_{0X , \mathsf v}} \|_1 + O(n^{-1/2}).$$ 
Then, by inequality \eqref{eq:equivalence} and Theorem \ref{theo:1} with $\beta>1$ and any sequence $h_n\rightarrow 0$,
\begin{equation*}
\begin{split}
W_1(\tilde \mu_{1n},\,\mu_{0X}) &\leq C_d \overline W_1(\tilde \mu_{1n},\,\mu_{0X})\\
&=C_d\sup_{\mathsf v \in \mathbb S^{d-1}} \| F_{\tilde \mu_{1 n, \mathsf v}} -  F_{ \mu_{0X, \mathsf v}} \|_1\\
&\lesssim \sup_{\mathsf v \in \mathbb S^{d-1} } \| F_{\tilde \mu_{ n, \mathsf v}} -  F_{\mu_{0X, \mathsf v}} \|_1 +n^{-1/2}\\
&\lesssim  h_n +   \sup_{\mathsf v \in \mathbb S^{d-1} } \| F_{\tilde \mu_{Yn, \mathsf v}} -  F_{\mu_{0Y, \mathsf v}} \|_1 + n^{-1/2} \\
& \qquad\,\,  +  (\log n)\sup_{\mathsf v \in \mathbb S^{d-1} } \bigg(h_n^{-\beta |I^*_{h_n} (\mathsf v)|  + 1}\prod_{j \in I^*_{h_n} (\mathsf v) } |v_j|^\beta \|f_{\tilde \mu_{Yn, \mathsf v}} -  f_{\mu_{0Y , \mathsf v}}\|_1\bigg).
\end{split}
\end{equation*}
We now bound $\|f_{\tilde \mu_{Yn, \mathsf v}} -  f_{\mu_{0Y , \mathsf v}}\|_1$ uniformly in $\mathsf v$. In \cite{comtelacourihp}, the authors control the errors in the $L^2$-distance, therefore we need to control the above term differently.
For every $\mathsf v\in\mathbb R^d$, with $\mathbb G_n:=\sqrt{n}(\mathbb P_n-P_{0Y})$ the empirical process, we have
\begin{eqnarray*} \label{comte:decomp1}
\| f_{\tilde \mu_{Yn, \mathsf v}} -  f_{\mu_{0Y, \mathsf v}} \|_1 &\leq& \frac{1 }{2\pi} \bigg[ 
 \int_{\mathbb R} \bigg(\abs{\int_{\mathbb R}e^{-\imath t x}\hat\mu_{0Y}(t\mathsf v) [1 - \hat K_{b_n}^{\otimes d}(t \mathsf v)] \,\di t} \nonumber\\
&&\qquad\qquad\qquad\qquad\qquad\quad + \abs{\int_{\mathbb R}e^{-\imath t x} \hat K_{b_n}^{\otimes d}(t \mathsf v) \frac{  \mathbb G_n( e^{\imath t \mathsf v \cdot \mathsf Y} )}{\sqrt{n}}  \,\di t}\bigg)\,\di x\bigg].
 \end{eqnarray*}
We denote by $B_1(\mathsf v)$ the first integral and by $B_2(\mathsf v)$ the second one.
In Section \ref{sec:prboundB1B2}, we show that 
\begin{equation}\label{bound B1B2}
B_1(\mathsf v)\lesssim  b_n^{2 | I^*_{b_n} (\mathsf v)|} \prod_{j \in I^*_{b_n} (\mathsf v)} v_j^{-2} \quad \mbox{ and }\quad  \mathbb E_{0Y}^{n}\bigg[\sup_{\mathsf v \in \mathbb S^{d-1}} B_2(\mathsf v)\bigg] \lesssim \frac{ (\log n)^{3/2} }{ \sqrt{n b_n}}.
\end{equation}
Since $\beta=2$, choosing $h_n = b_n= [n/(\log n )^3]^{-1/(4d+1)}$, the bounds in \eqref{bound B1B2} imply that
\begin{equation*}
\begin{split}
\sup_{\mathsf v \in\mathbb S^{d-1}}
\bigg(b_n^{-2|I^*_{b_n} (\mathsf v)|+1}\prod_{j \in I^*_{b_n} (\mathsf v)}v_j^2 \|f_{\tilde \mu_{Yn,\mathsf v}}-f_{\mu_{0Y,\mathsf v}}\|_1\bigg) &\lesssim b_n + \frac{ b_n^{-2d+1/2}(\log n)^{3/2}}{\sqrt{n}} \\
&\lesssim [n/(\log n )^3]^{-1/(4d+1)}. 
\end{split}
\end{equation*}
The bound on $\sup_{\mathsf v \in\mathbb S^{d-1}} \| F_{\tilde \mu_{Yn, \mathsf v}} -  F_{\mu_{0Y, \mathsf v}}\|_1$ proceeds similarly noting that 
\begin{equation*}
\begin{split}
F_{\tilde \mu_{Yn, \mathsf v}}(y) & = \frac{ 1 }{ 2\pi}\int_{-\infty}^y \int_{\mathbb R} e^{-\imath t x } \hat K_{b_n}^{\otimes d}(t \mathsf v) \phi_n(t \mathsf v)\,\di t\,\di x \\
& =  F_{\mu_{0Y, \mathsf v}}(y)+ \frac{ 1 }{ 2\pi \sqrt{n}} \int_{-\infty}^y\int_{\mathbb R} e^{-\imath t x } \hat K_{b_n}^{\otimes d}(t \mathsf v) \mathbb G_n(e^{\imath t\mathsf v \cdot\mathsf Y})\,\di t\,\di x\\
 & \qquad\qquad\quad+
 \frac{ 1 }{ 2\pi} \int_{\mathbb R} e^{-\imath t y }\frac{[\hat K_{b_n}^{\otimes d}(t \mathsf v)-1]}{ -\imath t}  \hat \mu_{0Y}(t\mathsf v)\,\di t,
 \end{split}
\end{equation*}
so that 
\begin{equation*}
\begin{split}
\int_{\mathbb R}| F_{\tilde \mu_{Yn, \mathsf v}} (y)- F_{\mu_{0Y , \mathsf v}}(y)|\,\di y & \leq    \frac{ 1 }{ 2\pi} \int_{\mathbb R}\bigg|\int_{\mathbb R} e^{-\imath t y }\frac{[\hat K_{b_n}^{\otimes d}(t \mathsf v) -1]}{ -\imath t}  \hat \mu_{0Y}(t\mathsf v)\,\di t \bigg|\,\di y  \\
 & \qquad +\frac{ 1 }{ 2\pi \sqrt{n}}\int_{\mathbb R}\bigg|  \int_{-\infty}^y \int_{\mathbb R} e^{-\imath t x } \hat K_{b_n}^{\otimes d}(t \mathsf v) \mathbb G_n(e^{\imath t\mathsf v \cdot\mathsf Y})\,\di t \,\di x\bigg|\,\di y \\
 &=: \bar B_1(\mathsf v) +  \bar B_2(\mathsf v).
 \end{split}
\end{equation*}
The first term $\bar B_1(\mathsf v)$  is similar to $B_1(\mathsf v)$, therefore $$\bar B_1(\mathsf v) \lesssim b_n$$
uniformly in $\mathsf v \in\mathbb S^{d-1}$. 
We now study the term $\bar B_2(\mathsf v)$ following the control of $B_2$ in Section \ref{sec:prboundB1B2} below. First note that
$\bar B_2(\mathsf v) =(2\pi \sqrt{n})^{-1} \int_{\mathbb R}|\mathbb G_n(\int_{-\infty}^y g_{x,\mathsf v}(\mathsf Y) \,\di x)|\,\di y$,
where the function $g_{x,\mathsf v}(\mathsf Y)$ is defined in \eqref{B2}. Set $G_{y, \mathsf v}(\mathsf Y):= \int_{-\infty}^y g_{x,\mathsf v}(\mathsf Y)\,\di x$, $y\in\mathbb R$, write
$$ \bar B_2(\mathsf v) = \bigg(\int_{-\infty}^0+\int_0^{\infty}\bigg)\big|\mathbb G_n\big(G_{y, \mathsf v}(\mathsf Y)\big)\big|\,\di y.$$
We only study the case where $y\leq0$ because the case $y>0$ can be treated similarly. Without loss of generality, we assume that $v_j>0$ for all $j \in  J^*_d(\mathsf v)$.  
Since 
$$g_{x,\mathsf v}(\mathsf Y)=(2\pi)\Conv_{j\in  J^*_d(\mathsf v)} K_{b_nv_j}(x-\mathsf v\cdot\mathsf Y),\quad x\in\mathbb R,$$

\begin{equation*}
\begin{split}
|G_{y,\mathsf v}(\mathsf Y) | \lesssim  \int_{-\infty}^y \Conv_{j\in  J^*_d(\mathsf v)} |K_{ b_n v_j} (x-\mathsf v \cdot \mathsf Y)|\,\di x
&=\int_{-\infty}^{y-\mathsf v \cdot \mathsf Y} \Conv_{j\in  J^*_d(\mathsf v)}|K_{ b_n v_j} (u)|\,\di u\\
    &\lesssim \mathsf P\left( \sum_{j \in J^*_d(\mathsf v)}  b_n v_j Z_j \leq y-\mathsf v \cdot \mathsf Y\right)\\
    &  \leq \sum_{j \in J^*_d(\mathsf v)} \mathsf P\left(  v_j  Z_j \leq (y +\|\mathsf Y\|_1)/( db_n)\right) ,
\end{split}
\end{equation*}
where the $Z_j$'s are i.i.d. random variables with density  $ |K|/\|K\|_1$.
Hence, for all $ k>1 $, defined $\mathsf v_{\min}:=\min_{j\in[d]}|v_j|$, we have
\begin{equation*}
\begin{split}
G_{y,\mathsf v}(\mathsf Y) &\leq \1_{(y\leq (-2 \|\mathsf Y\|_1\wedge -1))}  d\mathsf P\left(  Z_1 \leq y /(2db_n\mathsf v_{\min}) \right) + d\1_{(0\geq y>  (-2 \|\mathsf Y\|_1\wedge -1))} \\
 & \lesssim  \1_{(y\leq  -1)}(b_n/|y|)^{k-1}+ \1_{(-1< y\leq0)},
\end{split}
\end{equation*}
which in turns implies that $$\mathbb E_{0Y}(G_{y,\mathsf v}(\mathsf Y)^2 ) \leq \1_{(y\leq-1)} (b_n/|y|)^{2(k-1)} + \1_{(-1< y \leq 0)}.$$
Using Lemma 19.36 of \cite{vaart_1998}, p. 268, jointly with the computations of the integrated bracketing entropy in Section \ref{sec:prboundB1B2},
$$ \mathbb E_{0Y}^n\left( \sup_{\mathsf v \in\mathbb S^{d-1}} \bar B_2(\mathsf v) \right) \lesssim \sqrt{\frac{\log n}{n}}.$$ 
This implies that 
$$  \sup_{\mathsf v \in\mathbb S^{d-1}} \| F_{\tilde \mu_{Yn,\mathsf v}} -  F_{\mu_{0Y, \mathsf v}}\|_1 \lesssim b_n,$$
which concludes the proof. 
\subsection{Proof of the bounds in \eqref{bound B1B2}}\label{sec:prboundB1B2}

We begin to prove the bound on $B_1(\mathsf v)$. 
\vspace{0.2cm}  

\noindent 
$\bullet$ \emph{Bound on $B_1(\mathsf v)$}\\[4pt]
For $\mathsf v\in\mathbb R^d$, let 
$a_{\mathsf v}(x):=(2\pi)^{-1}\int_{\mathbb R}e^{-\imath t x}\hat\mu_{0Y}(t\mathsf v) [1 - \hat K_{b_n}^{\otimes d}(t \mathsf v)]\,\di t$, $x\in\mathbb R$. Using the inequality
$\|a_{\mathsf v}\|_1^2 \leq \|\hat a_{\mathsf v}\|_2\times \|\hat a_{\mathsf v}^{(1)}\|_2$, see, {\em e.g.}, (4.4) in \cite{Bobkov:2016}, p. 1030, we have
\begin{equation*}
\begin{split}
[B_1(\mathsf v)]^2 &\leq \|\hat\mu_{0Y}( \cdot \mathsf v) [1 - \hat K_{b_n}^{\otimes d}(\cdot \mathsf v)]\|_2\times  \left\|\frac{ \di}{\di t}\left(\hat\mu_{0Y}(\cdot \mathsf v) [1 - \hat K_{b_n}^{\otimes d}(\cdot\mathsf v) ]\right)\right\|_2.
\end{split}
\end{equation*}
Note that
\begin{equation*}
\begin{split} 
|1 - \hat K_{b_n}^{\otimes d}(t \mathsf v)| &\leq \sum_{j\in J_d^{*}(\mathsf v)} |1 - \hat K(b_n v_j t)|\leq d \1_{\left(|t|\geq (b_n \|\mathsf v\|_\infty\right)^{-1})}
 \end{split}
 \end{equation*}
 and 
\begin{equation*}
\begin{split}
\left|\frac{\di}{\di t}\left(\hat\mu_{0Y}(t \mathsf v) [1 - \hat K_{b_n}^{\otimes d}(t\mathsf v)]\right)\right| & \lesssim\frac{ \1_{(|t|\geq (b_n \|\mathsf v\|_\infty)^{-1})}}{  \prod_{j=1}^d (1 + v_j^2t^2)}  \\
&\qquad\quad\,\,\,\,\,\times
  \left[ \left|\frac{\di}{\di t} \hat\mu_{0X}(t\mathsf v) \right|  + 2 |\hat\mu_{0X}(t\mathsf v)|\sum_{j \in J_d^{*}(\mathsf v)}\frac{v_j^2|t|}{1+v_j^2 t^2 }\right]\\
  & \lesssim \frac{ \1_{(|t|\geq (b_n \|\mathsf v\|_\infty)^{-1})} }{\prod_{j=1}^d [1 + v_j^2/(b_n\|\mathsf v\|_\infty )^2]}  
  \left( \left|\frac{\di}{\di t}\hat\mu_{0X}(t\mathsf v)\right|  +  |\hat\mu_{0X}(t\mathsf v)| \right).
\end{split}
 \end{equation*}
By assumption \eqref{eq:exptails}, recalling inequality \eqref{eq:momentsineq}, we have
 $$\frac{1}{2\pi}\int_{\mathbb R} \left|\frac{\di}{\di t} \hat\mu_{0X}(t\mathsf v)  \right|^2\,\di t=\int_{\mathbb R}x^2|\mu_{0X,\mathsf v}(x)|^2\,\di x\leq \int_{\mathbb R^d}
 |\mathsf x|^2 f_{0X}(\mathsf x)\,\di \mathsf x=M_2(\mu_{0X})< \infty.$$ 
Combining previous bounds and recalling that $I_{b_n}^\ast(\mathsf v)=\{j\in[d]:\,|v_j|>b_n\}$, we obtain that
\begin{equation*}
\begin{split}
B_1(\mathsf v) &\lesssim  \frac{[M_2(\mu_{0X})]^{1/2}+\|\hat\mu_{0X}( \cdot \mathsf v)\|_2}{  \prod_{j=1}^d [1 + v_j^2/(b_n\|\mathsf v\|_\infty )^2]} \lesssim b_n^{2 |I_{b_n}^\ast(\mathsf v) |}\prod_{j \in I^\ast_{b_n}(\mathsf v) } v_j^{-2}.
\end{split}
\end{equation*}

\vspace{0.2cm}  

\noindent 
$\bullet$ \emph{Bound on $B_2(\mathsf v)$}\\[4pt]
Set
\begin{equation}\label{B2}
g_{x, \mathsf v}(\mathsf Y):= \int_{\mathbb R}e^{-\imath t (x-\mathsf v \cdot \mathsf Y)} \hat K_{b_n}^{\otimes d}(t \mathsf v) \,\di t,\quad x\in\mathbb R,
\end{equation}
we have 
\[B_2(\mathsf v)= \frac{1}{2\pi\sqrt{n} } \int_{\mathbb R} |\mathbb G_n ( g_{x, \mathsf v}(\mathsf Y)) |\,\di x.\]
We now control $\mathbb E_{0Y}^{n} [ \sup_{\mathsf v \in \mathbb S^{d-1}} |\mathbb G_n ( g_{x, \mathsf v}(\mathsf Y))|]$.
Let $\mathcal G_n(x):=\{ g_{x, \mathsf v}(\mathsf Y):\, \mathsf v \in \mathbb S^{d-1}\}$. We use Lemma 19.36 of \cite{vaart_1998}, p. 288. 
Since $\|\mathsf v\|_\infty \geq 1/d$ for all $\mathsf v \in \mathbb S^{d-1}$, we have $|g_{x, \mathsf v}(\mathsf Y)|<2d/b_n$. We now bound $\mathbb E_{0Y}[g_{x, \mathsf v}(\mathsf Y)^2 ]$. We have 
\begin{equation*}
\begin{split}
\mathbb E_{0Y}[g_{x, \mathsf v}(\mathsf Y)^2]& <\frac{ 2d }{ b_n }  \int_{\mathbb R}  \Conv_{j\in J_d^*(\mathsf v)}|K_{b_n v_j}(x-y)|f_{0Y,\mathsf v}(y)\,\di y\\
& \leq \frac{ 2 d}{ b_n} \|f_{0Y, \mathsf v}\|_\infty\|K\|_1^{|J_d^*(\mathsf v)|}\leq \frac{ 2 d}{ b_n}\|K\|_1\sup_{\mathsf v \in \mathbb S^{d-1}}  \|f_{0Y, \mathsf v}\|_\infty=:\delta_n^2.
\end{split}
\end{equation*}
For $|\mathsf v_1 - \mathsf v_2|\leq \tau b_n^2\epsilon$, with $\tau \in (0,\,1)$, we have
\begin{equation*}
\begin{split}
 |g_{x,\mathsf v_1}(\mathsf Y) - g_{x,\mathsf v_2}(\mathsf Y)| & \leq \tau b_n^2 \epsilon \int_{\mathbb R} |t| |\mathsf Y| \1_{|t| \leq d/b_n}\,\di t \leq \tau \epsilon |\mathsf Y|d^2 .
\end{split}
\end{equation*}
Using that $\mathbb E_{0Y}[ |\mathsf Y|^2] <\infty$, 
a $\delta_n$-bracket covering of $\mathcal G_n$ is obtained by an $(\epsilon\tau b_n^2)$-covering of $\mathbb S^{d-1}$ choosing $\tau $ accordingly. Hence 
\begin{equation*}
J_{[\,]}(\delta_n, \mathcal G_n(x) , L^2(P_{0Y})) \lesssim  \int_0^{\delta_n} \sqrt{ \log ( 1/b_n) + \log (1/\epsilon)_+} \,\di \epsilon \lesssim \sqrt{ \log n} \delta_n \lesssim \sqrt{ \frac{\log n}{b_n}},
\end{equation*}
which in turns implies  that 
\begin{equation*}\label{stochastic:comte}
\frac{1}{\sqrt{n}} \int_{|x| \leq R_n} \mathbb E_{0Y}^{n} \left[\sup_{\mathsf v \in\mathbb S^{d-1}}|\mathbb G_n ( g_{x, \mathsf v}(\mathsf Y)) | \right]\, \di x\lesssim \frac{ R_n \sqrt{\log n}  }{ \sqrt{nb_n}}.
\end{equation*}
We now study $\int_{|x| > R_n} |\mathbb G_n ( g_{x, \mathsf v}(\mathsf Y)) |\,\di x$. 
Consider the event $\Omega_n = \{|\mathsf Y_i| \leq R_n/2,\,i\in[n] \}$. Then, $|x- \mathsf v\cdot \mathsf Y_i| \geq |x|/2$ when $|x|\geq R_n$, and, for all $k\geq 1$, 
\begin{equation*}
g_{x, \mathsf v}(\mathsf Y_i) = \int_{\mathbb R} e^{ -\imath t  ( x-\mathsf v \cdot \mathsf Y_i) } \hat K_{b_n}^{\otimes d}(t \mathsf v)\,\di t = \frac{ 1 }{2\pi [\imath  ( x-\mathsf v \cdot \mathsf Y_i) ]^k } \int_{\mathbb R} e^{ -\imath t ( x-\mathsf v \cdot \mathsf Y_i) } \frac{\di^k}{\di t^k } \hat K_{b_n}^{\otimes d}(t \mathsf v)\,\di t
\end{equation*}
so that, since $\hat K$ is $k$-times continuously differentiable and each derivative is equal to 0 on the boundary of its support, 
$$ |g_{x, \mathsf v}(\mathsf Y_i)| \lesssim \frac{  1 }{ b_n |x|^k }.$$ 
Also on $\Omega_n$,
$$\mathbb G_n ( g_{x, \mathsf v}(\mathsf Y)) = \mathbb G_n ( g_{x, \mathsf v}(\mathsf Y)\1_{(|\mathsf Y|\leq |x|/2)} )  - \sqrt{n} \mathbb E_{0Y}[g_{x, \mathsf v}(\mathsf Y)\1_{(|\mathsf Y|>|x|/2)}]$$
and, for $|x|> R_n$, with the abuse of notation $c_2 := (c_2 \wedge 1)$, 
$$ \sqrt{n} \mathbb E_{0Y}[|g_{x, \mathsf v}(\mathsf Y)|\1_{(|\mathsf Y|>|x|/2)}] \lesssim \frac{ \sqrt{n}}{b_n} [P_{0X} (|\mathsf X| >|x|/4) + e^{-|x|/4}]\lesssim \frac{ \sqrt{n}}{ b_n} e^{-c_2 |x|/4}.$$
Therefore on $\Omega_n$, 
\begin{equation*}\label{boundgx}
\begin{split}
\sup_{\mathsf v \in\mathbb S^{d-1}} \int_{|x|>R_n} |\mathbb G_n ( g_{x, \mathsf v}(\mathsf Y)) |\,\di x &\lesssim  \sup_{\mathsf v\in\mathbb S^{d-1}} \int_{|x|>R_n} |\mathbb G_n ( g_{x, \mathsf v}(\mathsf Y)\1_{|\mathsf Y|\leq R_n/2})|\,\di x+\frac{ \sqrt{n} }{ b_n}   e^{-c_2R_n/4}.
 \end{split}
\end{equation*} 
Using the above construction of a covering of $\mathcal G_n(x)$ with the upper bounds, for $|x|>R_n$, 
$$ \|g_{x, \mathsf v}(\mathsf Y) \1_{|\mathsf Y|\leq R_n/2}\|_\infty \lesssim \frac{ 1 }{ b_n |x|^k  } \quad \mbox{ and }\quad \|g_{x, \mathsf v}(\mathsf Y) \1_{|\mathsf Y|\leq R_n/2}\|_2\lesssim \frac{ 1 }{ b_n |x|^k  },$$
we obtain 
\begin{equation*}
J_{[\,]}(2d/b_n, \mathcal G_n(x) , L^2(P_{0Y})) \lesssim  \int_0^{2/(b_n|x|^k) } \sqrt{ \log ( 1/b_n) + \log (1/\epsilon)_+}\, \di \epsilon \lesssim \frac{ \sqrt{ \log n}   }{b_n |x|^k },
\end{equation*}
so that 
\begin{equation*}
\begin{split}
\int_{|x|>R_n}\mathbb E_{0Y}^{n}\left[  \sup_{\mathsf v \in \mathbb S^{d-1}} |\mathbb G_n ( g_{x, \mathsf v}(\mathsf Y) \1_{|\mathsf Y|\leq R_n/2})| \right]\,\di x \lesssim \frac{ \sqrt{ \log n}}{\sqrt{n}b_nR_n^{k-1}},
 \end{split}
\end{equation*} 
which implies that
\begin{equation*}
 \frac{ 1 }{ \sqrt{n} }  \int_{|x|>R_n}\mathbb E_{0Y}^{n}\left[ \1_{\Omega_n} \sup_{\mathsf v \in \mathbb S^{d-1}} |\mathbb G_n ( g_{x, \mathsf v}(\mathsf Y)) | \right]\,\di x \lesssim \frac{ \sqrt{ \log n}   }{\sqrt{n}b_nR_n^{k-1}}  +\frac{ \sqrt{n} }{ b_n}   e^{-c_2R_n/4}.
 \end{equation*}
We also bound 
\begin{equation*}
\begin{split}
 \frac{ 1 }{ \sqrt{n} } &  \int_{|x|>R_n}\mathbb E_{0Y}^{n}\left[ \1_{\Omega_n^c} \sup_{\mathsf v \in \mathbb S^{d-1}} |\mathbb G_n ( g_{x, \mathsf v}(\mathsf Y)) | \right]\, \di x \\
&\qquad \qquad\leq 2\mathbb E_{0Y}^{n} \left[\1_{\Omega_n^c} \sup_{\mathsf v \in\mathbb S^{d-1}} \int_{|x|>R_n} \left| \int_{\mathbb R} e^{-\imath t (x-\mathsf v \cdot \mathsf Y)}\hat K_{b_n}^{\otimes d}(t \mathsf v) \,\di t\right|\,\di x \right].
 \end{split}
 \end{equation*} 
By symmetry of $K$, that is, $K(x) = K(-x)$, if $v_j\neq 0$, we have
$$\int e^{-\imath t  b_n v_j  x} \hat K(t)\,\di t =\frac{1}{(b_n |v_j|)} K( x/(b_n |v_j|))=: K_{b_n v_j}(x).$$
Therefore,
$$\left| \int_{\mathbb R} e^{-\imath t (x-\mathsf v \cdot \mathsf Y)}\hat K_{b_n}^{\otimes d}(t \mathsf v) \,\di t\right|  =  \left| \Conv_{j \in J^*_d(\mathsf v)} K_{b_n v_j}( x-\mathsf v \cdot \mathsf Y)\right| \leq  \Conv_{j \in J_d^*(\mathsf v)} |K_{b_n v_j}( x-\mathsf v \cdot \mathsf Y)|.$$ 
We thus obtain
\begin{equation*}
\begin{split}
 \frac{ 1 }{ \sqrt{n} } &  \int_{|x|>R_n}\mathbb E_{0Y}^{n}\left[ \1_{\Omega_n^c} \sup_{\mathsf v \in \mathbb S^{d-1}} |\mathbb G_n ( g_{x, \mathsf v}(\mathsf Y) )| \right]\, \di x \\
 & \qquad\qquad\qquad\qquad\qquad \leq 2 \mathbb E_{0Y}^{n} \left[\1_{\Omega_n^c} \sup_{\mathsf v \in \mathbb S^{d-1}} \big\|\Conv_{j=1}^d K_{b_n v_j}\big \|_1 \right] \leq 2 P_{0Y}^n(\Omega_n^c) \|K\|_1^d.
 \end{split} 
 \end{equation*} 
Since $P_{0Y}^n(\Omega_n^c) \leq e^{-R_n/4} +P_{0X} (|\mathsf X| >|x|/4)\leq e^{- c_2R_n/4}$ (assuming without loss of generality that $c_2\leq 1$), by choosing $R_n = R_0 \log n$, with $R_0$ large enough, we get that
\begin{equation*}
\begin{split}
\sup_{\mathsf v \in\mathbb S^{d-1}}B_2(\mathsf v) & \lesssim \frac{ R_n \sqrt{\log n } }{ \sqrt{n b_n} } + \frac{ 2\sqrt{n} }{ b_n}   e^{-c_2R_n/4} \lesssim \frac{  (\log n )^{3/2} }{ \sqrt{nb_n} }.
\end{split}
\end{equation*}
This concludes the proof. 
\qed

\bibliographystyle{imsart-number}
\bibliography{biblio}


\end{document}